\documentclass[10pt]{amsart}
\usepackage[font=small,labelfont=bf]{caption}
\usepackage[margin=1.25in]{geometry}
\newtheorem{theorem}{Theorem}[section]
\newtheorem{lemma}[theorem]{Lemma}
\newtheorem{proposition}{Proposition}[section]

\theoremstyle{definition}
\newtheorem{definition}[theorem]{Definition}

\usepackage{color}
\usepackage{verbatim}
\usepackage[colorinlistoftodos]{todonotes}

\theoremstyle{remark}
\newtheorem{remark}[theorem]{Remark}

\numberwithin{equation}{section}

\usepackage{graphicx, subcaption}
\usepackage{graphicx}
\usepackage{amsmath,amsfonts,amssymb,amsthm}
\usepackage{mathtools}
\usepackage{commath}
\usepackage{colortbl}
\usepackage{comment}

\begin{document}

\title{Convergence of an Eulerian scheme for the Vlasov-Poisson-BGK model}

\author[S. Y. Cho]{Seung Yeon Cho}
\address{Seung Yeon Cho\\
	Department of Mathematics\\ 
	Gyeongsang National University\\
	Jinju 52828, Republic of Korea}
\email{chosy89@gnu.ac.kr}

\author[S. Park]{Sungsu Park}
\address{Sung su Park\\
Department of Mathematics\\ 
Sungkyunkwan University\\
Suwon 440-746, Republic of Korea}
\email{ppssu@skku.edu}

\author[S.-B. Yun]{Seok-Bae Yun}
\address{Seok-Bae Yun\\
Department of Mathematics\\ 
Sungkyunkwan University\\
 Suwon 440-746, Republic of Korea}
\email{sbyun01@skku.edu}

\begin{abstract}   
    The Vlasov-Poisson-BGK (VPBGK) model is a kinetic model for describing the dynamics of collisional plasmas. Although various numerical schemes have been developed for it, a corresponding convergence theory has been absent. This paper fills this gap by presenting the first convergence analysis for a non-splitting, finite-difference Eulerian scheme discretized on the full phase-space grid. A major theoretical obstacle is the mixing of velocity indices induced by the electric field, which hinders the derivation of a uniform lower bound for the discrete solution. To overcome this stability challenge, we propose a modified lower bound estimate suitable for ionized systems that incorporates the step-wise degradation. Under a truncated velocity domain with a Neumann boundary condition, we establish error estimates for the distribution function in a weighted $L^{\infty}$ norm and for the electric field in a $L^{\infty}$ norm, respectively.
\end{abstract}
\maketitle
\section{Introduction}\label{sec1}
\subsection{The Vlasov-Poisson-BGK model}
The dynamics of a colliding ionized rarefied gas system is described by the Vlasov-Poisson-Boltzmann system in kinetic theory. This system couples the Vlasov-Poisson system, which governs the evolution of charged particle distribution under an electric field, with the Boltzmann collision operator, which accounts for particle interactions through elastic collisions:
\begin{align*}
    \begin{split}
        &\partial_t f+ v\cdot \nabla_x{f} +\mathbb{E}\cdot \nabla_vf = \frac{1}{\varepsilon}Q(f,f),\cr
        &-\triangle_x \phi = \int_{\mathbb{R}}fdv - 1,\quad\mathbb{E}=-\nabla_x \phi.
    \end{split}
\end{align*}
However, despite its physical accuracy, the complexity of the Boltzmann collision operator $Q(f,f)$, which comes from involving multi-dimensional integrals, causes challenges in numerical computation. To address this issue, Bhatnagar, Gross, and Krook \cite{PhysRev.94.511} proposed a simplified model that replaces the Boltzmann collision operator with a relaxation operator driving the distribution function toward a local Maxwellian equilibrium. This BGK model has been widely used as an approximation of the Boltzmann system, since the model preserves fundamental physical properties of the Boltzmann system while reducing computational cost dramatically.\\
\indent In this paper, we consider the Cauchy problem for the one-dimensional Vlasov-Poisson-BGK (VPBGK) model with a uniform background ion density:
\begin{align}\label{VPBGK model}
\begin{split}
&\partial_t f+ v \partial_x{f} +\mathbb{E} \partial_vf = \frac{1}{\varepsilon}\left(\mathcal{M}(f)-f\right),\cr
&-\partial^2_x \phi = \int_{\mathbb{R}}fdv - 1, \quad\mathbb{E}=-\partial_x \phi,\cr
&f(0) = f_0,\quad \int_{\mathbb{T}}\int_{\mathbb{R}} f_0dxdv = 1.
\end{split}
\end{align}
The velocity distribution function $f(x,v,t)$ is the number density of electrons at a phase point $(x,v)\in \mathbb{T}\times\mathbb{R}$ and time $t\geq 0$, where $\mathbb{T}=[0,1]$ with a periodic boundary condition and $\mathbb{R}$ is the whole real line. The functions $\phi(x,t)$ and $\mathbb{E}(x,t)$ are the electric potential and field, respectively. The multiscale parameter $\varepsilon>0$ denotes a Knudsen number defined as the ratio between the mean free path and the typical macroscopic length. The local Maxwellian $\mathcal{M}(f)$ reads
\begin{align}\label{local Maxwellian}
    \begin{split}
        \mathcal{M}(f)(x,v,t)=\frac{\rho(x,t)}{\sqrt{2 \pi T(x,t) }}\exp\left(-\frac{|v-U(x,t)|^2}{2T(x,t)}\right),
    \end{split}
\end{align}
where the macroscopic local density $\rho$, bulk velocity $U$, energy $E$ are defined as follows:
\begin{align*}
\begin{split}
\rho(x,t)& = \int_{\mathbb{R}}f(x,v,t)dv,\cr
\rho(x,t) U(x,t) &= \int_{\mathbb{R}}vf(x,v,t)dv,\cr
E(x,t) &= \int_{\mathbb{R}}\frac{|v|^2}{2} f(x,v,t)dv.
\end{split}
\end{align*}
The temperature $T$ is given by
\begin{align*}
    \begin{split}
        \rho(x,t) T(x,t) &= \int_{\mathbb{R}}|v-U(x,t)|^2 f(x,v,t)dv.
    \end{split}
\end{align*}
Moreover, the relaxation operator satisfies the following cancellation property since $\mathcal{M}(f)$ shares the first three moments with $f$:
\begin{align*}
    \begin{split}
        \int_{\mathbb{R}} (\mathcal{M}(f) -f ) (1, v, |v|^2) dv = 0.
    \end{split}
\end{align*}
This further leads to the conservation of mass, momentum, and total energy as follows:
\begin{align*}
    \begin{split}
        \frac{d}{dt}\int_{\mathbb{T}\times\mathbb{R}}fdxdv=0,\quad\frac{d}{dt}\int_{\mathbb{T}\times\mathbb{R}}vfdxdv=0,\quad \frac{d}{dt}\left(\int_{\mathbb{T}\times\mathbb{R}}|v|^2fdxdv + \int_{\mathbb{T}}|\mathbb{E}|^2dx \right)=0.
    \end{split}
\end{align*}
In addition, this model satisfies an entropy dissipation, which is called the $H$-theorem:
\begin{align*}
    \begin{split}
        \frac{d}{dt}\int_{\mathbb{R}}f\log f dv=\int_{\mathbb{R}} (\mathcal{M}(f) -f )\ln fdv \leq 0.
    \end{split}
\end{align*}
\subsection{Main goal of this paper}
Various numerical methods for the VPBGK model have already been studied. A micro-macro approaches based on a decomposition of the solution into an equilibrium and a perturbation were applied to the VPBGK model in \cite{crestetto2012kinetic,crouseilles2011asymptotic,laidin2022hybrid}. In \cite{crouseilles2016multiscale}, the authors constructed an asymptotic preserving (AP) scheme with the rescaled Poisson equation using two multiscale parameters and provided a stability estimate for the first-order scheme. On the other hand, the automatic domain decomposition scheme was proposed in \cite{dimarco2014asymptotic} to treat multiscale issues in the computational domain. Although the aforementioned methods show good numerical performance in multiscale problems, to the authors' knowledge, no rigorous convergence analysis has ever been established for any numerical scheme applied to the VPBGK model. This leaves a gap between numerical practice and theoretical understanding.\\
\indent The goal of this paper is to fill this gap by developing a first convergence theory for a standard scheme of the VPBGK model. More precisely, we analyze a first-order finite-difference Eulerian scheme:
\begin{align*}
    \begin{split}       &f^{n+1}_{i,j}=\frac{\varepsilon\tilde{f}^n_{i,j}+\Delta t\mathcal{M}(\tilde{f}^n_{i,j})}{\varepsilon+\Delta t},\quad\mathbb{E}^n_i = \sum_{k}K(x_i,y_k)(\rho^n_k-1)\Delta y.
    \end{split}
\end{align*}
The rigorous detail of each term for the scheme is described in Section \ref{sec2}. The main idea in the construction of the scheme lies in the IMEX(implicit-explicit) method. In our scheme, we treat the convection and force terms explicitly while handling the relaxation term implicitly. This enables the scheme to overcome the stiffness issue of the small Knudsen number, which is the core of the construction of an AP scheme \cite{Cho2024,hu2018asymptotic,hu2017class,pieraccini2007implicit}. Indeed, despite the implicit treatment of the relaxation term, the cancellation property of the Maxwellian allows for an explicit computation. Since existing AP schemes share a similar implicit structure, our method can be viewed as a standard formulation among various existing AP schemes. In this work, we derive error estimates for the distribution function in a weighted $L^{\infty}$ norm and for the electric field in a $L^{\infty}$ norm, respectively:
\begin{align*}
        \begin{split}
            \|f(T^{f}) - f^{N_t}\|_{L^{\infty}_{q}} + \|\mathbb{E}(T^{f}) - \mathbb{E}^{N_t}\|_{L^{\infty}_{x}} \leq C\left\{\Delta x + \Delta v +\Delta t +\frac{1}{V^{q-3}_M} \right\}.
        \end{split}
    \end{align*}
\subsection{New idea and key difficulties}
The VPBGK model simultaneously considers the relaxation operator and the electric field. The interaction between them leads to new difficulties in our analysis. One major challenge is estimating a uniform lower bound for the discrete solution, which is complicated by the mixing of velocity indices. Another problem arises from the restrictions imposed by the CFL condition. In this subsection, we introduce these specific issues and how we treat them.\\\\  
\noindent $\bullet$ \textbf{Uniform stability estimate for the discrete solution $\boldsymbol{f^n_{i,j}}$}: 
The most important step in our analysis is establishing stability estimates for the discrete solution $f^n_{i,j}$. Assuming the initial lower bound condition $f_0\geq C_0e^{-|v|^\alpha}$ $(\alpha\in[1,2])$, we obtain the following uniform-in-$n$ stability estimates (see Section \ref{sec4} for detailed notation):
\begin{align*}
    \begin{split}
        &C(q,T^f,\alpha,\varepsilon)e^{-|v_j|^\alpha}e^{-2\alpha C_{\mathbb{E},1}T^f|v_j|}\leq f^n_{i,j}\leq e^{\left(C_{\mathbb{E},2}+\frac{C_{\mathcal{M}}-1}{\varepsilon} \right)T^f}\|f_0\|_{L^{\infty}_q}(1+|v_j|)^{-q}.
    \end{split}
\end{align*}
\indent Unlike the Vlasov-Poisson system, a convergence analysis for the BGK-type models necessitates a uniform positive lower bound on the discrete solution. This requirement guarantees a positive lower bound for the temperature $T^n_i$, ensuring that the local Maxwellian in the relaxation term remains well-defined (see Lemma \ref{Mac_bound_lemma}):
\begin{align*}
    \begin{split}
        T^n_i \geq \left( \frac{\sum_j f^n_{i,j}\Delta v}{C_M\|f^n\|_{L^{\infty}_q}}\right)^{2}.
    \end{split}
\end{align*}
For non-ionized systems \cite{boscarino2022convergence,russo2012convergence,russo2018convergence}, this bound was established via a straightforward recurrence relation:
\begin{align}\label{previous lb}
    \begin{split}
        f^n_{i,j}\geq \left(\frac{\varepsilon}{\varepsilon+\Delta t} \right)^n\sum_m w^n_{s+m,j}f^0_{s+m,j}\geq  C_0\left(\frac{\varepsilon}{\varepsilon+\Delta t} \right)^n e^{-|v_j|^\alpha},
    \end{split}
\end{align} 
where $\{w^n_{s+m,j} \}$ are binomial weights summing to one. This leads to a uniform-in-$n$ bound using $(1+x)^{-n}\geq e^{-nx}$ and $n\Delta t\leq T^f$:
\begin{align*}
    \begin{split}
        f^n_{i,j}\geq C_0e^{-\frac{T^f}{\varepsilon}}e^{-|v_j|^\alpha}.
    \end{split}
\end{align*}
The analysis in \eqref{previous lb} relies on a key simplification: every initial datum $f^0_{s+m,j}$ contributing to the lower bound of $\tilde{f}^n_{i,j}$ shares the same common factor $e^{-|v_j|^\alpha}$, which is independent of the summation index $m$.\\
\indent However, for the VPBGK model, the presence of the electric field $\mathbb{E}^n_i$ accelerates or decelerates the velocity of particle $v_j$, causing a mixing of velocity indices during the update process. This necessitates a more delicate estimate for the lower bound of the discrete solution. More precisely, if we apply the methodology of previous works \cite{boscarino2022convergence,russo2012convergence,russo2018convergence} to our scheme, the following estimate is derived involving multinomial weights $\{w^n_{i+m,j+l}\}_{m,l}$ that depend on the entire sequence of the preceding electric fields $\{\mathbb{E}^k_{i+m}\}$ for $m\in[-n,n]$ and $k\in[0,n]$:
\begin{align}\label{bad iteration}
    \begin{split}
        f^{n}_{i,j}\geq \left(\frac{\varepsilon}{\varepsilon+\Delta t} \right)^n\sum_{m,l}w^{n}_{i+m,j+l}f^0_{i+ m,j+ l}.
    \end{split}
\end{align}
Crucially, the lower bound from each term $f^0_{i+m,j+l}$ becomes $C_0e^{-|v_{j+l}|^\alpha}$, which depends on the summation index $l$ (see figure \ref{update figure}):
\begin{align*}
    \begin{split}
         f^{n}_{i,j}&\geq C_0\left(\frac{\varepsilon}{\varepsilon+\Delta t} \right)^ne^{-|v_{j}|^\alpha},\quad \hspace{2.2cm}\text{(non-ionized case)},\cr f^{n}_{i,j}&\geq C_0\left(\frac{\varepsilon}{\varepsilon+\Delta t} \right)^n\sum_{m,l}w^{n}_{i+m,j+l}e^{-|v_{j+ l}|^\alpha},\quad \text{(ionized case)}.
    \end{split}
\end{align*}
\begin{figure}
    \centering
    \includegraphics[width=0.75\linewidth]{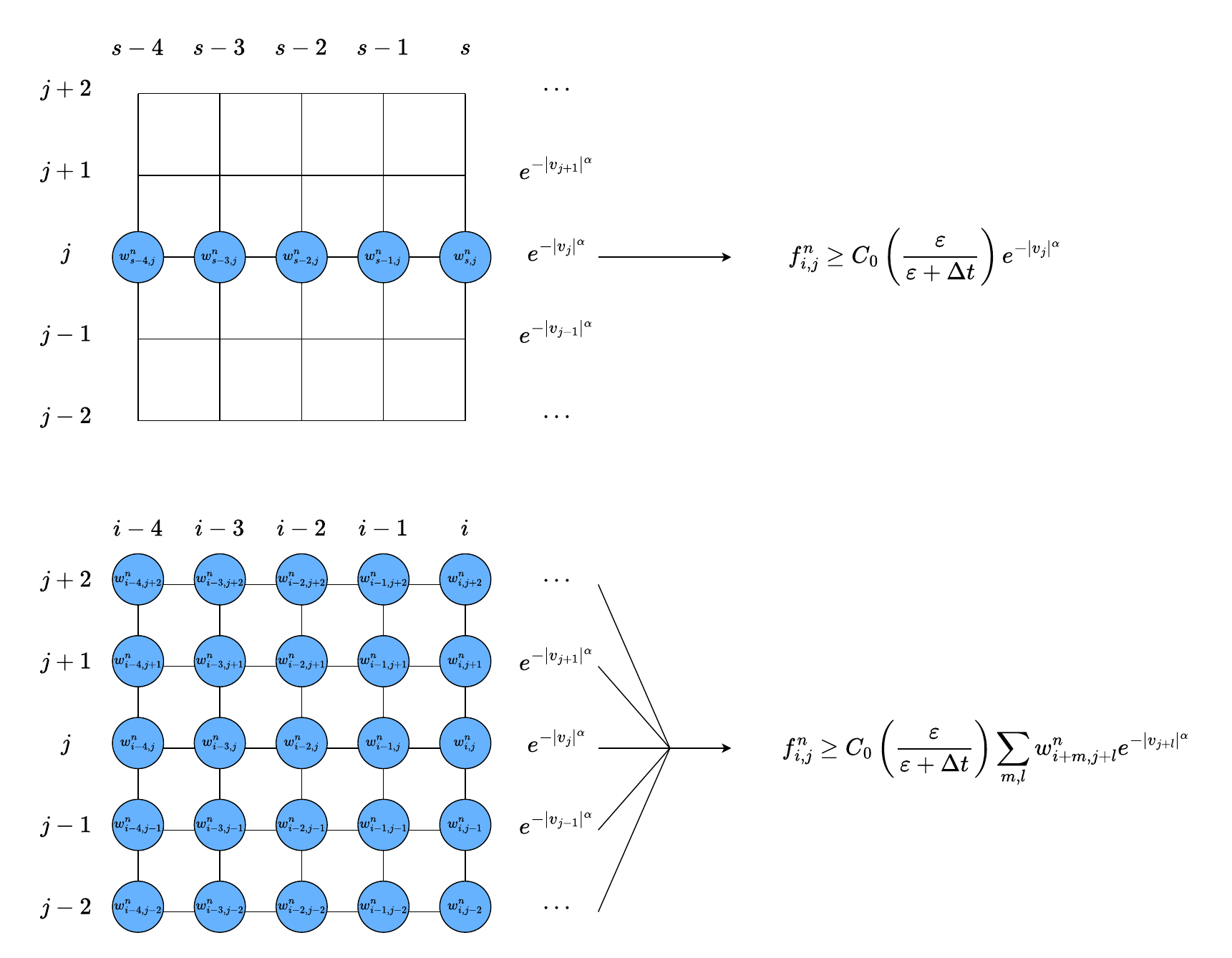}
    \caption{The blue circles on the grid represent the phase points at time $t=0$, which contribute to the lower bound estimate of $f^n_{i,j}$ for $v_j>0$. 
    The top figure is the case of the non-ionized \cite{boscarino2022convergence,russo2012convergence,russo2018convergence}, where $s$ is the spatial nodes such that $x_i-v_j\Delta t$ lies in $[x_s,x_{s+1})$. The bottom figure is the case of the VPBGK model.}
    \label{update figure}
\end{figure}
This mixing of velocity indices prevents factorization used in the non-ionized case, thereby obstructing the establishment of a uniform lower bound.\\
\indent To overcome this difficulty, we adopt a strategy that focuses on controlling the rate at which the lower bound decays at each time step due to the electric field. We split our analysis into two cases: whether the particle velocity $v_j$ is accelerated or decelerated by $\mathbb{E}^n_i$.\\
\textbf{(1) Deceleration Case:} We consider the case where the particle is decelerated by the electric field (e.g., $v_j>0$ and $\mathbb{E}^n_i<0$). In this scenario, the mixing of velocity indices degrades the lower bound. Applying the inductive argument \eqref{previous lb}, we deduce
\begin{align*}
    \begin{split}        
    {f}^n_{i,j}&\geq C_0\left(\frac{\varepsilon}{\varepsilon+\Delta t} \right)^n\left(e^{-|v_j|^\alpha}-\frac{\Delta t}{\Delta v}|\mathbb{E}^{n-1}_i|\left(e^{-|v_j|^\alpha}-e^{-|v_{j+1}|^\alpha} \right) \right).
    \end{split}
\end{align*}
Since $e^{-|v_j|^\alpha}-e^{-|v_{j+1}|^\alpha}>0$, the term involving the electric field is negative, making the lower bound strictly less than $e^{-|v_j|^\alpha}$ and breaking the simple recurrence. This failure suggests a decay in the lower bound over time. We therefore introduce a new inductive hypothesis specifically designed for the ionized case to capture this step-wise degradation:
\begin{align}\label{new lb}
    \begin{split}
        f^{n}_{i,j}\geq C_0\left\{\prod^{n}_{k=1}(1-\alpha C_{\mathbb{E},1}\Delta t(|v_j|+k\Delta v + 1))\right\}\left(\frac{\varepsilon}{\varepsilon+\Delta t}\right)^ne^{-|v_j|^\alpha}.
    \end{split}
\end{align}
Under this new inductive hypothesis, we can close the induction using the mean value theorem to bound the difference $e^{-|v_j|^\alpha}-e^{-|v_{j+1}|^\alpha}$ as follows:
\begin{align*}
    \begin{split}
        e^{-|v_j|^\alpha}-e^{-|v_{j+1}|^\alpha}\leq \alpha \Delta v(|v_j|+\Delta v +1)e^{-|v_j|^\alpha}.
    \end{split}
\end{align*}
Moreover, with appropriate smallness condition on the mesh sizes $\Delta x$, $\Delta v$, and $\Delta t$, we yield the uniform-in-$n$ estimate for $f^n_{i,j}$ (see Lemma \ref{A3}):
\begin{align*}
    \begin{split}
        f^n_{i,j}\geq C_0e^{- \left(2\alpha C_{\mathbb{E},1}+\frac{\alpha C^2_{\mathbb{E},1}(T^f+1)}{\eta_1}+\frac{1}{\varepsilon }\right)T^f}e^{-|v_j|^\alpha}e^{-2\alpha C_{\mathbb{E},1}T^f|v_j|}.
    \end{split}
\end{align*}
\textbf{(2) Acceleration Case:} When the electric field $\mathbb{E}^n_i$ accelerates particle velocity $v_j$ (e.g., $v_j >0$ and $\mathbb{E}^n_i>0$), although the electric field induces a mixing of particle velocity, the lower bound can still be expressed as a function of $v_j$. This allows us to apply the argument used for the non-ionized case \cite{boscarino2022convergence,russo2012convergence,russo2018convergence}. Utilizing the standard inductive hypothesis \eqref{previous lb}, we can close the induction and obtain an analogous result.\\\\
\noindent $\bullet$ \textbf{Constraints imposed by the CFL condition in the presence of the electric field:} Another challenge arises from the CFL condition, which must be satisfied at each time step:
\begin{align*}
    \begin{split}
        \sup_{i,j}\left(\frac{\Delta t}{\Delta x}|v_j|+\frac{\Delta t}{\Delta v}|\mathbb{E}^n_i| \right)<1.
    \end{split}
\end{align*}
Ensuring that this condition holds uniformly for all $n$ introduces two difficulties that must be addressed.\\
\indent First, since the CFL condition involves $\sup|v_j|$, we must truncate the velocity domain at a maximum velocity $V_M$, defined as
\begin{align}\label{truncation VM}
    \begin{split}
        V_M=\frac{2C_{\mathbb{E},1}}{{\Delta v}^{\gamma}},\quad \frac{1}{2}\leq \gamma<1.
    \end{split}
\end{align}
This truncation introduces an additional error that must be accounted for in the overall analysis. In particular, when estimating the difference between the continuous and discrete solutions, we have to consider the estimate for the continuous solution outside the truncated domain. We mention that such a choice of $V_M$, which is inversely proportional to $\Delta v$, is a common approach for initial data that is defined on the entire velocity space with a decay condition \cite{besse2004convergence,filbet2001convergence,schaeffer1998convergence}.\\
\indent Second, the CFL condition depends on the time-varying field $\mathbb{E}^n_i$. To construct $\Delta t$ which satisfies the CFL condition globally for all $n$, we establish a uniform-in-$n$ upper bound for $\mathbb{E}^n_i$. We achieve this by estimating the uniform $L^1$ norm of the discrete solution $f^n_{i,j}$ via the conservative form of the scheme (see Lemma \ref{A1}):
\begin{align*}
    \begin{split}
        \sum_{i,j}f^n_{i,j}\Delta x\Delta v\leq (2+T^f)e^{\frac{T^f}{\varepsilon}}. 
    \end{split}
\end{align*}
This bound, combined with the properties of the Green kernel, allows us to derive  the uniform upper bound for $\mathbb{E}^n_i$ (see Lemma \ref{A2}):
\begin{align*}
    \begin{split}    
        \|\mathbb{E}^n\|_{L^{\infty}_x}\leq (2+T^f)e^{\frac{T^f}{\varepsilon}}+1.
    \end{split}
\end{align*}

\subsection{Literature review}
Although this paper is the first to study the convergence analysis for the VPBGK model, numerous convergence theories have been constructed for the Vlasov-Poisson system and the BGK-type models. We close this introduction with a brief review of these related works.\\\\ 
\noindent$\bullet$ \textbf{Vlasov-Poisson system:}
The convergence analysis of the Vlasov-Poisson system was initiated for the particle method \cite{cottet1984particle,ganguly1989convergence,victory1991convergence1,victory1991convergence2}. In \cite{besse2004convergence}, an error estimate of the semi-Lagrangian scheme was established. The extension of the high-order scheme was considered in \cite{besse2008convergence1,besse2008convergence2}. In \cite{besse2008convergence1}, the author provided a convergence for a specific scheme based on gradient propagation and Hermite interpolation. This result was generalized in \cite{besse2008convergence2}. The authors showed that the scheme with appropriate interpolations, such as symmetric Lagrangian and B-spline, that satisfy certain stability and accuracy conditions, will converge. We also refer to convergence studies for the finite-volume method \cite{filbet2001convergence} and for the discontinuous Galerkin approximation \cite{einkemmer2014convergence}. While the challenge of deriving a positive lower bound is important to our analysis, it is not a necessary consideration in the study of the Vlasov-Poisson system, as the model itself does not impose such a requirement.\\\\
\noindent$\bullet$ \textbf{BGK type model:}
In the context of the BGK type model, the error estimate of a particle method for the BGK model was provided in \cite{issautier1996convergence} with regular initial data. The convergence theory of the semi-Lagrangian scheme for the monatomic BGK model in an $L^1$ space was studied in \cite{russo2012convergence}. After that, in \cite{russo2018convergence}, the result was extended to the ES-BGK model in a weighted $L^{\infty}$ space, and generalized to the polyatomic ES-BGK model \cite{boscarino2022convergence}, which are the main motivations of the present work. Although the models in these studies also necessitate a positive lower bound, the challenge of deriving this bound in the presence of an electric field is a feature unique to our work. Furthermore, in \cite{boscarino2022convergence,russo2018convergence}, the authors analyze the convergence of a semi-Lagrangian scheme, which overcomes the CFL restriction. As a result, in contrast to this paper, the truncation of the velocity domain need not be considered in their analysis.  

\subsection{Notations}
At the end of this section, we introduce the notation that will be frequently used in this paper:
\begin{itemize}
    \item $C_{a,b,\cdots,}$ represents generic constants which depends on $a$, $b$, $\cdots$. It can have different values for each line. 
    \item We use lower and upper indices $n$, $i$, $j$, respectively, for time, space, and velocity. For example, $t^n$, $x_{i}$, $v_{j}$.
    \item In this paper, we truncate the velocity grid to $[-V_M,V_M]$ where $V_M$ is the maximum velocity in the grid.
    \item Unless otherwise specified, the interval of the summation $\underset{i}{\sum}$ and $\underset{j}{\sum}$ be
    \begin{align*}
        \begin{split}
            0\leq i\leq N_x-1,\quad -N_v\leq j\le N_v.
        \end{split}
    \end{align*}
    \item We use the following notation for weighted $L^{\infty}$ norm for continuous solution:
        \begin{align*}
            \begin{split}
                \|f(t)\|_{L^{\infty}_{q}} = \underset{x,v}{\sup}|f(x,v,t)(1+|v|)^q|,\quad \|f(t)\|_{W^{N,\infty}_{q}} = \sum_{|\alpha|+|\beta|\leq N}\|\partial^{\alpha}_{\beta} f(t)\|_{L^{\infty}_{q}},
            \end{split}
        \end{align*}
        where $\alpha$, $\beta$ belong to $\mathbb{Z}_{+}$ and the differential operator $\partial^{\alpha}_{\beta}$ stands for $\partial^{\alpha}_{\beta} = \partial^{\alpha}_x \partial^{\beta}_v$.
    \item For any sequence $a^{n}_{i,j}$, we use the following notation for the weighted $L^{\infty}$ norm for of the sequence:
        \begin{align*}
            \begin{split}
                \|a^n\|_{L^{\infty}_{q}} = \underset{i,j}{\sup}|a^{n}_{i,j}(1+|v_j|)^q|.
            \end{split}
        \end{align*}
    \item In view of the above norms, we use the following notations:
        \begin{align*}
            \begin{split}
                \|f(t^n) - f^n\|_{L^{\infty}_{q}} = \underset{i,j}{\sup}|\{f(x_i,v_j,t^n) - f^n_{i,j}\}(1+|v_j|)^q|.
            \end{split}
        \end{align*}
\end{itemize}

\indent This paper is organized as follows. In Section \ref{sec2}, we provide the finite-difference Eulerian scheme for the VPBGK model. In Section \ref{sec3}, we state the main convergence result of the proposed scheme. In Section \ref{sec4}, we derive the stability of the discrete solution. In Section \ref{sec5}, we construct the consistent form of the VPBGK model to estimate the error between the continuous and discrete solutions. Finally, we prove the main theorem in Section \ref{sec6}.


\section{Description of the numerical scheme}\label{sec2}
In this section, we introduce an implicit-explicit (IMEX) scheme for the VPBGK model in a finite difference framework. In the method, we consider a fixed time step $\Delta t$ and uniform grids in space and velocity with mesh sizes $\Delta x$ and $\Delta v$, respectively. Throughout this paper, the following notations will be used to denote grid points: 
\begin{align*}
	t^n&=n\Delta t, \quad n=0,1,\dots,N_t,\cr
	x_i&= i \Delta x, \quad i=0,\pm1,\dots,\pm N_x, \pm(N_x+1),\dots,\cr
	v_j&=j\Delta v, \quad j=0,\pm1,\dots,\pm N_v,
\end{align*}
where $N_t \Delta t= T^f$, $N_x \Delta x= 1$, and $N_v\Delta v=V_M$. We divide the description of numerical methods into two parts: one is for the first equation in \eqref{VPBGK model} and the other is for the Poisson equation in \eqref{VPBGK model}.

\subsection{Description of the numerical scheme for the VPBGK model}
\,\\
\noindent$\bullet$ \textbf{Time discretization}\\
In the VPBGK model, we have the main stability issue, which is the stiffness arising in the relaxation term. To address this issue, we treat the convection and force terms explicitly and the relaxation term implicitly in the first equation of \eqref{VPBGK model}.
Letting $f^n=f(x,v,t^n)$ and $\mathbb{E}^n=\mathbb{E}(x,t^n)$, our scheme reads
\begin{align}\label{discretization 1}
    \begin{split}
        \frac{f^{n+1} -f^{n}}{\Delta t} + v\partial_{x}f^{n} + \mathbb{E}^n\partial_{v}f^n = \frac{1}{\varepsilon}(\mathcal{M}(f^{n+1}) - f^{n+1}).
    \end{split}
\end{align}
Although an implicit term $\mathcal{M}(f^{n+1})$ appears, we are able to compute this explicitly. Multiplying \eqref{discretization 1} by $\varphi(v)$ and integrating over $v$, the cancellation property of the BGK operator gives 
\begin{align}\label{conserv law}
    \begin{split}
        \int_{\mathbb{R}}f^{n+1}\left(1,v,\frac{|v|^2}{2}\right)dv = \int_{\mathbb{R}}\varphi(v)(f^{n}-\Delta t v\partial_{x}f^{n}- \Delta t \mathbb{E}^n\partial_{v}f^n)dv,\quad \varphi(v)=\left(1,v,\frac{|v|^{2}}{2}\right).
    \end{split}
\end{align}
This means that the macroscopic variables at time $t=t^{n+1}$ could be obtained by computing the right side of \eqref{conserv law}, and hence
$\mathcal{M}(f^{n+1})$ is computable explicitly. Finally, this time discretization allows us to compute $f^{n+1}$ explicitly.\\\\
$\bullet$ \textbf{Space and velocity discretization}\\
First, we explain how to approximate the convection term $v\partial_x f$. Define a flux function $F(f) = vf(x,v,t)$ and consider a function $\hat{F}$ which interpolates $vf$ in the sense of cell average:
\begin{align}\label{FVM formula x}
    \begin{split}
        F(f(x,\cdot,\cdot))=vf(x,\cdot,\cdot) = \frac{1}{\Delta x}\int ^{x+\frac{\Delta x}{2}} _ {x-\frac{\Delta x}{2}} \hat{F}(x)dx. 
    \end{split}
\end{align}
We will look for a function $\hat{F}$ which approximate $F_i=F(f(x_i,\cdot,\cdot))$:
\begin{align*}
    \begin{split}
        F_i = \frac{1}{\Delta x}\int ^{x_{i+\frac{1}{2}}} _ {x_{i-\frac{1}{2}}} \hat{F}(x)dx,\quad\text{ for all }i,
    \end{split}
\end{align*}
where $x_{i\pm\frac{1}{2}}:=x_i\pm\frac{\Delta x}{2}$. Differentiating (\ref{FVM formula x}) in $x$, the derivative of $F$ can be written as
\begin{align*}
    \begin{split}
        \partial_x F\big|_{x=x_i}=v\partial_x f(x_i,\cdot,\cdot) = \frac{1}{\Delta x}\left(\hat{F}\left(x_{i+\frac{1}{2}}\right)- \hat{F}\left(x_{i-\frac{1}{2}}\right)\right). 
    \end{split}
\end{align*}
Typically, $\hat{F}$ is a piecewise polynomial that has discontinuities at the cell boundaries $x_{i}\pm\frac{\Delta x}{2}$.
To ensure stability, we use upwinding to choose the correct direction from which information comes. For the construction of the first-order scheme, it is sufficient to consider a piecewise constant function for $\hat{F}$, i.e., we compute the numerical flux $\hat{F}_{i+\frac{1}{2}}$ as follows:
\begin{align}\label{Spatialflux}
    \begin{split}
        \hat{F}_{i+{\frac{1}{2}}} =
        v^{+}f(x_i)+
        v^{-}f(x_{i+1}),  
    \end{split}
\end{align}
where $v^{+} = \max(v,0)$ and $v^{-}=-\min(v,0)$. This means that we choose the value from the left for the positive flux: $\hat{F}_{i+{\frac{1}{2}}} =v^{+}f(x_i)$, and we choose the value from the right for the negative flux: $\hat{F}_{i+{\frac{1}{2}}} =v^{-}f(x_{i+1})$. Thus, the approximation to the space derivative is written as
\begin{align*}
    \begin{split}
        \partial_x (vf)\big|_{x=x_i}\approx\frac{1}{\Delta x}\left(\hat{F}_{i+\frac{1}{2}} - \hat{F}_{i-\frac{1}{2}} \right).
    \end{split}
\end{align*}
Similar to the space discretization, to approximate the force term $\mathbb{E}\partial_v f$, we define $G(f)=\mathbb{E}f(x,v,t)$ and consider $\hat{G}$ interpolating $G$ in the sense of cell average, such that
\begin{align}\label{FVM formula v}
    \begin{split}
        G(f(\cdot,v,\cdot)) = \mathbb{E}f(\cdot,v,\cdot)=\frac{1}{\Delta v}\int^{v+\frac{\Delta v}{2}}_{v-{\frac{\Delta v}{2}}} \hat{G}(v)dv.
    \end{split}
\end{align}
We look for a function $\hat{G}$ which satisfies 
\begin{align*}
    \begin{split}
        G_j=G(f(\cdot,v_j,\cdot))=\frac{1}{\Delta v}\int^{v_{{j}+{\frac{1}{2}}}}_{v_{j-\frac{1}{2}}}\hat{G}(v)dv,\quad\text{ for all }j.
    \end{split}
\end{align*}
where $v_{j\pm\frac{1}{2}}:=v_j\pm\frac{\Delta v}{2}$. The derivative of $G$ is represented as follows by differentiating \eqref{FVM formula v}:
\begin{align*}
    \begin{split}
        \partial_v G\big|_{v=v_j}=\mathbb{E} \partial_v f(\cdot,v_j,\cdot)=\frac{1}{\Delta v}\left(\hat{G}\left(v_{j+\frac{1}{2}}\right) -  \hat{G}\left(v_{j-\frac{1}{2}}\right)\right).
    \end{split}
\end{align*}
By upwinding, we compute the numerical flux $\hat{G}_{j+\frac{1}{2}}$ such that
\begin{align}\label{Veloflux}
    \begin{split}
        \hat{G}_{j+\frac{1}{2}} = \mathbb{E}^{+}f(v_j) + \mathbb{E}^{-}f(v_{j+1}),
    \end{split}
\end{align}
where $\mathbb{E}^{+}=\max(\mathbb{E},0)$ and $\mathbb{E}^{-}=-\min(\mathbb{E},0)$. Therefore, this leads to the approximation of the velocity derivative:
\begin{align*}
    \begin{split}
        \partial_v (\mathbb{E}f)\big|_{v=v_j}\approx\frac{1}{\Delta v}\left(\hat{G}_{j+\frac{1}{2}}-\hat{G}_{j-\frac{1}{2}} \right).
    \end{split}
\end{align*}\\
$\bullet$ \textbf{Numerical method on the full grid}\\
Before presenting our numerical method, we introduce the discrete macroscopic fields and local Maxwellian as follows.
\begin{definition}\label{discrete Mx and MF}
 We define the discrete macroscopic fields and the discrete local Maxwellian associated with $\{f^n_{i,j}\}$ by
		\begin{align*}
		    \begin{split}		        
			(\rho_i^n, \rho_i^n U_i^n , \rho^n_iT_i^n) &= \sum_{j}f_{i,j}^n \left(1,v_j,|v_j - U^n_i|^2\right)\Delta v,\quad\mathcal{M}(f_{i,j}^n)=\frac{\rho_i^n}{\sqrt{2 \pi T_i^n }}e^{-\frac{|v_j-U_i^n|^2}{2T_i^n}}.
		    \end{split}
		\end{align*}
\end{definition}
Finally, on the full grid, our scheme is given by
\begin{align}\label{implicit scheme}
    \begin{split}
        \frac{f_{i,j}^{n+1} -f_{i,j}^{n}}{\Delta t} + \frac{\hat{F}_{i+\frac{1}{2},j}^n - \hat{F}_{i-\frac{1}{2},j}^n}{\Delta x} + \frac{\hat{G}_{i,j+\frac{1}{2}}^n - \hat{G}_{i,j-\frac{1}{2}}^n}{\Delta v} = \frac{1}{\varepsilon}(\mathcal{M}(f^{n+1}_{i,j}) - f^{n+1}_{i,j}),
    \end{split}
\end{align}
where 
\begin{align*}
    \begin{split}
        \hat{F}^n_{i+\frac{1}{2},j}=v^{+}_jf^n_{i,j}+v^{-}_{j}f^n_{i+1,j},\quad \hat{G}^n_{i,j+\frac{1}{2}}={\mathbb{E}^n_i}^{+}f^n_{i,j}+{\mathbb{E}^n_i}^{-} f^n_{i,j+1}.
    \end{split}
\end{align*}
We note that the discrete electric field $\mathbb{E}^n_i$ will be defined in the next subsection. Now we let
\begin{align}\label{tilde f}
    \begin{split}
        \tilde{f}^n_{i,j}=  f_{i,j}^n -\Delta t \left(\frac{\hat{F}_{i+\frac{1}{2},j}^n - \hat{F}_{i-\frac{1}{2},j}^n}{\Delta x} + \frac{\hat{G}_{i,j+\frac{1}{2}}^n - \hat{G}_{i,j-\frac{1}{2}}^n}{\Delta v} \right).
    \end{split}
\end{align}
In view of \eqref{conserv law}, the macroscopic variables at time $t^{n+1}$ can be computed by 
\begin{align}\label{Scheme_mac}
    \begin{split}
        &(\rho^{n+1}_{i},\rho^{n+1}_{i}U^{n+1}_{i},E^{n+1}_{i}) 
        =\sum_j \varphi(v_j)f^{n+1}_{i,j}\Delta v\approx \sum_{j}\varphi(v_{j})\tilde{f}^{n}_{i,j}\Delta v = (\tilde{\rho}^{n}_{i},\tilde{\rho}^{n}_{i}\tilde{U}^{n}_{i},\tilde{E}^{n}_{i}).
    \end{split}
\end{align}
Then, the discrete Maxwellian which defined implicitly in \eqref{implicit scheme} can be computed explicitly as  $\mathcal{M}(f^{n+1}_{i,j}) \approx\mathcal{M}(\tilde{f}^{n}_{i,j})$. In conclusion, our scheme for the VPBGK model reads
\begin{align}\label{VPBGK_Scheme}
    \begin{split}
        f^{n+1}_{i,j} = \frac{\varepsilon \tilde{f}^{n}_{i,j} + \Delta t \mathcal{M}(\tilde{f}_{i,j}^{n})}{\varepsilon + \Delta t},\quad f^{n+1}_{i,-N_v-1}=f^{n+1}_{i,-N_v},\quad f^{n+1}_{i,N_v+1}=f^{n+1}_{i,N_v},\quad\text{for }n\geq 0.
    \end{split}
\end{align}
The initial iteration is defined as follows, with a homogeneous Neumann boundary condition:
\begin{align*}
    \begin{split}
        f^0_{i,j}=
        \begin{cases}
        f_0(x_i,v_j),&\text{ for }j\in[-N_v,N_v],\cr
        f_0(x_i,V_M),&\text{ for }j=N_{v}+1,\cr
        f_0(x_i,-V_M),&\text{ for }j=-N_{v}-1,\cr
        0,&\text{ otherwise.}
        \end{cases}
    \end{split}
\end{align*}
\subsection{Computation of the electric field}
The approximation of the electric field is carried out using the explicit form of the Green kernel defined on the one-dimensional torus \cite{besse2004convergence,besse2008convergence1,besse2008convergence2,cottet1984particle,filbet2001convergence}:
\begin{align}\label{Green kernel}
    K(x,y) = 
\begin{cases}
    y,  & \mbox{for } 0 \leq y \leq x, \\
    y-1, & \mbox{for } x< y\leq 1.
\end{cases}
\end{align}
The continuous electric field is given by
\begin{align}\label{continuous electric field}
    \begin{split}
        \mathbb{E}(x_i,t^n) = \int^{1}_{0} K(x_i,y)(\rho(y,t^n)-1)dy.
    \end{split}
\end{align}
Similarly, we calculate the discrete electric field as follows:
\begin{align}\label{Discrete field 2}
    \begin{split}
        \mathbb{E}^n_i = \sum_k K(x_i,y_k)(\rho^n_k -1)\Delta y,
    \end{split}
\end{align}
where $k$ shares the same mesh size as $i$.\\
\section{Main results}\label{sec3}
In this section, we present our main results. To do this, we record the well-posedness for the VPBGK model, which will be used in the error estimate.
\begin{theorem}\label{analysis}\textnormal{(Existence theorem)}\cite{YunPark2026}
    Let $q>3$ and $\alpha\in[1,2]$. Assume that $f_0\in C^2(\mathbb{T}\times\mathbb{R})$ and $\|f_0\|_{W^{2,\infty}_q}<\infty$. Assume that there is a positive constant $C_{0}$ which satisfies
    \begin{align*}
        \begin{split}
            f_0(x,v)\geq C_{0}e^{-|v|^{\alpha}}.
        \end{split}
    \end{align*}
    Then, there exists $T^f>0$ such that the system \eqref{VPBGK model} has a unique classical solution $(f,\mathbb{E})$ that satisfies 
    \begin{enumerate}
        \item The velocity distribution function $f$ is uniformly bounded in $[0,T^f]:$
        \begin{align*}
            \begin{split}
                \|f(t)\|_{W^{2,\infty}_q}\leq Ce^{Ct}(\|f_0\|_{W^{2,\infty}_q}+1),\text{ for }t\in[0,T^f],
            \end{split}
        \end{align*}
        where $C$ is some constant that depends on $q$, $f_0$, and $T^f$.
        \item The macroscopic fields satisfy that
        \begin{align*}
            \begin{split}
                &\rho(t,x)\geq C_{q,T^f},\quad T(t,x)\geq C_{q,f_0,T^f},\quad T(t,x) + |U(t,x)|^2\leq C_{q,f_0,T^f}.
            \end{split}
        \end{align*}
        \item The partial derivatives of $f$ are continuous in $[0,T^f]$:
        \begin{equation*}
            \begin{split}
                \partial^{\alpha}_{x}\partial^{\beta}_vf\in C([0,T^f]\times \mathbb{T}\times\mathbb{R}),\quad \text{for }|\alpha|+|\beta|\leq 2.
            \end{split}
        \end{equation*}
        \item The electric field $\mathbb{E}$ is uniformly bounded in $[0,T^f]:$
        \begin{align*}
            \begin{split}
                \|\mathbb{E}(t)\|_{L^{\infty}_x} \leq C_{q,f_0,T^f}.
            \end{split}
        \end{align*}
    \end{enumerate}
\end{theorem}
We now state our main theorem.
\begin{theorem}\label{Main Thm}\textnormal{(Main convergence theorem)}
    Let $f$ be the unique smooth solution of (\ref{VPBGK model}) corresponding to a non-negative initial datum $f_0$ satisfying the hypotheses of Theorem \ref{analysis} and $\|f_0\|_{W^{2,\infty}_{q+2}}<\infty$. Let $f^{n}$ and $\mathbb{E}^{n}$ be the discrete solution constructed by (\ref{VPBGK_Scheme}) and (\ref{Discrete field 2}) given in Section \ref{sec2}. Then there exists positive numbers $r_{\Delta x}$, $r_{\Delta v}$, $C_{\mathbb{E},1}$, $C_{\mathbb{E},2}$, and $C_{\mathcal{M}}$ which are determined in Section \ref{sec4}, such that if $\Delta x<r_{\Delta x}$, $\Delta v<r_{\Delta v}$, and $\Delta t$ satisfy
    \begin{align}\label{CFL condition}
        \begin{split}
            \eta_1\frac{\Delta v}{C_{\mathbb{E},1}}<\Delta t<\eta_2\frac{\Delta v}{C_{\mathbb{E},1}},\quad\text{for }\;0<\eta_1<\eta_2<1,
        \end{split}
    \end{align}
    and $V_M$ is determined by
    \begin{align}\label{V_M large}
        \begin{split}
            V_M=\frac{C_{\mathbb{E},1}}{(\Delta v)^{\gamma}},\quad \frac{1}{2}\leq \gamma<1,
        \end{split}
    \end{align}
    then, we have
    \begin{align*}
        \begin{split}
            \|f(T^{f}) - f^{N_t}\|_{L^{\infty}_{q}} + \|\mathbb{E}(T^{f}) - \mathbb{E}^{N_t}\|_{L^{\infty}_{x}} \leq C\left\{\Delta x + \Delta v + \Delta t + \frac{1}{V^{q-3}_M} \right\},
        \end{split}
    \end{align*}
    where $N_t$ is defined by $T^f = N_t \Delta t$ and $C$ depends on $q$, $T^f$, $f_0$, $\alpha$, $\eta_1$, $\eta_2$, and $\varepsilon$.
\end{theorem}

\begin{remark}
    In Theorem \ref{Main Thm}, under \eqref{CFL condition}, \eqref{V_M large}, and $\Delta v<r_{\Delta v}$, we have
    \begin{align*}
        \begin{split}
            \Delta t\approx \Delta v, \quad V_M\approx (\Delta v)^{-\gamma}, \quad \Delta v\approx (\Delta x)^{\frac{1}{1-\gamma}}.
        \end{split}
    \end{align*}    
    Then, the error estimate is presented by
    \begin{align*}
        \begin{split}
            \|f(T^{f}) - f^{N_t}\|_{L^{\infty}_{q}} + \|\mathbb{E}(T^{f}) - \mathbb{E}^{N_t}\|_{L^{\infty}_{x}}\leq \Delta x^{\min\left(1,\frac{(q-3)\gamma}{1-\gamma}\right)}.
        \end{split}
    \end{align*}
    Therefore, if $q\geq 4$, a first-order convergence rate is achieved automatically. For the case $3<q<4$, we can derive a first-order rate by selecting $\gamma\in\Big[\frac{1}{q-2},1\Big)$.
\end{remark}

\begin{remark}
While Theorem \ref{analysis} establishes a local-in-time solution for \eqref{VPBGK model}, the convergence analysis imposes no restriction on the final time $T^f$.
\end{remark}
\section{Stability estimates}\label{sec4}
The purpose of this section is to show that the discrete solution $\tilde{f}^{n}_{i,j}$, and the corresponding electric field and macroscopic fields have uniform-in-$n$ bounds. In the following stability estimates, we use the relation \eqref{tilde f}, which implies that $\tilde{f}^n_{i,j}$ can be represented by the convex combination of $f^{n}_{i,j}$, $f^{n}_{i-1,j}$, $f^{n}_{i+1,j}$, $f^{n}_{i,j-1}$, $f^{n}_{i,j+1}$:
\begin{align}\label{convex_combi}
        \begin{split}
            \tilde{f}^{n}_{i,j}&= \left(1-\frac{\Delta t}{\Delta x}v_j-\frac{\Delta t}{\Delta v}\mathbb{E}^{n}_{i}\right)f^{n}_{i,j}+ \frac{\Delta t}{\Delta x}v_j^{+} f^{n}_{i-1,j} +\frac{\Delta t}{\Delta x}v_j^{-} f^{n}_{i+1,j}\cr  
            &+\frac{\Delta t}{\Delta v}{\mathbb{E}^{n}_{i}}^{+} f^{n}_{i,j-1} + \frac{\Delta t}{\Delta v}{\mathbb{E}^{n}_{i}}^{-} f^{n}_{i,j+1},
        \end{split}
    \end{align}
under the CFL condition:
\begin{align}\label{CFL condition 0}
    \begin{split}
        \sup_{i,j}\left(\frac{\Delta t}{\Delta x}|v_j| + \frac{\Delta t}{\Delta v}|\mathbb{E}^n_i|\right)< 1.
    \end{split}
\end{align}
Recalling \eqref{VPBGK_Scheme}, we assume the Neumann boundary condition in velocity:
\begin{align}\label{Neumann boundary condition}
    \begin{split}
        f^n_{i,-N_v-1}=f^n_{i,-N_v},\quad f^n_{i,N_v+1}=f^n_{i,N_v}.
    \end{split}
\end{align}

\indent We begin by defining several specific constants that arise for technical reasons. 
\begin{definition}\label{Constant series}
    We define $D_{\alpha,\beta}$, $\tilde{D}_{\alpha,\beta}$, and $\bar{D}_{q-m}$ by
    \begin{align*}
        \begin{split}
            D_{\alpha,\beta} = \int_{\mathbb{R}} e^{-|v|^{\alpha}-\beta|v|}dv,\quad \tilde{D}_{\alpha,\beta}=\sup_v\left\{(1+|v|)^qe^{-|v|^{\alpha}-\beta|v|}\right\},\quad \bar{D}_{q-m}=\int_{\mathbb{R}}\frac{1}{(1+|v|)^{q-m}}dv,
        \end{split}
    \end{align*}
    and we define $C_{\mathbb{E},1}$ and $C_{\mathbb{E},2}$ by
    \begin{align*}
        \begin{split}
            C_{\mathbb{E},1} =(2+T^f)e^{\frac{T^f}{\varepsilon}}+1 ,\quad C_{\mathbb{E},2}=(2^q-1)C_{\mathbb{E},1}.
        \end{split}
    \end{align*}
\end{definition}

\begin{definition}
    We define the constants $a_1$, $a_2$, $a_3$, $a_4$, and $a_5$ by
    \begin{align*}
        \begin{split}
            &a_1 = \frac{(C_0D_{\alpha,2\alpha C_{\mathbb{E},1}T^f})^{\frac{5}{9}}}{2^{\frac{16}{9}}(1+C_{\mathbb{E},2})^{\frac{5}{9}}C_M^{\frac{2}{9}}\|f_0\|^{\frac{5}{9}}_{L^{\infty}_q}}\exp\left({-\frac{5}{9}\left(2\alpha C_{\mathbb{E},1}+\frac{\alpha C^2_{\mathbb{E},1}(T^f+1)}{\eta_1}+C_{\mathbb{E},2}+\frac{C_{\mathcal{M}}}{\varepsilon} \right)T^f}\right),\cr
            &a_2 = \left\{\frac{(C_0)^2D_{\alpha,2\alpha C_{\mathbb{E},1}T^f} \tilde{D}_{\alpha,2\alpha C_{\mathbb{E},1}T^f}}{4(q-3)\bar{D}_{q-2}(C_0D_{\alpha,2\alpha C_{\mathbb{E},1}T^f} +8)(1+C_{\mathbb{E},2})\|f_0\|_{L^{\infty}_q}}\right\}^{\frac{1}{q-1}}\cr
            &\quad \times \exp\left({-\frac{1}{q-1}\left(4\alpha C_{\mathbb{E},1}+\frac{2\alpha C^2_{\mathbb{E},1}(T^f+1)}{\eta_1}+C_{\mathbb{E},2}+\frac{C_{\mathcal{M}}+1}{\varepsilon} \right)T^f}\right),\cr
            &a_3 = \frac{(C_0D_{\alpha, 2\alpha C_{\mathbb{E},1}T^f})^{\frac{2q+3}{3q+3}}}{2^{\frac{4q+12}{3q+3}}(1+C_{\mathbb{E},2})^{\frac{2q+3}{3q+3}}C_M^{\frac{2q}{3q+3}}\|f_0\|^{\frac{2q+3}{3q+3}}_{L^{\infty}_q}}\cr
            &\quad \times \exp\left({-\frac{2q+3}{3q+3}\left(2\alpha C_{\mathbb{E},1}+\frac{\alpha C^2_{\mathbb{E},1}(T^f+1)}{\eta_1}+C_{\mathbb{E},2} + \frac{C_{\mathcal{M}}}{\varepsilon}\right)T^f}\right),\cr
            &a_4=\left(\frac{C_0 D_{\alpha,2\alpha C_{\mathbb{E},1}T^f}}{4(1+C_{\mathbb{E},2})C_M\|f_0\|_{L^{\infty}_q} }\exp\left({-\left(2\alpha C_{\mathbb{E},1}+\frac{\alpha C^2_{\mathbb{E},1}(T^f+1)}{\eta_1}+C_{\mathbb{E},2}+\frac{C_{\mathcal{M}}}{\varepsilon } \right)T^f} \right)^{\frac{2}{3}}\right),\cr
            &a_5=\left(2 C^{1-q}_{\mathbb{E},1}(1+C_{\mathbb{E},2})e^{\left(C_{\mathbb{E},2}+\frac{2(C_{\mathcal{M}}-1)}{\varepsilon} \right)T^f}\|f_0\|_{L^{\infty}_q} \right)^{-\frac{1}{q\gamma}}
        \end{split}
    \end{align*}
\end{definition}

Next, we introduce the stability estimate $\mathcal{A}^n$.
\begin{definition}\label{stability condition}\textnormal{(Key stability estimates)}
     We say that $f^n_{i,j}$ satisfies the stability estimate $\mathcal{A}^n$ if the following six statements hold ($C_M$ and $C_{\mathcal{M}}$ will be determined in Lemma \ref{Mac_bound_lemma}, Lemma \ref{Mx to f}): 
     \begin{align*}
        \begin{split}
            (A_1^{n})\;&\text{Uniform $L^1$ bound for the discrete solution:}\cr 
            &\sum_{i,j}f^n_{i,j}\Delta x\Delta v\leq (2+T^f)e^{\frac{T^f}{\varepsilon}}.\cr
            (A_2^n)\,&\text{Uniform upper bound for the discrete electric field:}\cr
            &\|\mathbb{E}^n\|_{L^{\infty}_x}\leq C_{\mathbb{E},1}.\cr
            (A_3^{n})\,&\text{Uniform Maxwellian lower bound for the discrete solution:}\cr
            &\tilde{f}^n_{i,j} \geq \frac{C_0}{2}e^{-\left(2\alpha C_{\mathbb{E},1}+\frac{\alpha C^2_{\mathbb{E},1}(T^f+1)}{\eta_1}+\frac{1}{\varepsilon }\right)T^f}e^{-|v_j|^\alpha}e^{-2\alpha C_{\mathbb{E},1}T^f|v_j|}.\cr
            (A_4^{n})\,&\text{Uniform upper bound for the discrete solution:}\cr
            &\|\tilde{f}^{n}\|_{L^{\infty}_{q}} \leq (1+C_{\mathbb{E},2})e^{\left(C_{\mathbb{E},2} + \frac{C_{\mathcal{M}}-1}{\varepsilon}\right)T^{f}}\|f_{0}\|_{L^{\infty}_{q}}.\cr
            (A_5^{n})\,&\text{Uniform lower bound for the discrete macroscopic fields:}\cr
            &\tilde{\rho}^n_i \geq \frac{C_0 D_{\alpha,2\alpha C_{\mathbb{E},1}T^f}}{4}e^{-\left(2\alpha C_{\mathbb{E},1}+\frac{\alpha C^2_{\mathbb{E},1}(T^f+1)}{\eta_1}+\frac{1}{\varepsilon }\right)T^f},\cr
            &\tilde{T}^n_i\geq \left(\frac{C_0 D_{\alpha,2\alpha C_{\mathbb{E},1}T^f}}{4(1+C_{\mathbb{E},2})C_M\|f_0\|_{L^{\infty}_q} }e^{-\left(2\alpha C_{\mathbb{E},1}+\frac{\alpha C^2_{\mathbb{E},1}(T^f+1)}{\eta_1}+C_{\mathbb{E},2}+\frac{C_{\mathcal{M}}}{\varepsilon } \right)T^f} \right)^{\frac{2}{3}}.\cr   
            (A_6^{n})\,&\text{Uniform upper bound for the discrete macroscopic fields:}\cr
            &\|\tilde{\rho}^n\|_{L^{\infty}_{x}},\,\|\tilde{U}^n\|_{L^{\infty}_{x}},\,\|\tilde{T}^n\|_{L^{\infty}_{x}}\cr 
            &\hspace{1.3cm}\leq  \bar{D}_{q-2}\left(1+\frac{8}{C_0  D_{\alpha,2\alpha C_{\mathbb{E},1}T^f}} \right)(1+C_{\mathbb{E},2})e^{\left(2\alpha C_{\mathbb{E},1}+\frac{\alpha C^2_{\mathbb{E},1}(T^f+1)}{\eta_1}+C_{\mathbb{E},2} + \frac{C_{\mathcal{M}}}{\varepsilon}\right)T^{f}}\|f_{0}\|_{L^{\infty}_{q}}. 
        \end{split}
    \end{align*}
\end{definition}

We are now ready to state the main result of this section.

\begin{theorem}\label{stability theorem}\textnormal{(Stability theorem)}
    Let $\alpha\in[1,2]$, $\beta\geq 0$, $\gamma\in(0,1)$, $0<\eta_1<\eta_2<1$, and $q>3$. Choose $l>0$ sufficiently small such that $\Delta x,\Delta v<l$ implies
    \begin{align*}
        \begin{split}
            &\frac{1}{2}D_{\alpha,\beta}\leq \sum_{j}e^{-|v_j|^{\alpha}-\beta|v_j|}\Delta v\leq 2D_{\alpha,\beta},\cr
            &\frac{1}{2}\tilde{D}_{\alpha,\beta}\leq \sup_j \left\{(1+|v_j|)^q e^{-|v_j|^{\alpha}-\beta|v_j|} \right\}\leq 2\tilde{D}_{\alpha,\beta},\cr
            &\frac{1}{2}\bar{D}_{q-m} \leq \sum_{j}\frac{\Delta v}{(1+|v_{j}|)^{q-m}}\leq 2\bar{D}_{q-m},\cr
            &\sum_{i,j}f_0(x_i,v_j)\Delta x\Delta v\leq 2\int_\mathbb{T}\int_{\mathbb{R}}f_0(x,v)dxdv\leq 2.
        \end{split}
    \end{align*}
    Assume that $V_M$ is chosen to satisfies \eqref{V_M large}, and $\Delta x$, $\Delta v$, $\Delta t$ satisfy
    \begin{align*}
        \begin{split}
            \Delta x<r_{\Delta x},\;\;\Delta v<r_{\Delta v},\;\;\eta_1\frac{\Delta v}{C_{\mathbb{E},1}}<\Delta t<\eta_2\frac{\Delta v}{C_{\mathbb{E},1}},
        \end{split}
    \end{align*}
    where
    \begin{align*}
        \begin{split}
            &r_{\Delta x}=\min\left\{l,\frac{1}{4\alpha C_{\mathbb{E},1}},\frac{1}{2} \right\},\cr &r_{\Delta v} =\min\left\{a_{1},a_{2},a_{3},l,l \sqrt{2a_4},a_5,\left(\frac{1-\eta_2}{\eta_2}\Delta x\right)^{\frac{1}{1-\gamma}},\frac{1}{4\alpha\{ C_{\mathbb{E},1}(T^f+1)+1\}},\frac{1}{2} \right\}.
        \end{split}
    \end{align*}
    Then, $f^{n}_{i,j}$ satisfies $\mathcal{A}^{n}$ for all $n\geq 0$.
\end{theorem}

\begin{remark}\label{Delta x Delta v restriction}  
    The time step $\Delta t$ introduced in Theorem \ref{stability theorem} always satisfies the CFL condition \eqref{CFL condition 0}. Since the discrete velocity $v_j$ and the discrete electric field $\mathbb{E}^n_i$ are bounded by the maximum velocity $V_M$ from the truncated domain and by the constant $C_{\mathbb{E},1}$ from the stability estimate $A^n_2$, respectively, we have
    \begin{align*}
        \begin{split}
            \frac{\Delta t}{\Delta x}|v_j| + \frac{\Delta t}{\Delta v}|\mathbb{E}^n_i|\leq \frac{V_M\Delta t}{\Delta x}+\frac{C_{\mathbb{E},1}\Delta t}{\Delta v}<\frac{\eta_2 V_M\Delta v}{C_{\mathbb{E},1}\Delta x}+\eta_2.
        \end{split}
    \end{align*}
    Using $\Delta v<\left(\frac{1-\eta_2}{\eta_2}\Delta x\right)^{\frac{1}{1-\gamma}}$ and the choice of $V_M$ in \eqref{V_M large}, we obtain
    \begin{align*}
        \begin{split}
            \frac{\eta_2 V_M\Delta v}{C_{\mathbb{E},1}\Delta x}+\eta_2=(1-\eta_2)+\eta_2<1.
        \end{split}
    \end{align*}
    Moreover, since $\eta_2<1$ and $C_{\mathbb{E},1}>1$, $\Delta t$ is automatically less than 1.
\end{remark}

\begin{remark}
    Under the hypothesis for Theorem \ref{stability theorem}, the positivity of the discrete solution $\tilde{f}^n_{i,j}$ is obtained by induction. Since \eqref{CFL condition 0} always holds, $\tilde{f}^n_{i,j}$ is positive by \eqref{convex_combi}. Then, we can directly show that $f^n_{i,j}>0$ from our scheme \eqref{VPBGK_Scheme}.
    \begin{align*}
        \begin{split}
            f^{n}_{i,j}>\frac{\varepsilon}{\varepsilon+\Delta t}\tilde{f}^n_{i,j}>0.
        \end{split}
    \end{align*}
\end{remark}

To prove Theorem \ref{stability theorem}, we divide this section into two subsections. In the first subsection, we establish the technical lemmas for the proof. Using these lemmas, we show each stability estimate $A^n_k$ $(k=1,\cdots,6)$ in the second subsection. 

\subsection{Technical lemmas}
The following lemma gives estimates for the discrete macroscopic fields.
\begin{lemma}\label{Mac_bound_lemma}
    Let $q>3$. Assume that $f^{n}_{i,j}$ satisfies $\mathcal{A}^{n}$ and $\Delta v<r_{\Delta v}$ where $r_{\Delta v}$ is stated in the hypothesis of Theorem \ref{stability theorem}. Then, for all $n$, there exists a positive constant $C_{M}$ which depends on $q$, such that
\begin{align*}
    \begin{split}
        &\text{(1) } \tilde{\rho}^{n}_{i} ({\tilde{T}^{n}_{i}})^{-\frac{1}{2}} \leq C_{M}\|\tilde{f}^{n}\|_{L^{\infty}_{q}}, \cr
        &\text{(2) } \tilde{\rho}^{n}_{i} ( {\tilde{T}^{n}_{i}} + |\tilde{U}^{n}_{i}|^{2})^{\frac{q-1}{2}}\leq C_{M}\|\tilde{f}^{n}\|_{L^{\infty}_{q}},\cr
        &\text{(3) } \tilde{\rho}^{n}_{i}|\tilde{U}^{n}_{i}|^{q+1} [(\tilde{T}^{n}_{i} + |\tilde{U}^{n}_{i}|^{2})\tilde{T}^{n}_{i}]^{-\frac{1}{2}} \leq C_{M}\|\tilde{f}^{n}\|_{L^{\infty}_{q}}.
    \end{split}
\end{align*}
\end{lemma}
\begin{proof}
        (1) We split $\tilde{\rho}^{n}_{i}$ into two parts:
        \begin{align*}
            \begin{split}
                \tilde{\rho}^{n}_{i} = \sum_{|v_{j}-\tilde{U}^{n}_{i}| \leq r+\Delta v}\tilde{f}^{n}_{i,j}\Delta v+ \sum_{|v_{j}-\tilde{U}^{n}_{i}| > r+\Delta v}\tilde{f}^{n}_{i,j}\Delta v\equiv I_{1} + I_{2}.
            \end{split}
        \end{align*}
        Then, we have
        \begin{align*}
            \begin{split}
                I_{1} \leq \|\tilde{f}^{n}\|_{L^{\infty}_{q}} \sum_{|v_{j}-\tilde{U}^{n}_{i}| \leq r+\Delta v} \Delta v\leq \|\tilde{f}^{n}\|_{L^{\infty}_{q}} \int_{|v-\tilde{U}^{n}_{i}| \leq 2r+\Delta v}dv  = 4\|\tilde{f}^{n}\|_{L^{\infty}_{q}}(r+\Delta v),
            \end{split}
        \end{align*}
        and 
        \begin{align*}
            \begin{split}
                I_{2}= \sum_{|v_{j}-\tilde{U}^{n}_{i}| > r+\Delta v} \tilde{f}^{n}_{i,j}\frac{|v_{j}-\tilde{U}^{n}_{i}|^{2}}{|v_{j}-\tilde{U}^{n}_{i}|^{2}}\Delta v\leq \frac{1}{(r+\Delta v)^{2}} \sum_{j} \tilde{f}^{n}_{i,j}|v_{j} - \tilde{U}^{n}_{i}|^{2}\Delta = \frac{\tilde{\rho}^{n}_{i} \tilde{T}^{n}_{i}}{(r+\Delta v)^{2}}.
            \end{split}
        \end{align*}
        These lead to
        \begin{align*}
            \begin{split}
                \tilde{\rho}^{n}_{i} \leq 4\|\tilde{f}^{n}\|_{L^{\infty}_{q}}(r+\Delta v)+ \frac{\tilde{\rho}^{n}_{i} \tilde{T}^{n}_{i}}{(r+\Delta v)^{2}}.
            \end{split}
        \end{align*}
       We optimize $r$ by equating the two terms on the right-hand sides, i.e., 
        \begin{align}\label{r setting 1}
            \begin{split}
                r+\Delta v = \left(\frac{\tilde{\rho}^{n}_{i} \tilde{T}^{n}_{i}}{4\|\tilde{f}^{n}\|_{L^{\infty}_{q}}}\right)^{\frac{1}{3}}.
            \end{split}
        \end{align}
       Since $f^{n}_{i,j}$ satisfies $\mathcal{A}^n$, we have
        \begin{align*}
            \begin{split}
               \left(\frac{\tilde{\rho}^{n}_{i} \tilde{T}^{n}_{i}}{4\|\tilde{f}^{n}\|_{L^{\infty}_{q}}}\right)^{\frac{1}{3}} \geq a_{1} > \Delta v.
            \end{split}
         \end{align*}
        Thus, we can always find a positive number $r$ satisfying \eqref{r setting 1}, and we have
        \begin{align*}
            \begin{split}
                \tilde{\rho}^{n}_{i} \leq C_{M}(\tilde{T}^{n}_{i})^{\frac{1}{2}}\|\tilde{f}^{n}\|_{L^{\infty}_{q}}.
            \end{split}
        \end{align*}
        By condition $A_5^n$, $\tilde{T}^{n}_{i}$ is always positive. This completes the proof.\\
        \indent (2) Similar to the proof of (1), we split $\tilde{\rho}^n_i(\tilde{T}^n_i + |\tilde{U}^n_i|^2)$ into two parts:
        \begin{align*}
            \begin{split}
                \tilde{\rho}^{n}_{i} (\tilde{T}^{n}_{i} + |\tilde{U}^{n}_{i}|^{2}) = \sum_{|v_{j}|
                > r+2\Delta v}\tilde{f}^{n}_{i,j}|v_{j}|^{2}\Delta v + \sum_{|v_{j}|
                \leq r+2\Delta v}\tilde{f}^{n}_{i,j}|v_{j}|^{2}\Delta v \equiv I_{1} + I_{2}.
            \end{split}
        \end{align*}
        For $I_1$, we compute
        \begin{align*}
            \begin{split}
                I_{1} &=  \sum_{|v_{j}|
                > r+2\Delta v}\tilde{f}^{n}_{i,j}\frac{|v_{j}|^{q}}{|v_{j}|^{q-2}}\Delta v\cr
                &\leq \| \tilde{f}^{n}\|_{L^{\infty}_{q}}\sum_{|v_{j}|
                > r+2\Delta v}\frac{1}{|v_{j}|^{q-2}}\Delta v\cr
                &\leq \|\tilde{f}^{n}\|_{L^{\infty}_{q}}\int_{|v|>r+\Delta v}\frac{dv}{|v|^{q-2}}\cr
                &= \frac{2}{(q-3){(r+\Delta v)}^{q-3}}\| \tilde{f}^{n}\|_{L^{\infty}_{q}}.
            \end{split}
        \end{align*}
        For $I_2$, we have
        \begin{align*}
            \begin{split}
                I_{2} &\leq (r+2\Delta v)^2\sum_{j} \tilde{f}^{n}_{i,j}\Delta v \leq 4\tilde{\rho}^{n}_{i}(r+\Delta v)^{2}.
            \end{split}
        \end{align*}
        If we set $(r+\Delta v)^{q-1} = \frac{\|\tilde{f}^{n}\|_{L^{\infty}_{q}}}{2(q-3)\tilde{\rho}^{n}_{i}}$, then we obtain
        \begin{align*}
            \begin{split}
                \tilde{\rho}^{n}_{i} (\tilde{T}^{n}_{i} + |\tilde{U}^{n}_{i}|^{2}) &\leq \frac{2}{(q-3){(r+\Delta v)}^{q-3}}\| \tilde{f}^{n}\|_{L^{\infty}_{q}}+4\tilde{\rho}^{n}_{i}(r+\Delta v)^{2}\leq C(\tilde{\rho}^{n}_{i})^{1-\frac{2}{q-1}}\|\tilde{f}^{n}\|^{\frac{2}{q-1}}_{L^{\infty}_{q}}.
            \end{split}
        \end{align*}
        Since $f^{n}_{i,j}$ satisfies $\mathcal{A}^{n}$, $\left(\frac{\|\tilde{f}^{n}\|_{L^{\infty}_{q}}}{2(q-3)\tilde{\rho}^{n}_{i}}\right)^{\frac{1}{q-1}} \geq a_{2} > \Delta v$, that is, there always exists positive $r$. This gives the desired result.\\
        \indent (3) We split $\tilde{\rho}^{n}_{i}|\tilde{U}^{n}_{i}|$ as follows:
        \begin{align*}
            \begin{split}
                \tilde{\rho}^{n}_{i} |\tilde{U}^{n}_{i}| \leq \sum_{|v_{j}-\tilde{U}^{n}_{i}| \leq r+\Delta v}\tilde{f}^{n}_{i,j}|v_{j}|\Delta v + \sum_{|v_{j}-\tilde{U}^{n}_{i}| > r+\Delta v}\tilde{f}^{n}_{i,j}|v_{j}|\Delta v\equiv I_1 + I_2.
            \end{split}
        \end{align*}
        Then, by H\"older inequality, we obtain
        \begin{align*}
            \begin{split}
                I_{1} &\leq \left( \sum_{|v_{j}-\tilde{U}^{n}_{i}|\leq r+\Delta v}\tilde{f}^{n}_{i,j}\Delta v \right)^{1-\frac{1}{q}} \left( \sum_{|v_{j}-\tilde{U}^{n}_{i}|\leq r+\Delta v}\tilde{f}^{n}_{i,j}|v_{j}|^{q}\Delta v\right)^{\frac{1}{q}}\cr
                &\leq (\tilde{\rho}^{n}_{i})^{1-\frac{1}{q}}(4\|\tilde{f}^{n}_{i}\|_{L^{\infty}_{q}}(r+\Delta v))^{\frac{1}{q}}.
            \end{split}
        \end{align*}
        The Cauchy-Schwarz inequality gives 
        \begin{align*}
            \begin{split}
                I_{2} &\leq \frac{1}{r+\Delta v}\sum_{|v_{j}-\tilde{U}^{n}_{i}| >r+\Delta v}\tilde{f}^{n}_{i,j}|v_{j}-\tilde{U}^{n}_{i}||v_{j}|\Delta v \cr
                &\leq \frac{1}{r+\Delta v} \left(\sum_{j} \tilde{f}^{n}_{i,j}|v_{j}|^{2}\Delta v \right)^{\frac{1}{2}} \left( \sum_{j}\tilde{f}^{n}_{i,j}|v_{j}-\tilde{U}^{n}_{i}|^{2}\Delta v\right)^{\frac{1}{2}}\cr
                &\leq \frac{1}{r+\Delta v} (\tilde{\rho}^{n}_{i}(\tilde{T}^{n}_{i} + |\tilde{U}^{n}_{i}|^{2}))^{\frac{1}{2}}(\tilde{\rho}^{n}_{i}\tilde{T}^{n}_{i})^{\frac{1}{2}}\cr
                &= \frac{\tilde{\rho}^{n}_{i}}{r+\Delta v}(\tilde{T}^{n}_{i} + |\tilde{U}^{n}_{i}|^{2})^{\frac{1}{2}}(\tilde{T}^{n}_{i})^{\frac{1}{2}}. 
            \end{split}
        \end{align*}
        Combining these, we get
        \begin{align*}
            \begin{split}
                \tilde{\rho}^n_i |\tilde{U}^n_i|\leq (\tilde{\rho}^{n}_{i})^{1-\frac{1}{q}}(4\|\tilde{f}^{n}_{i}\|_{L^{\infty}_{q}}(r+\Delta v))^{\frac{1}{q}}+\frac{\tilde{\rho}^{n}_{i}}{r+\Delta v}(\tilde{T}^{n}_{i} + |\tilde{U}^{n}_{i}|^{2})^{\frac{1}{2}}(\tilde{T}^{n}_{i})^{\frac{1}{2}}. 
            \end{split}
        \end{align*}
        To obtain the third result, We set $r$ as 
        \begin{align*}
            \begin{split}
                (r+\Delta v) =\left( \frac{(\tilde{\rho}^{n}_{i})^{\frac{1}{q}}(\tilde{T}^{n}_{i} + |\tilde{U}^{n}_{i}|^{2})^{\frac{1}{2}}(\tilde{T}^{n}_{i})^{\frac{1}{2}}}{4^{\frac{1}{q}}\|\tilde{f}^{n}\|^{\frac{1}{q}}_{L^{\infty}_{q}}}\right)^{\frac{q}{q+1}}.
            \end{split}
        \end{align*}
       Note that $r$ is always positive since the right-hand side of the above is greater than or equal to $a_3$ by condition $\mathcal{A}^n$.
\end{proof}

Applying the above lemma, we control the local Maxwellian by the discrete solution $\tilde{f}^n_{i,j}$ in $L^{\infty}_q$ norm.
\begin{remark}
    To avoid confusion, we state the definition of $\mathcal{M}(\tilde{f}^n_{i,j})$ as follows:
    \begin{align*}
        \begin{split}
            \mathcal{M}(\tilde{f}^n_{i,j})=\frac{\tilde{\rho}^n_i}{\sqrt{2\pi \tilde{T}^n_i}}\exp\left(-\frac{|v_j-\tilde{U}^n_i|^2}{2\tilde{T}^n_i}\right),
        \end{split}
    \end{align*}
    where
    \begin{align*}
        \begin{split}
            (\tilde{\rho}^n_i,\tilde{\rho}^n_i\tilde{U}^n_i,\tilde{\rho}^n_i\tilde{T}^n_i)=\sum_j \tilde{f}^n_{i,j}(1,v_j,|v_j-\tilde{U}^n_i|^2)\Delta v.
        \end{split}
    \end{align*}
\end{remark}

\begin{lemma}\label{Mx to f}
    Assume that $f^{n}_{i,j}$ satisfies $\mathcal{A}^{n}$, and $\Delta v<r_{\Delta v}$. Then, we have
    \begin{align*}
        \begin{split}
            \|\mathcal{M}(\tilde{f}^{n})\|_{L^{\infty}_{q}} \leq C_{\mathcal{M}}\|\tilde{f}^{n}\|_{L^{\infty}_{q}},
        \end{split}
    \end{align*}
    where $C_{\mathcal{M}}$ depends only on $q$.
\end{lemma}
\begin{proof}
    We split $|v_{j}|^q\mathcal{M}(\tilde{f}^{n}_{i,j})$ into two parts as follows:
    \begin{align*}
        \begin{split}
            |v_{j}|^{q}\mathcal{M}(\tilde{f}^{n}_{i,j}) &\leq C_q(|\tilde{U}^{n}_{i}|^{q} + |v_{j}-\tilde{U}^{n}_{i}|^{q})\mathcal{M}(\tilde{f}^{n}_{i,j})\equiv I_1 + I_2.
        \end{split}
    \end{align*}
    For $I_1$, we have
    \begin{align*}
        \begin{split}
            I_1 &\leq C_q|\tilde{U}^{n}_{i}|^{q}\frac{\tilde{\rho}^{n}_{i}}{(2\pi \tilde{T}^{n}_{i})^{\frac{1}{2}}}.
        \end{split}
    \end{align*}
    If $|\tilde{U}^{n}_{i}|>(\tilde{T}^{n}_{i})^{\frac{1}{2}}$, Lemma \ref{Mac_bound_lemma}$_3$ gives
    \begin{align*}
        \begin{split}
            I_{1} \leq C_q\frac{|\tilde{U}^{n}_{i}|^{q}\tilde{\rho}^{n}_{i}}{(\tilde{T}^{n}_{i})^{\frac{1}{2}}} \leq C_q\frac{|\tilde{U}^{n}_{i}|^{q+1}\tilde{\rho}^{n}_{i}}{|\tilde{U}^{n}_{i}| (\tilde{T}^{n}_{i})^{\frac{1}{2}}} \leq C_q\frac{|\tilde{U}^{n}_{i}|^{q+1}\tilde{\rho}^{n}_{i}}{(\tilde{T}^{n}_{i} + |\tilde{U}^{n}_{i}|^{2})^{\frac{1}{2}} (\tilde{T}^{n}_{i})^{\frac{1}{2}}} \leq C_q\|\tilde{f}^{n}\|_{L^{\infty}_{q}}.
        \end{split}
    \end{align*}
    If $|\tilde{U}^{n}_{i}|\leq(\tilde{T}^{n}_{i})^{\frac{1}{2}}$, Lemma \ref{Mac_bound_lemma}$_2$ leads to
    \begin{align*}
        \begin{split}
            I_{1} \leq C_q\frac{(\tilde{T}^{n}_{i})^{\frac{q}{2}}\tilde{\rho}^{n}_{i}}{(\tilde{T}^{n}_{i})^{\frac{1}{2}}} \leq \tilde{\rho}^{n}_{i} (\tilde{T}^{n}_{i})^{\frac{q-1}{2}} \leq \tilde{\rho}^{n}_{i}(\tilde{T}^{n}_{i} + |\tilde{U}^{n}_{i}|^{2})^{\frac{q-1}{2}}\leq C_qC_M\|\tilde{f}^{n}\|_{L^{\infty}_{q}}.
        \end{split}
    \end{align*}
   On the other hand, estimate for $I_{2}$ is presented by
    \begin{align*}
        \begin{split}
            I_2 &\leq C_q|v_{j} - \tilde{U}^{n}_{i}|^{q}\frac{\tilde{\rho}^{n}_{i}}{(\tilde{T}^{n}_{i})^{\frac{1}{2}}}\exp\left( -\frac{|v_{j}-\tilde{U}^{n}_{i}|^{2}}{2\tilde{T}^{n}_{i}}\right)\cr
            &= C_q\frac{\tilde{\rho}^{n}_{i}}{(\tilde{T}^{n}_{i})^{\frac{1}{2}}}(\tilde{T}^{n}_{i})^{\frac{q}{2}} \left[\left( \frac{|v_{j}-\tilde{U}^{n}_{i}|^{2}}{2\tilde{T}^{n}_{i}}\right)^{\frac{q}{2}}\exp\left( -\frac{|v_{j}-\tilde{U}^{n}_{i}|^{2}}{2\tilde{T}^{n}_{i}}\right)\right]\cr
            &\leq C_q\tilde{\rho}^{n}_{i} (\tilde{T}^{n}_{i})^{\frac{q-1}{2}} \cr
            &\leq \tilde{\rho}^{n}_{i}(\tilde{T}^{n}_{i} + |\tilde{U}^{n}_{i}|^{2})^{\frac{q-1}{2}}\cr
            &\leq C_qC_M\|\tilde{f}^{n}\|_{L^{\infty}_{q}}.
        \end{split}
    \end{align*}
    In the last line, we used $x^{n}e^{-ax} \leq C$ for some $C>0$. Combining the estimates for $I_1$ and $I_2$, we have
    \begin{align*}
        \begin{split}
            |v_j|^q\mathcal{M}(\tilde{f}^n_{i,j})\leq C_qC_M\|\tilde{f}^{n}\|_{L^{\infty}_{q}}.
        \end{split}
    \end{align*}
    By Lemma \ref{Mac_bound_lemma}$_1$, we have
    \begin{align*}
        \begin{split}
            \mathcal{M}(\tilde{f}^{n}_{i,j}) \leq \frac{\tilde{\rho}^{n}_{i}}{(2\pi \tilde{T}^{n}_{i})^{\frac{1}{2}}}
            \leq C_M\|\tilde{f}^{n}\|_{L^{\infty}_{q}}.
        \end{split}
    \end{align*}
    Combining these, we obtain
    \begin{align*}
        \begin{split}
            \|\mathcal{M}(\tilde{f}^n)\|_{L^{\infty}_q}\leq C_{\mathcal{M}}\|\tilde{f}^n\|_{L^{\infty}_q},
        \end{split}
    \end{align*}
    for some $C_{\mathcal{M}}$ depending only on $q$.
\end{proof}

Next, we prove that the discrete solution $f^n_{i,j}$ is periodic in the spatial nodes.
\begin{lemma}\label{discrete periodic}
    Let $f^n_{i,j}$ satisfy \eqref{VPBGK_Scheme}. Then, we have
    \begin{align*}
        \begin{split}
            f^{n}_{i+N_{x},j} = f^{n}_{i,j} \quad \text{for all }n=0,1,\cdots,N_t.
        \end{split}
    \end{align*}
\end{lemma}
\begin{proof}
    By the definition of the $f^{0}_{i,j}$, for $i=0,1,\cdots,N_x-1$ and $j\in[-N_v,N_v]$, we have
    \begin{align*}
        \begin{split}
            f^{0}_{i+N_{x},j} = f_{0}(x_{i}+N_{x}\Delta x, v_{j}) = f_{0}(x_{i}+1,v_{j})=f_{0}(x_{i},v_{j}) = f^{0}_{i,j}.
        \end{split}
    \end{align*}
    Since the periodicity of the spatial domain $f^{n}_{i+N_{x},j} = f^{n}_{i,j}$. Then, we get
    \begin{align*}
        \begin{split}
            \hat{F}^{n}_{i+N_{x},j} = v_{j}^{+}f^{n}_{i+N_{x},j} + v_{j}^{-}f^{n}_{i+N_{x}+1,j} = v_{j}^{+}f^{n}_{i,j} + v_{j}^{-}f^{n}_{i+1,j} = \hat{F}^{n}_{i,j},
        \end{split}
    \end{align*}
    and
    \begin{align*}
        \begin{split}
            \hat{G}^{n}_{i+N_{x},j} = \mathbb{E}_{i+N_{x}}^{+}f^{n}_{i+N_{x},j} + \mathbb{E}_{i+N_{x}}^{-}f^{n}_{i+N_{x},j+1} = \mathbb{E}_{i}^{+}f^{n}_{i,j} + \mathbb{E}_{i}^{-}f^{n}_{i,j+1} = \hat{G}^{n}_{i,j}.
        \end{split}
    \end{align*}
    These lead to 
    \begin{align*}
        \begin{split}
            \tilde{f}^{n}_{i+N_{x},j} = \tilde{f}^{n}_{i,j}.
        \end{split}
    \end{align*}
   Using this, we have 
    \begin{align*}
        \begin{split}
            (\tilde{\rho}^{n}_{i+N_{x}},\tilde{\rho}^{n}_{i+N_{x}}\tilde{U}^{n}_{i+N_{x}}, \tilde{E}^{n+1}_{i+N_{x}}) &= \sum_{j}\left(1,v_j,\frac{|v_j|^2}{2} \right)\tilde{f}^{n}_{i+N_{x},j}\Delta v\cr
            &= \sum_{j}\left(1,v_j,\frac{|v_j|^2}{2} \right)\tilde{f}^{n}_{i,j}\Delta v\cr 
            &= (\tilde{\rho}^{n}_{i},\tilde{\rho}^{n}_{i}\tilde{U}^{n}_{i}, \tilde{E}^{n}_{i}).
        \end{split}
    \end{align*}
    The above periodicity of the macroscopic fields gives 
    \begin{align*}
        \begin{split}
            \mathcal{M}(\tilde{f}^{n}_{i+N_{x},j}) =  \mathcal{M}(\tilde{f}^{n}_{i,j}).
        \end{split}
    \end{align*}
    Therefore, from (\ref{VPBGK_Scheme}), we have
    \begin{align*}
        \begin{split}
            f^{n+1}_{i+N_{x},j} = \frac{\varepsilon \tilde{f}^{n}_{i+N_{x},j} + \Delta t\mathcal{M}(\tilde{f}^{n}_{i+N_{x},j})}{\varepsilon +\Delta t} = \frac{\varepsilon \tilde{f}^{n}_{i,j} + \Delta t\mathcal{M}(\tilde{f}^{n}_{i,j})}{\varepsilon +\Delta t} = f^{n+1}_{i,j}.
        \end{split}
    \end{align*}
    Then, induction gives the desired result.
\end{proof}

This lemma presents that $\tilde{f}^n_{i,j}$ can be controlled by $f^n_{i,j}$.
\begin{lemma}\label{tilde-original}
    Assume $\Delta x$, $\Delta v$, $\Delta t$, and $V_M$ satisfy the hypothesis of Theorem \ref{stability theorem} and $\mathbb{E}^n_i$ satisfies
    \begin{align*}
        \begin{split}
            \|\mathbb{E}^n\|_{L^{\infty}_x}\leq C_{\mathbb{E},1},\text{ for all }n.
        \end{split}
    \end{align*}
    Then, we have 
    \begin{align*}
        \begin{split}
            \|\tilde{f}^{n}\|_{L^{\infty}_{q}} \leq (1+C_{\mathbb{E},2}\Delta t)\|f^{n}\|_{L^{\infty}_{q}}.
        \end{split}
    \end{align*}
\end{lemma}
\begin{proof}
    By the symmetry formula of the update rule \eqref{convex_combi}, we only consider the case of $v_j\geq 0$ and $\mathbb{E}^n_i\geq 0$. In this case, $\tilde{f}^{n}_{i,j}$ is computed by
    \begin{align*}
        \begin{split}
            \tilde{f}^{n}_{i,j} = \left(1-\frac{\Delta t}{\Delta x}v_j-\frac{\Delta t}{\Delta v}\mathbb{E}^{n}_{i}\right)f^{n}_{i,j} + \frac{\Delta t}{\Delta x}v_j f^{n}_{i-1,j} +\frac{\Delta t}{\Delta v}\mathbb{E}^{n}_{i} f^{n}_{i,j-1}.
        \end{split}
    \end{align*}
    Multiplying $(1+|v_j|)^q$ to both sides, we have
    \begin{align}\label{tilde to orig}
        \begin{split}
            &(1+|v_j|)^{q}\tilde{f}^{n}_{i,j}\cr
            &\quad= (1+|v_j|)^{q}\left\{ \left(1-\frac{\Delta t}{\Delta x}v_j-\frac{\Delta t}{\Delta v}\mathbb{E}^{n}_{i}\right)f^{n}_{i,j} + \frac{\Delta t}{\Delta x}v_jf^{n}_{i-1,j} + \frac{\Delta t}{\Delta v}\mathbb{E}^{n}_{i}f^{n}_{i,j-1}\right\}\cr
            &\quad\leq \left(1-\frac{\Delta t}{\Delta x}v_j-\frac{\Delta t}{\Delta v}\mathbb{E}^{n}_{i}\right)\|f^n\|_{L^{\infty}_q} + \frac{\Delta t}{\Delta x}v_j\|f^n\|_{L^{\infty}_q} + \frac{\Delta t}{\Delta v}\mathbb{E}^{n}_{i}\|f^n\|_{L^{\infty}_q}\left(\frac{1+|v_j|}{1+|v_{j-1}|} \right)^q.
        \end{split}
    \end{align}    
    Since $\Delta v<r_{\Delta v}<1$, we get
    \begin{align*}
        \begin{split}
            \left(\frac{1+|v_j|}{1+|v_{j-1}|}\right)^{q}\leq \left(1 + \frac{\Delta v}{1+|v_{j-1}|}\right)^q \leq (1+\Delta v)^q \leq 1+(2^q-1)\Delta v.
        \end{split}
    \end{align*}
    Then, Lemma \ref{A2} gives
    \begin{align*}
        \begin{split}
            \frac{\Delta t}{\Delta v}\mathbb{E}^{n}_{i}\|f^n\|_{L^{\infty}_q}\left(\frac{1+|v_j|}{1+|v_{j-1}|} \right)^q&\leq \frac{\Delta t}{\Delta v}\mathbb{E}^{n}_{i}\|f^n\|_{L^{\infty}_q} + (2^q-1)\|\mathbb{E}^n\|_{L^{\infty}_x}\Delta t \|f^n\|_{L^{\infty}_q}\cr
            &\leq \frac{\Delta t}{\Delta v}\mathbb{E}^{n}_{i}\|f^n\|_{L^{\infty}_q} +  (2^q-1)C_{\mathbb{E},1}\Delta t\|f^n\|_{L^{\infty}_q}\cr
            &= \frac{\Delta t}{\Delta v}\mathbb{E}^{n}_{i}\|f^n\|_{L^{\infty}_q} +  C_{\mathbb{E},2}\Delta t\|f^n\|_{L^{\infty}_q}.
        \end{split}
    \end{align*}
    Plugging the above estimate into \eqref{tilde to orig}, we have
    \begin{align*}
        \begin{split}
            \|\tilde{f}^{n}\|_{L^{\infty}_q}&\leq \left(1-\frac{\Delta t}{\Delta x}v_j-\frac{\Delta t}{\Delta v}\mathbb{E}^{n}_{i}\right)\|f^n\|_{L^{\infty}_q} + \frac{\Delta t}{\Delta x}v_j\|f^n\|_{L^{\infty}_q} + \frac{\Delta t}{\Delta v}\mathbb{E}^{n}_{i}\|f^n\|_{L^{\infty}_q}+C_{\mathbb{E},2}\Delta t\|f^n\|_{L^{\infty}_q} \cr
            &=(1+C_{\mathbb{E},2}\Delta t)\|f^n\|_{L^{\infty}_q}. 
        \end{split}
    \end{align*}
\end{proof}

\subsection{Proof for Theorem \ref{stability theorem}}
Now, we show that $f^n_{i,j}$ satisfies the stability estimate $\mathcal{A}^n$ for all $n$ under the assumptions of Theorem \ref{stability theorem} using an inductive argument. The proof of Theorem \ref{stability theorem} follows from Lemma \ref{A1} - \ref{A6} below. 

\begin{lemma}\label{A1}\textnormal{(Proof of the stability estimate $A^n_1$)}
    Assume $\Delta x$, $\Delta v$, $\Delta t$ and $V_M$ satisfy the hypothesis of Theorem \ref{stability theorem}. Then, $f^{n}_{i,j}$ satisfies $A^n_1$ for all $n\geq 0$, that is,
    \begin{align*}
        \begin{split}
            \sum_{i,j}f^n_{i,j}\Delta x\Delta v\leq (2+T^f)e^{\frac{T^f}{\varepsilon}}.
        \end{split}
    \end{align*}
\end{lemma}
\begin{proof}
    To prove this lemma, we claim that
    \begin{align}\label{lem A1 inductive}
        \begin{split}
            \sum_{i,j}f^n_{i,j}\Delta x\Delta v\leq 2\left(\frac{\varepsilon + 2\Delta t}{\varepsilon+\Delta t} \right)^n + \left(\frac{\varepsilon + 2\Delta t}{\varepsilon+\Delta t} \right)n\Delta t,\quad\text{for all }n.
        \end{split}
    \end{align}
    Since $f^{0}_{i,j}=f_0(x_i,v_j)$ for $j\in[-N_v,N_v]$, we get
    \begin{align*}
        \begin{split}
            \sum_{i,j}f^0_{i,j}\Delta x\Delta v= \sum_{i,j}f_0(x_i,v_j)\Delta x\Delta v\leq 2.
        \end{split}
    \end{align*}
    We assume that $f^{m}_{i,j}$ satisfies the stability estimate $\mathcal{A}^m$ (see Definition \ref{stability condition}) for all $m=1,\cdots,n-1$, and \eqref{lem A1 inductive}. Multiplying $\Delta x\Delta v$ to our scheme \eqref{VPBGK_Scheme} and summing over $i,j\in[0,N_x-1]\times [-N_v,N_v]$, we have
    \begin{align}\label{lem A1 1}
        \begin{split}
            \sum_{i,j}f^{n}_{i,j}\Delta x\Delta v = \frac{\varepsilon}{\varepsilon+\Delta t}\sum_{i,j}\tilde{f}^{n-1}_{i,j}\Delta x\Delta v + \frac{\Delta t}{\varepsilon+\Delta t}\sum_{i,j}\mathcal{M}(\tilde{f}^{n-1}_{i,j})\Delta x\Delta v.
        \end{split}
    \end{align}
    Let $u^{n}_{i,j} = \frac{v_j  - \tilde{U}^{n}_i}{\sqrt{2\tilde{T}^{n}_{i}}}$ and $\Delta u^n_i$ is denoted by $\Delta u^n_i=u^n_{i,j}-u^n_{i,j-1}$. Then. we have
    \begin{align}\label{lem A1 delta w}
        \begin{split}
            \Delta u^{n-1}_{i}=\frac{v_j  - \tilde{U}^{n-1}_i}{\sqrt{2\tilde{T}^{n-1}_{i}}}- \frac{v_{j-1}  - \tilde{U}^{n-1}_i}{\sqrt{2\tilde{T}^{n-1}_{i}}}  =\frac{\Delta v}{\sqrt{2\tilde{T}^{n-1}_{i}}}.
        \end{split}
    \end{align}
    Since $A^{n-1}_5$ is valid by the inductive assumption, we have
    \begin{align*}
        \begin{split}
            \tilde{T}^{n-1}_i \geq \left(\frac{C_0 D_{\alpha,2\alpha C_{\mathbb{E},1}T^f}}{4(1+C_{\mathbb{E},2})C_M\|f_0\|_{L^{\infty}_q} }e^{-\left(2\alpha C_{\mathbb{E},1}+\frac{\alpha C^2_{\mathbb{E},1}(T^f+1)}{\eta_1}+C_{\mathbb{E},2}+\frac{C_{\mathcal{M}}}{\varepsilon } \right)T^f} \right)^{\frac{2}{3}}  =a_4.
        \end{split}
    \end{align*}
    Inserting this into \eqref{lem A1 delta w}, the smallness condition $\Delta v<r_{\Delta v}\leq l\sqrt{2a_4}$ gives 
    \begin{align*}
        \begin{split}
            \Delta u^{n-1}_{i}\leq \frac{r_{\Delta v}}{\sqrt{2a_4}}<l.
        \end{split}
    \end{align*}
    Thus, we obtain
    \begin{align*}
        \begin{split}
            \sum_j \mathcal{M}(\tilde{f}^{n-1}_{i,j})\Delta v=\frac{\tilde{\rho}^{n-1}_i}{\sqrt{2\pi \tilde{T}^{n-1}_i}}\sum_je^{-\frac{|v_j -\tilde{U}^{n-1}_i|^2}{2\tilde{T}^{n-1}_{i}}}\Delta v=\frac{\tilde{\rho}^{n-1}_i}{\sqrt{\pi}} \sum_j e^{-|u^{n-1}_{i,j}|^2}\Delta u^{n-1}_{i}\leq 2\tilde{\rho}^{n-1}_i,
        \end{split}
    \end{align*}
    where we used
    \begin{align*}
        \begin{split}
            \sum_j e^{-|u^{n-1}_{i,j}|^2}\Delta u^{n-1}_{i}\leq 2D_{2,0}=2\int_{\mathbb{R}}e^{-|v|^2}dv=2\sqrt{\pi}.
        \end{split}
    \end{align*}    
    This implies 
    \begin{align}\label{lem A1 2}
        \begin{split}
            \sum_{i,j} \mathcal{M}(\tilde{f}^{n-1}_{i,j})\Delta x\Delta v\leq 2\sum_i \tilde{\rho}^{n-1}_i\Delta x= 2\sum_{i,j}\tilde{f}^{n-1}_{i,j}\Delta x\Delta v.
        \end{split}
    \end{align}
    Inserting \eqref{lem A1 2} into \eqref{lem A1 1}, we get
    \begin{align}\label{lem A1 3}
        \begin{split}
            \sum_{i,j}f^n_{i,j}\Delta x\Delta v&\leq \left(\frac{\varepsilon + 2\Delta t}{\varepsilon+\Delta t}\right)\sum_{i,j}\tilde{f}^{n-1}_{i,j}\Delta x\Delta v.
        \end{split}
    \end{align}
    Recall the definition of $\tilde{f}^{n-1}_{i,j}$ from \eqref{tilde f}:
    \begin{align}\label{lem A1 tilde f}
        \begin{split}
            \tilde{f}^{n-1}_{i,j}&= {f}^{n-1}_{i,j}-\frac{\Delta t}{\Delta x}\left(\hat{F}^{n-1}_{i+\frac{1}{2},j}-\hat{F}^{n-1}_{i-\frac{1}{2},j} \right)-\frac{\Delta t}{\Delta v}\left(\hat{G}^{n-1}_{i,j+\frac{1}{2}}-\hat{G}^{n-1}_{i,j-\frac{1}{2}} \right).
        \end{split}
    \end{align}
    Multiplying $\Delta x\Delta v$ to \eqref{lem A1 tilde f}, and summing over $[0,N_x-1]\times [-N_v,N_v]$, we yield
    \begin{align}\label{lem A1 4}
        \begin{split}
            &\sum_{i,j}\tilde{f}^{n-1}_{i,j}\Delta x\Delta v\cr
            &\quad=\sum_{i,j}{f}^{n-1}_{i,j}\Delta x\Delta v\cr
            &\quad-\frac{\Delta t}{\Delta x}\sum_{i,j}\left(\hat{F}^{n-1}_{i+\frac{1}{2},j}-\hat{F}^{n-1}_{i-\frac{1}{2},j} \right)\Delta x\Delta v-\frac{\Delta t}{\Delta v}\sum_{i,j}\left(\hat{G}^{n-1}_{i,j+\frac{1}{2}}-\hat{G}^{n-1}_{i,j-\frac{1}{2}} \right)\Delta x\Delta v.
        \end{split}
    \end{align}
    By Lemma \ref{discrete periodic}, the second term vanishes by the telescoping summation:
    \begin{align}\label{lem A1 5}
        \begin{split}
            \sum_i \left(\hat{F}^{n-1}_{i+\frac{1}{2},j}-\hat{F}^{n-1}_{i-\frac{1}{2},j} \right)&=v^{+}_j\sum_i (f^{n-1}_{i,j}-f^{n-1}_{i-1,j})+v^{-}_j\sum_i (f^{n-1}_{i+1,j}-f^{n-1}_{i,j})\cr
            &= v^+_j(f^{n-1}_{N_x-1}-f^{n-1}_{-1,j})+v^-_j(f^{n-1}_{N_x,j}-f^{n-1}_{0,j})\cr
            &=0.
        \end{split}
    \end{align}
    To derive the third term, we apply the boundary condition \eqref{Neumann boundary condition} to get
    \begin{align*}
        \begin{split}
            \sum_{j}\left(\hat{G}^{n-1}_{i,j+\frac{1}{2}}-\hat{G}^{n-1}_{i,j-\frac{1}{2}} \right)\Delta v&={\mathbb{E}^{n-1}_i}^{+}\sum_j(f^{n-1}_{i,j}-f^{n-1}_{i,j-1})\Delta v+{\mathbb{E}^{n-1}_i}^{-}\sum_j(f^{n-1}_{i,j+1}-f^{n-1}_{i,j})\Delta v\cr
            &={\mathbb{E}^{n-1}_i}^{+}(f^{n-1}_{i,N_v}-f^{n-1}_{i,-N_v-1})\Delta v+{\mathbb{E}^{n-1}_i}^{-}(f^{n-1}_{i,N_v+1}-f^{n-1}_{i,-N_v})\Delta v\cr
            &={\mathbb{E}^{n-1}_i}^{+}(f^{n-1}_{i,N_v}-f^{n-1}_{i,-N_v})\Delta v+{\mathbb{E}^{n-1}_i}^{-}(f^{n-1}_{i,N_v}-f^{n-1}_{i,-N_v})\Delta v\cr
            &=|\mathbb{E}^{n-1}_i|(f^{n-1}_{i,N_v}-f^{n-1}_{i,-N_v})\Delta v.
        \end{split}
    \end{align*}
    Then, the truncation assumption \eqref{V_M large} for $V_M$ leads to
    \begin{align}\label{lem A1 6}
        \begin{split}
            -\sum_{j}\left(\hat{G}^{n-1}_{i,j+\frac{1}{2}}-\hat{G}^{n-1}_{i,j-\frac{1}{2}} \right)\Delta v&= |\mathbb{E}^{n-1}_i|(f^{n-1}_{i,-N_v}-f^{n-1}_{i,N_v})\Delta v\cr
            &\leq |\mathbb{E}^{n-1}_i|\frac{(1+V_M)^q(|f^{n-1}_{i,-N_v}|+|f^{n-1}_{i,N_v}|)\Delta v}{V_M^q}\cr
            &\leq \frac{2|\mathbb{E}^{n-1}_i|\|f^{n-1}\|_{L^{\infty}_q}\Delta v}{V_M^q}\cr
            &\leq 2C^{-q}_{\mathbb{E},1}|\mathbb{E}^{n-1}_i|\|f^{n-1}\|_{L^{\infty}_q}(\Delta v)^{q\gamma +1}.
        \end{split}
    \end{align} 
    From inductive assumptions $A^{n-1}_2$, we have
    \begin{align}\label{lem A1 inductiv field}
        \begin{split}
            |\mathbb{E}^{n-1}_i|\leq C_{\mathbb{E},1}.
        \end{split}
    \end{align}
    By inductive assumption $A^{n-2}_4$ and Lemma \ref{Mx to f}, we get
    \begin{align}\label{lem A1 inductiv f}
        \begin{split}
           \|f^{n-1}\|_{L^{\infty}_q}&\leq \frac{\varepsilon}{\varepsilon+\Delta t}\|\tilde{f}^{n-2}\|_{L^{\infty}_q}+\frac{\Delta t}{\varepsilon+\Delta t}\|\mathcal{M}(\tilde{f}^{n-2})\|_{L^{\infty}_q}\cr
            &\leq \frac{\varepsilon+C_{\mathcal{M}}\Delta t}{\varepsilon+\Delta t}\|\tilde{f}^{n-2}\|_{L^{\infty}_q}\cr
            &\leq (1+C_{\mathbb{E},2})\left(\frac{\varepsilon+C_{\mathcal{M}}\Delta t}{\varepsilon+\Delta t}\right)e^{\left(C_{\mathbb{E},2}+\frac{C_{\mathcal{M}}-1}{\varepsilon} \right)T^f}\|f_0\|_{L^{\infty}_q}\cr
            &\leq (1+C_{\mathbb{E},2})e^{\left(C_{\mathbb{E},2}+\frac{2(C_{\mathcal{M}}-1)}{\varepsilon} \right)T^f}\|f_0\|_{L^{\infty}_q},
        \end{split}
    \end{align}
    where we used
    \begin{align*}
        \begin{split}
            \frac{\varepsilon+C_{\mathcal{M}}\Delta t}{\varepsilon+\Delta t}\leq 1+ \frac{(C_{\mathcal{M}}-1)\Delta t}{\varepsilon+\Delta t}\leq e^{\frac{(C_{\mathcal{M}}-1)\Delta t}{\varepsilon+\Delta t}}\leq e^{\frac{(C_{\mathcal{M}}-1)T^f}{\varepsilon}}.
        \end{split}
    \end{align*}
    Plugging \eqref{lem A1 inductiv field} and \eqref{lem A1 inductiv f} into \eqref{lem A1 6}, we obtain
    \begin{align*}
        \begin{split}
            -\sum_{j}\left(\hat{G}^{n-1}_{i,j+\frac{1}{2}}-\hat{G}^{n-1}_{i,j-\frac{1}{2}} \right)\Delta v\leq 2 C^{1-q}_{\mathbb{E},1}(1+C_{\mathbb{E},2})e^{\left(C_{\mathbb{E},2}+\frac{2(C_{\mathcal{M}}-1)}{\varepsilon} \right)T^f}\|f_0\|_{L^{\infty}_q}(\Delta v)^{q\gamma+1}.
        \end{split}
    \end{align*}
    Then, the smallness condition $\Delta v<r_{\Delta v}\leq a_5$ gives 
    \begin{align}\label{lem A1 7}
        \begin{split}
            &-\frac{\Delta t}{\Delta v}\sum_{i,j}\left( \hat{G}^{n-1}_{i,j+\frac{1}{2}} - \hat{G}^{n-1}_{i,j-\frac{1}{2}}\right)\Delta x\Delta v\cr
            &\qquad\leq 2 C^{1-q}_{\mathbb{E},1}(1+C_{\mathbb{E},2})e^{\left(C_{\mathbb{E},2}+\frac{2(C_{\mathcal{M}}-1)}{\varepsilon} \right)T^f}\|f_0\|_{L^{\infty}_q}(\Delta v)^{q\gamma}\Delta t\sum_i\Delta x\cr
            &\qquad\leq \Delta t\sum_i\Delta x\cr
            &\qquad=\Delta t.
        \end{split}
    \end{align}
    Inserting \eqref{lem A1 5} and \eqref{lem A1 7} into \eqref{lem A1 4}, we get
    \begin{align}\label{lem A1 8}
        \begin{split}
            \sum_{i,j}\tilde{f}^{n-1}_{i,j}\Delta x\Delta v\leq \sum_{i,j}f^{n-1}_{i,j}\Delta x\Delta v+\Delta t.
        \end{split}
    \end{align}
    This implies that \eqref{lem A1 3} gives
    \begin{align*}
        \begin{split}
            \sum_{i,j}f^n_{i,j}\Delta x\Delta v\leq \left(\frac{\varepsilon+2\Delta t}{\varepsilon+\Delta t} \right)\sum_{i,j}f^{n-1}_{i,j}\Delta x\Delta v+ \left(\frac{\varepsilon+2\Delta t}{\varepsilon+\Delta t} \right)\Delta t.
        \end{split}
    \end{align*}
    Finally, the inductive assumption on $f^{n-1}_{i,j}$ leads to
    \begin{align*}
        \begin{split}
            \sum_{i,j}f^{n}_{i,j}\Delta x\Delta v\leq 2\left(\frac{\varepsilon + 2\Delta t}{\varepsilon+\Delta t} \right)^n + \left(\frac{\varepsilon + 2\Delta t}{\varepsilon+\Delta t} \right)n\Delta t.
        \end{split}
    \end{align*}
    Therefore, the claim \eqref{lem A1 inductive} is proved by induction. Moreover, using $(1+x)^n\leq e^{nx}$ and $n\Delta t\leq N_t \Delta t\leq T^f$, we have 
    \begin{align*}
        \begin{split}
            2\left(\frac{\varepsilon + 2\Delta t}{\varepsilon+\Delta t} \right)^n+\left(\frac{\varepsilon + 2\Delta t}{\varepsilon+\Delta t} \right)n\Delta t
            &\leq (2+n\Delta t)\left(\frac{\varepsilon + 2\Delta t}{\varepsilon+\Delta t} \right)^n\cr
            &=(2+n\Delta t)\left(1+\frac{\Delta t}{\varepsilon+\Delta t} \right)^n\cr
            &\leq (2+n\Delta t)e^{\frac{n\Delta t}{\varepsilon+\Delta t}}\cr
            &\leq (2+T^f)e^{\frac{T^f}{\varepsilon}}.
        \end{split}
    \end{align*}
    This gives the desired result.
\end{proof}

\begin{lemma}\label{A2}\textnormal{(Proof of the stability estimate $A^n_2$)}
    Assume $\Delta x$, $\Delta v$, $\Delta t$, and $V_M$ satisfy the hypothesis of Theorem \ref{stability theorem}. Then, $f^{n}_{i,j}$ satisfies $A^n_2$ for all $n\geq 0$, that is,
    \begin{align*}
        \begin{split}
            \|\mathbb{E}^n\|_{L^{\infty}_x}\leq C_{\mathbb{E},1}.
        \end{split}
    \end{align*}
\end{lemma}
\begin{proof}
    According to the definition of the discrete electric field \eqref{Discrete field 2}, for $k\in[0,N_x-1]$ and $j\in[-N_v,N_v]$, we have
    \begin{align*}
        \begin{split}
            |\mathbb{E}^n_i|\leq \sum_{k,j} |K(x_i,y_k)|f^n_{k,j}\Delta y\Delta v + \sum_k |K(x_i,y_k)|\Delta y.
        \end{split}
    \end{align*}
    Since $\underset{i,k}{\sup}|K(x_i,y_k)|\leq 1$ by the definition of the Green kernel \eqref{Green kernel} and $y_k$ shares the same mesh size with $x_i$, Lemma \ref{A1} gives
    \begin{align*}
        \begin{split}
            \|\mathbb{E}^n\|_{L^{\infty}_x}\leq \sum_{k,j} f^n_{k,j}\Delta y\Delta v + 1\leq (2+T^f)e^{\frac{T^f}{\varepsilon}}+1 = C_{\mathbb{E},1}.
        \end{split}
    \end{align*}
    This completes the proof.
\end{proof}

\begin{lemma}\label{A3}\textnormal{(Proof of the stability estimate $A^n_3$)}
    Assume $\Delta x$, $\Delta v$, $\Delta t$, and $V_M$ satisfy the hypothesis of Theorem \ref{stability theorem}. Then, $f^{n}_{i,j}$ satisfies $A^n_3$ for all $n\geq 0$, that is,
    \begin{align*}
        \begin{split}
            \tilde{f}^{n}_{i,j}\geq \frac{C_0}{2}e^{-\left(2\alpha C_{\mathbb{E},1}+\frac{\alpha C^2_{\mathbb{E},1}(T^f+1)}{\eta_1}+\frac{1}{\varepsilon }\right)T^f}e^{-|v_j|^\alpha}e^{-2\alpha C_{\mathbb{E},1}T^f|v_j|}.
        \end{split}
    \end{align*}
\end{lemma}
\begin{proof}
    This proof only focuses on the case $v_j>0$. The analysis for $v_j\leq 0$ is omitted, as it proceeds analogously due to the symmetry of the update rule \eqref{convex_combi}. We divide our analysis into two cases, whether the velocity of the particle $v_j$ is accelerated or decelerated by the electric field $\mathbb{E}^n_i$. The first aim of this proof is to establish the following estimate: 
    \begin{align}\label{lem A3 inductive}
        \begin{split}
            \tilde{f}^n_{i,j} \geq C_{0}\left\{ \prod^{n+1}_{k=1}(1-\alpha C_{\mathbb{E},1}\Delta t(|v_j|+k\Delta v+1))\right\}\left(\frac{\varepsilon}{\varepsilon+\Delta t} \right)^{n}e^{-|v_j|^{\alpha}},\quad\text{for all }n.
        \end{split}
    \end{align}
\textbf{(1-1) $\boldsymbol{0<v_j< V_M}$ and $\boldsymbol{\mathbb{E}^n_i<0}$ for all $\boldsymbol{n}$.}\\
First, we consider the case in which the particle is decelerated by the electric field. This is the case where the mixing of velocity grid indices has a critical impact on the stability estimate. From the update rule \eqref{convex_combi}, we have 
\begin{align}\label{lem A3 bad update}
    \begin{split}
        \tilde{f}^n_{i,j}=\left(1-\frac{\Delta t}{\Delta x}|v_j|-\frac{\Delta t}{\Delta v}|\mathbb{E}^n_i| \right)f^n_{i,j}+\frac{\Delta t}{\Delta x}|v_j|f^n_{i-1,j}+\frac{\Delta t}{\Delta v}|\mathbb{E}^n_i|f^n_{i,j+1},\quad \text{for all }n.
    \end{split}
\end{align}
For the initial step, we compute $\tilde{f}^0_{i,j}$ using \eqref{lem A3 bad update}:
    \begin{align*}
        \begin{split}
            \tilde{f}^0_{i,j}&=\left(1-\frac{\Delta t}{\Delta x}|v_j|-\frac{\Delta t}{\Delta v}|\mathbb{E}^0_i| \right)f^{0}_{i,j} + \frac{\Delta t}{\Delta x}|v_j|f^0_{i-1,j} + \frac{\Delta t}{\Delta v}|\mathbb{E}^0_i|f^0_{i,j+1}\cr
            &\geq C_0\left\{\left(1-\frac{\Delta t}{\Delta x}|v_j|-\frac{\Delta t}{\Delta v}|\mathbb{E}^0_i| \right)e^{-|v_j|^{\alpha}}+ \frac{\Delta t}{\Delta x}|v_j|e^{-|v_j|^{\alpha}}+ \frac{\Delta t}{\Delta v}|\mathbb{E}^0_i|e^{-|v_{j+1}|^{\alpha}}\right\}\cr
            &\geq C_0\left\{e^{-|v_j|^{\alpha}} - \frac{\Delta t}{\Delta v}|\mathbb{E}^0_i|\left(e^{-|v_j|^{\alpha}}-e^{-|v_{j+1}|^{\alpha}}\right)\right\}.
        \end{split}
    \end{align*}
    Since $\alpha\in[1,2]$, for $v_j> 0$ and $\theta\in(0,1)$, we apply the mean value theorem to obtain
    \begin{align}\label{lem A3 2}
        \begin{split}
            e^{-|v_j|^{\alpha}}-e^{-|v_{j+1}|^{\alpha}}&\leq \alpha\Delta v(v_j+\theta\Delta v)^{\alpha-1} e^{-|v_j+\theta\Delta v|^{\alpha}}\cr
            &\leq \alpha\Delta v(|v_j|+\Delta v+1) e^{-|v_j|^{\alpha}}.
        \end{split}
    \end{align}
    Applying this, Lemma \ref{A2} leads to
    \begin{align*}
        \begin{split}
            \tilde{f}^0_{i,j} &\geq C_0\left\{e^{-|v_j|^{\alpha}} - \frac{\Delta t}{\Delta v}|\mathbb{E}^0_i|\left(e^{-|v_j|^{\alpha}}-e^{-|v_{j+1}|^{\alpha}}\right)\right\}\cr
            &\geq C_0(1-\alpha |\mathbb{E}^0_i|(|v_j|+\Delta v+1)\Delta t)e^{-|v_j|^{\alpha}}\cr
            &\geq C_0(1-\alpha C_{\mathbb{E},1}(|v_j|+\Delta v+1)\Delta t)e^{-|v_j|^{\alpha}}.
        \end{split}
    \end{align*}
    Now we assume that the inductive hypothesis \eqref{lem A3 inductive} holds for the $(n-1)$-th step. Then, from our scheme \eqref{VPBGK_Scheme}, we have
    \begin{align}\label{lem A3 1}
        \begin{split}
            f^n_{i,j}\geq \frac{\varepsilon}{\varepsilon+\Delta t} \tilde{f}^{n-1}_{i,j} \geq C_{0}\left\{ \prod^{n}_{k=1}(1-\alpha C_{\mathbb{E},1}\Delta t(|v_j|+k\Delta v+1))\right\}\left(\frac{\varepsilon}{\varepsilon+\Delta t} \right)^{n}e^{-|v_j|^{\alpha}}.
        \end{split}
    \end{align}
    Plugging this into the update rule \eqref{lem A3 bad update}, we obtain
    \begin{align*}
        \begin{split}
            \tilde{f}^{n}_{i,j}&= \left(1-\frac{\Delta t}{\Delta x}|v_j|-\frac{\Delta t}{\Delta v}|\mathbb{E}^n_i| \right)f^{n}_{i,j} + \frac{\Delta t}{\Delta x}|v_j|f^n_{i-1,j} + \frac{\Delta t}{\Delta v}|\mathbb{E}^n_i|f^n_{i,j+1}\cr
            &\geq C_{0}\left\{ \prod^{n}_{k=1}(1-\alpha C_{\mathbb{E},1}\Delta t(|v_j|+k\Delta v+1))\right\}\left(\frac{\varepsilon}{\varepsilon+\Delta t} \right)^n\left(1-\frac{\Delta t}{\Delta x}|v_j|-\frac{\Delta t}{\Delta v}|\mathbb{E}^n_i| \right)e^{-|v_j|^{\alpha}}\cr
            &+ C_{0}\left\{ \prod^{n}_{k=1}(1-\alpha C_{\mathbb{E},1}\Delta t(|v_j|+k\Delta v+1))\right\}\left(\frac{\varepsilon}{\varepsilon+\Delta t} \right)^n\frac{\Delta t}{\Delta x}|v_j|e^{-|v_j|^{\alpha}}\cr
            &+ C_{0}\left\{ \prod^{n}_{k=1}(1-\alpha C_{\mathbb{E},1}\Delta t(|v_{j+1}|+k\Delta v+1))\right\}\left(\frac{\varepsilon}{\varepsilon+\Delta t} \right)^n\frac{\Delta t}{\Delta v}|\mathbb{E}^n_i|e^{-|v_{j+1}|^{\alpha}}.
        \end{split}
    \end{align*}
    Since $|v_{j+1}|>|v_j|$ for positive $v_j$, the terms of the partial product on the right-hand sides are bounded below such that
    \begin{align*}
        \begin{split}
            \left\{ \prod^{n}_{k=1}(1-\alpha C_{\mathbb{E},1}\Delta t(|v_j|+k\Delta v+1))\right\}\geq \left\{ \prod^{n}_{k=1}(1-\alpha C_{\mathbb{E},1}\Delta t(|v_{j+1}|+k\Delta v+1))\right\}.
        \end{split}
    \end{align*}
    Using this, we have
    \begin{align*}
        \begin{split}
            \tilde{f}^n_{i,j}&\geq C_0\left\{ \prod^{n}_{k=1}(1-\alpha C_{\mathbb{E},1}\Delta t(|v_{j+1}|+k\Delta v+1))\right\}\left(\frac{\varepsilon}{\varepsilon+\Delta t} \right)^n\cr
            &\qquad\times\left(1-\frac{\Delta t}{\Delta x}|v_j|-\frac{\Delta t}{\Delta v}|\mathbb{E}^n_i| \right)e^{-|v_j|^\alpha}+\frac{\Delta t}{\Delta x}|v_j|e^{-|v_j|^\alpha} + \frac{\Delta t}{\Delta v}|\mathbb{E}^n_i|e^{-|v_{j+1}|^\alpha}\cr
            &=C_{0}\left\{ \prod^{n}_{k=1}(1-\alpha C_{\mathbb{E},1}\Delta t(|v_{j+1}|+k\Delta v+1))\right\}\left(\frac{\varepsilon}{\varepsilon+\Delta t} \right)^n\cr
            &\qquad\times\left\{e^{-|v_j|^{\alpha}} - \frac{\Delta t}{\Delta v}|\mathbb{E}^n_i|\left(e^{-|v_j|^{\alpha}}-e^{-|v_{j+1}|^{\alpha}} \right)\right\}\cr
            &=C_{0}\left\{ \prod^{n+1}_{k=2}(1-\alpha C_{\mathbb{E},1}\Delta t(|v_{j}|+k\Delta v+1))\right\}\left(\frac{\varepsilon}{\varepsilon+\Delta t} \right)^n\cr
            &\qquad\times\left\{e^{-|v_j|^{\alpha}} - \frac{\Delta t}{\Delta v}|\mathbb{E}^n_i|\left(e^{-|v_j|^{\alpha}}-e^{-|v_{j+1}|^{\alpha}} \right)\right\},
        \end{split}
    \end{align*}
     where we used $|v_{j+1}|+k\Delta v=|v_j|+(k+1)\Delta v$ in the last line. Using \eqref{lem A3 2} and Lemma \ref{A2}, we have
    \begin{align*}
        \begin{split}
            e^{-|v_{j}|^\alpha}-\frac{\Delta t}{\Delta v}|\mathbb{E}^{n}_i|(e^{-|v_j|^\alpha}-e^{-|v_{j+1}|^\alpha}) \geq (1-\alpha C_{\mathbb{E},1}\Delta t(|v_j|+\Delta v+1))e^{-|v_j|^2}.
        \end{split}
    \end{align*}
    This implies
    \begin{align*}
        \begin{split}
            \tilde{f}^{n}_{i,j}\geq C_{0}\left\{ \prod^{n+1}_{k=1}(1-\alpha C_{\mathbb{E},1}\Delta t(|v_{j}|+k\Delta v+1))\right\}\left(\frac{\varepsilon}{\varepsilon+\Delta t} \right)^n e^{-|v_j|^\alpha}.
        \end{split}
    \end{align*}
    Then, induction gives the desired result. We note that the positivity of $(1-\alpha C_{\mathbb{E},1}\Delta t(|v_{j}|+k\Delta v+1))$ follows from   
    \begin{align}\label{positivity of the rate}
        \begin{split}
            1-\alpha C_{\mathbb{E},1}|v_{j}|\Delta t-\alpha C_{\mathbb{E},1}k\Delta t\Delta v-\alpha C_{\mathbb{E},1}\Delta t&\geq 1-\alpha C_{\mathbb{E},1}\Delta x-\alpha C_{\mathbb{E},1}(T^f+1)\Delta v -\alpha \Delta v\cr 
            &\geq 1-\frac{1}{4}-\frac{1}{4}\cr
            &= \frac{1}{2}.
        \end{split}
    \end{align}
    In the first line, we applied the CFL condition and the definition of $T^f$: 
    \begin{align*}
        \begin{split}
            |v_j|\Delta t<\Delta x,\quad C_{\mathbb{E},1}\Delta t<\Delta v,\quad k\Delta t\leq (N_t+1)\Delta t\leq T^f+\Delta t\leq T^f+1.
        \end{split}
    \end{align*}
    In the second line, we used the smallness condition for $\Delta x$ and $\Delta v$ in the hypothesis of Theorem \ref{stability theorem} as follows:
    \begin{align*}
        \begin{split}
            \Delta x<\frac{1}{4\alpha C_{\mathbb{E},1}},\quad \Delta v<\frac{1}{4\alpha\{ C_{\mathbb{E},1}(T^f+1)+1\}}.
        \end{split}
    \end{align*}
   \noindent\textbf{(1-2) $\boldsymbol{v_j=V_M}$ and $\boldsymbol{\mathbb{E}^n_i<0}$ for all $\boldsymbol{n}$.}\\
   The update of $\tilde{f}^n_{i,j}$ at the boundary necessitates data from outside the truncated velocity domain. Thus, we additionally consider the estimate on the boundary. Assume that \eqref{lem A3 inductive} holds for the $(n-1)$-th step. Then, $f^n_{i,j}$ satisfies \eqref{lem A3 1}. The Neumann boundary condition \eqref{Neumann boundary condition} in $v$ direction implies
    \begin{align*}
        \begin{split}
            \tilde{f}^{n}_{i,N_v} &= \left(1-\frac{\Delta t}{\Delta x}V_M-\frac{\Delta t}{\Delta v}|\mathbb{E}^n_i| \right)f^{n}_{i,N_v} + \frac{\Delta t}{\Delta x}V_Mf^n_{i-1,N_v} + \frac{\Delta t}{\Delta v}|\mathbb{E}^n_i|f^n_{i,N_v+1}\cr
            &= \left(1-\frac{\Delta t}{\Delta x}V_M-\frac{\Delta t}{\Delta v}|\mathbb{E}^n_i| \right)f^{n}_{i,N_v} + \frac{\Delta t}{\Delta x}V_Mf^n_{i-1,N_v} + \frac{\Delta t}{\Delta v}|\mathbb{E}^n_i|f^n_{i,N_v}\cr
            &\geq C_0\left\{\prod^{n}_{k=1}(1-\alpha C_{\mathbb{E},1}\Delta t(V_M+\Delta v + 1))\right\}\left(\frac{\varepsilon}{\varepsilon+\Delta t} \right)^n e^{-V_M^{\alpha}}\cr
             &\geq  C_0\left\{\prod^{n+1}_{k=1}(1-\alpha C_{\mathbb{E},1}\Delta t(V_M+\Delta v + 1))\right\}\left(\frac{\varepsilon}{\varepsilon+\Delta t} \right)^n e^{-V_M^{\alpha}}.
        \end{split}
    \end{align*}
    Therefore, induction completes the proof on the boundary.\\\\
    \noindent\textbf{(2) $\boldsymbol{0<v_j\leq  V_M}$ and $\boldsymbol{\mathbb{E}^n_i\geq 0}$ for all $\boldsymbol{n}$.}\\
Second, we consider the case in which the particle is accelerated by the electric field. Although index mixing induced by the electric field is still present in this case, its impact on the analysis is not critical. Therefore, the argument from the non-ionized problem \cite{boscarino2022convergence,russo2012convergence,russo2018convergence} can be largely adapted. From the update rule \eqref{convex_combi}, we have
\begin{align}\label{lem A3 good update}
    \begin{split}
        \tilde{f}^n_{i,j}=\left(1-\frac{\Delta t}{\Delta x}|v_j|-\frac{\Delta t}{\Delta v}|\mathbb{E}^n_i| \right)f^n_{i,j}+\frac{\Delta t}{\Delta x}|v_j|f^n_{i-1,j}+\frac{\Delta t}{\Delta v}|\mathbb{E}^n_i|f^n_{i,j-1},\quad \text{for all }n.
    \end{split}
\end{align}
Since $v_j$ is positive, we have
    \begin{align*}
        \begin{split}
            e^{-|v_{j-1}|^{\alpha}}=e^{-|v_j-\Delta v|^{\alpha}}\geq e^{-|v_j|^{\alpha}}.
        \end{split}
    \end{align*}
    Using this with the update rule \eqref{lem A3 good update}, we present the estimate for the initial step:
    \begin{align*}
        \begin{split}
            \tilde{f}^0_{i,j}&=\left(1-\frac{\Delta t}{\Delta x}|v_j|-\frac{\Delta t}{\Delta v}|\mathbb{E}^0_i| \right)f^{0}_{i,j} + \frac{\Delta t}{\Delta x}|v_j|f^0_{i-1,j} + \frac{\Delta t}{\Delta v}|\mathbb{E}^0_i|f^0_{i,j-1}\cr
            &\geq C_{0}\left\{\left(1-\frac{\Delta t}{\Delta x}|v_j|-\frac{\Delta t}{\Delta v}|\mathbb{E}^0_i| \right)e^{-|v_j|^{\alpha}}+ \frac{\Delta t}{\Delta x}|v_j|e^{-|v_j|^{\alpha}}+ \frac{\Delta t}{\Delta v}|\mathbb{E}^0_i|e^{-|v_{j-1}|^{\alpha}}\right\}\cr
            &\geq C_{0}\left\{\left(1-\frac{\Delta t}{\Delta x}|v_j|-\frac{\Delta t}{\Delta v}|\mathbb{E}^0_i| \right)e^{-|v_j|^{\alpha}}+ \frac{\Delta t}{\Delta x}|v_j|e^{-|v_j|^{\alpha}}+ \frac{\Delta t}{\Delta v}|\mathbb{E}^0_i|e^{-|v_{j}|^{\alpha}}\right\}\cr
            &= C_0e^{-|v_j|^{\alpha}}\cr
            &\geq C_0(1-{\alpha}C_{\mathbb{E},1}(|v_j|+\Delta v+1)\Delta t)e^{-|v_j|^{\alpha}}.
        \end{split}
    \end{align*}
    Now we assume that the inductive hypothesis \eqref{lem A3 inductive} holds for the $(n-1)$-th step. Then, $f^n_{i,j}$ satisfies \eqref{lem A3 1} by our scheme \eqref{VPBGK_Scheme}. Plugging this into the update rule \eqref{lem A3 good update}, we have
    \begin{align*}
        \begin{split}
            \tilde{f}^{n}_{i,j} &= \left(1-\frac{\Delta t}{\Delta x}|v_j|-\frac{\Delta t}{\Delta v}|\mathbb{E}^n_i| \right)f^{n}_{i,j} + \frac{\Delta t}{\Delta x}|v_j|f^n_{i-1,j} + \frac{\Delta t}{\Delta v}|\mathbb{E}^n_i|f^n_{i,j-1}\cr
            &\geq C_{0}\left\{ \prod^{n}_{k=1}(1-\alpha C_{\mathbb{E},1}\Delta t(|v_j|+k\Delta v+1))\right\}\left(\frac{\varepsilon}{\varepsilon+\Delta t} \right)^{n}\left(1-\frac{\Delta t}{\Delta x}|v_j|-\frac{\Delta t}{\Delta v}|\mathbb{E}^n_i| \right)e^{-|v_j|^{\alpha}}\cr
            &+ C_{0}\left\{ \prod^{n}_{k=1}(1-\alpha C_{\mathbb{E},1}\Delta t(|v_j|+k\Delta v+1))\right\}\left(\frac{\varepsilon}{\varepsilon+\Delta t} \right)^n\frac{\Delta t}{\Delta x}|v_j|e^{-|v_j|^{\alpha}}\cr
            &+ C_0\left\{ \prod^{n}_{k=1}(1-\alpha C_{\mathbb{E},1}\Delta t(|v_{j-1}|+k\Delta v+1))\right\}\left(\frac{\varepsilon}{\varepsilon+\Delta t} \right)^{n}\frac{\Delta t}{\Delta v}|\mathbb{E}^n_i|e^{-|v_{j-1}|^{\alpha}}.
        \end{split}
    \end{align*}
    Since $e^{-|v_{j-1}|^\alpha}\geq e^{-|v_j|^\alpha}$ for positive $v_j$, the last term can be controlled by
    \begin{align*}
        \begin{split}
            &C_0\left\{ \prod^{n}_{k=1}(1-\alpha C_{\mathbb{E},1}\Delta t(|v_{j-1}|+k\Delta v+1))\right\}\left(\frac{\varepsilon}{\varepsilon+\Delta t} \right)^{n}\frac{\Delta t}{\Delta v}|\mathbb{E}^n_i|e^{-|v_{j-1}|^{\alpha}}\cr
            &\quad\geq C_0\left\{ \prod^{n}_{k=1}(1-\alpha C_{\mathbb{E},1}\Delta t(|v_{j}|+k\Delta v+1))\right\}\left(\frac{\varepsilon}{\varepsilon+\Delta t} \right)^{n}\frac{\Delta t}{\Delta v}|\mathbb{E}^n_i|e^{-|v_{j}|^{\alpha}}.
        \end{split}
    \end{align*}
    Using this, we obtain
    \begin{align*}
        \begin{split}
            \tilde{f}^n_{i,j}&\geq C_{0}\left\{ \prod^{n}_{k=1}(1-\alpha C_{\mathbb{E},1}\Delta t(|v_j|+k\Delta v+1))\right\}\left(\frac{\varepsilon}{\varepsilon+\Delta t} \right)^{n}\cr
            &\qquad\times \left(1-\frac{\Delta t}{\Delta x}|v_j|-\frac{\Delta t}{\Delta v}|\mathbb{E}^n_i| \right)e^{-|v_j|^\alpha}+\frac{\Delta t}{\Delta x}|v_j|e^{-|v_j|^\alpha} + \frac{\Delta t}{\Delta v}|\mathbb{E}^n_i|e^{-|v_{j}|^\alpha}\cr
            &= C_{0}\left\{ \prod^{n}_{k=1}(1-\alpha C_{\mathbb{E},1}\Delta t(|v_j|+k\Delta v+1))\right\}\left(\frac{\varepsilon}{\varepsilon+\Delta t} \right)^{n}e^{-|v_j|^{\alpha}}\cr
            &\geq C_{0}\left\{ \prod^{n+1}_{k=1}(1-\alpha C_{\mathbb{E},1}\Delta t(|v_j|+k\Delta v+1))\right\}\left(\frac{\varepsilon}{\varepsilon+\Delta t} \right)^{n}e^{-|v_j|^{\alpha}}.
        \end{split}
    \end{align*}
    This completes the proof of \eqref{lem A3 inductive}.\\\\
    \indent To show the uniform boundedness, we use $(1+x)^{-n}\geq e^{-nx}$ and $n\Delta t\leq T^f$ to get
    \begin{align*}
        \begin{split}
            \left( \frac{\varepsilon}{\varepsilon+\Delta t}\right)^n = \left(1+\frac{\Delta t}{\varepsilon} \right)^{-n}\geq e^{-\frac{n\Delta t}{\varepsilon}}\geq e^{-\frac{T^f}{\varepsilon}}.
        \end{split}
    \end{align*}
    Moreover, using $(1+x)^{-1}\geq e^{-x}$ and \eqref{positivity of the rate}, we have
    \begin{align*}
        \begin{split}
            (1-\alpha C_{\mathbb{E},1}\Delta t(|v_j|+k\Delta v))&=\left(1+\frac{\alpha C_{\mathbb{E},1}\Delta t(|v_j|+k\Delta v+1)}{1-\alpha C_{\mathbb{E},1}\Delta t(|v_j|+k\Delta v+1)} \right)^{-1}\cr
            &\geq \left(1+2\alpha C_{\mathbb{E},1}\Delta t(|v_j|+k\Delta v+1) \right)^{-1}\cr
            &\geq e^{-2\alpha C_{\mathbb{E},1}\Delta t(|v_j|+k\Delta v +1)}.
        \end{split}
    \end{align*}
    Thus, we yield
    \begin{align*}
        \begin{split}
            \prod^{n+1}_{k=1}(1-\alpha C_{\mathbb{E},1}\Delta t(|v_j|+k\Delta v+1))&\geq \frac{1}{2}\prod^{n}_{k=1}(1-\alpha C_{\mathbb{E},1}\Delta t(|v_j|+k\Delta v+1))\cr
            &= \frac{1}{2}e^{-2\alpha C_{\mathbb{E},1}n\Delta t(|v_j|+1)}e^{-2\alpha C_{\mathbb{E},1}\Delta t\sum^{n}_{k=1} k\Delta v}\cr
            &\geq \frac{1}{2}e^{-2\alpha C_{\mathbb{E},1}T^f}e^{-2\alpha C_{\mathbb{E},1}T^f|v_j|}e^{-\alpha C_{\mathbb{E},1}n(n+1)\Delta t\Delta v}.
        \end{split}
    \end{align*}
    From the hypothesis of this lemma, $\Delta t$ satisfies
    \begin{align*}
        \begin{split}
            \Delta v<\frac{C_{\mathbb{E},1}\Delta t}{\eta_1}.
        \end{split}
    \end{align*}
    This gives
    \begin{align*}
        \begin{split}
            e^{-\alpha C_{\mathbb{E},1}n(n+1)\Delta t\Delta v}=e^{-\frac{\alpha C^2_{\mathbb{E},1}n(n+1)(\Delta t)^2}{\eta_1}}> e^{-\frac{\alpha C^2_{\mathbb{E},1}({T^f}^2+T^f)}{\eta_1}}.
        \end{split}
    \end{align*}
    Consequently, we obtain
    \begin{align*}
        \begin{split}
            \tilde{f}^{n}_{i,j}&\geq C_0\prod^{n+1}_{k=1}(1-\alpha C_{\mathbb{E},1}\Delta t(|v_j|+k\Delta v+1))\left(\frac{\varepsilon}{\varepsilon+\Delta t} \right)^ne^{-|v_j|^\alpha}\cr
            &\geq \frac{C_0}{2}e^{-\left(2\alpha C_{\mathbb{E},1}+\frac{\alpha C^2_{\mathbb{E},1}(T^f+1)}{\eta_1}+\frac{1}{\varepsilon }\right)T^f}e^{-|v_j|^\alpha}e^{-2\alpha C_{\mathbb{E},1}T^f|v_j|}.
        \end{split}
    \end{align*}
\end{proof}

Now we apply Lemma \ref{tilde-original} to show that $f^n_{i,j}$ satisfies $A^n_4$.

\begin{lemma}\label{A4}\textnormal{(Proof of the stability estimate $A^n_4$)}
        Assume $\Delta x$, $\Delta v$, $\Delta t$, and $V_M$ satisfy the hypothesis of Theorem \ref{stability theorem}. Then, $f^{n}_{i,j}$ satisfies $A^n_4$ for all $n\geq 0$, that is,
    \begin{align*}
        \begin{split}
            \|\tilde{f}^{n}\|_{L^{\infty}_{q}} \leq (1+C_{\mathbb{E},2})e^{\left(C_{\mathbb{E},2} + \frac{C_{\mathcal{M}}-1}{\varepsilon}\right)T^{f}}\|f_ {0}\|_{L^{\infty}_{q}}.
        \end{split}
    \end{align*}
\end{lemma}
\begin{proof}
    We claim that $f^n_{i,j}$ satisfies
    \begin{align}\label{lem A4 inductive}
        \begin{split}
            \|\tilde{f}^{n}\|_{L^{\infty}_{q}} \leq (1+C_{\mathbb{E},2}\Delta t)^{n+1}\left(\frac{\varepsilon + C_{\mathcal{M}}\Delta t}{\varepsilon + \Delta t}\right)^{n}\|f^{0}\|_{L^{\infty}_{q}},\text{ for all }n.
        \end{split}
    \end{align}
    For $n=0$, since $\Delta t<1$ and $|\mathbb{E}^0_i|\leq C_{\mathbb{E},1}$ Lemma \ref{tilde-original} gives
    \begin{align*}
        \begin{split}
            \|\tilde{f}^0\|_{L^{\infty}_{q}} \leq (1+C_{\mathbb{E},2}\Delta t)\|f^0\|_{L^{\infty}_{q}}\leq (1+C_{\mathbb{E},2})\|f^0\|_{L^{\infty}_{q}}.
        \end{split}
    \end{align*}
    Assume that \eqref{lem A4 inductive} is valid for the $(n-1)$-th step. Our scheme \eqref{VPBGK_Scheme} gives
    \begin{align*}
        \begin{split}
            \|\tilde{f}^{n}\|_{L^{\infty}_{q}} &\leq (1+C_{\mathbb{E},2}\Delta t)\|f^n\|_{L^{\infty}_{q}} \cr
            &\leq (1+C_{\mathbb{E},2}\Delta t)\left(\frac{\varepsilon}{\varepsilon+\Delta t}\|\tilde{f}^{n-1}\|_{L^{\infty}_{q}} + \frac{\Delta t}{\varepsilon+\Delta t}\|\mathcal{M}(\tilde{f}^{n-1})\|_{L^{\infty}_{q}}\right).
        \end{split}
    \end{align*}
    By Lemma \ref{Mx to f} and Lemma \ref{tilde-original}, we obtain
    \begin{align*}
        \begin{split}
            \|\tilde{f}^{n}\|_{L^{\infty}_{q}} &\leq (1+C_{\mathbb{E},2}\Delta t)\left(\frac{\varepsilon}{\varepsilon+\Delta t}\|\tilde{f}^{n-1}\|_{L^{\infty}_{q}} + \frac{C_{\mathcal{M}}\Delta t}{\varepsilon+\Delta t}\|\tilde{f}^{n-1}\|_{L^{\infty}_{q}}\right)\cr
            &\leq (1+C_{\mathbb{E},2}\Delta t)\left(\frac{\varepsilon + C_{\mathcal{M}}\Delta t}{\varepsilon + \Delta t}\right)\|\tilde{f}^{n-1}\|_{L^{\infty}_{q}}.
        \end{split}
    \end{align*}
    Then, the inductive assumption for $\tilde{f}^{n-1}_{i,j}$ leads to
    \begin{align*}
        \begin{split}
            \|\tilde{f}^{n}\|_{L^{\infty}_{q}}\leq (1+C_{\mathbb{E},2}\Delta t)^{n+1}\left(\frac{\varepsilon + C_{\mathcal{M}}\Delta t}{\varepsilon + \Delta t}\right)^{n}\|f^{0}\|_{L^{\infty}_{q}}.
        \end{split}
    \end{align*}
    This gives that \eqref{lem A4 inductive} holds for all $n$ by induction. To complete the proof of this lemma, we used $(1+x)^n \leq e^{nx}$ and $n\Delta t \leq N_t \Delta t = T^f$ to get
    \begin{align*}
        \begin{split}
            \left(\frac{\varepsilon + C_{\mathcal{M}}\Delta t}{\varepsilon + \Delta t} \right)^{n} \leq \left(1+\frac{ (C_{\mathcal{M}}-1)\Delta t}{\varepsilon + \Delta t} \right)^{n} \leq e^{\frac{(C_{\mathcal{M}}-1)n\Delta t}{\varepsilon + \Delta t}} \leq e^{\frac{(C_{\mathcal{M}}-1) T^{f}}{\varepsilon }},
        \end{split}
    \end{align*}
    and
    \begin{align*}
        \begin{split}
            (1+C_{\mathbb{E},2}\Delta t)^{n+1} = (1+C_{\mathbb{E},2}\Delta t)(1+C_{\mathbb{E},2}\Delta t)^{n} \leq (1+C_{\mathbb{E},2})e^{C_{\mathbb{E},2}n\Delta t} \leq (1+C_{\mathbb{E},2})e^{C_{\mathbb{E},2}T^f}.
        \end{split}
    \end{align*}
\end{proof}

The uniform lower and upper bounds for $\tilde{\rho}^n_i$, $\tilde{U}^n_i$, and $\tilde{T}^n_i$ follow from Lemma \ref{A3} and Lemma \ref{A4}:
\begin{lemma}\label{A5}\textnormal{(Proof of the stability estimate $A^n_5$)}
        Assume $\Delta x$, $\Delta v$, $\Delta t$, and $V_M$ satisfy the hypothesis of Theorem \ref{stability theorem}. Then, $f^{n}_{i,j}$ satisfies $A^n_5$ for all $n\geq 0$, that is,
    \begin{align*}
        \begin{split}
            \tilde{\rho}^n_i &\geq \frac{C_0 D_{\alpha,2\alpha C_{\mathbb{E},1}T^f}}{4}e^{-\left(2\alpha C_{\mathbb{E},1}+\frac{\alpha C^2_{\mathbb{E},1}(T^f+1)}{\eta_1}+\frac{1}{\varepsilon }\right)T^f},\cr
            \tilde{T}^n_i&\geq  \left(\frac{C_0 D_{\alpha,2\alpha C_{\mathbb{E},1}T^f}}{4(1+C_{\mathbb{E},2})C_M\|f_0\|_{L^{\infty}_q} }e^{-\left(2\alpha C_{\mathbb{E},1}+\frac{\alpha C^2_{\mathbb{E},1}(T^f+1)}{\eta_1}+C_{\mathbb{E},2}+\frac{C_{\mathcal{M}}}{\varepsilon } \right)T^f} \right)^{\frac{2}{3}}.
        \end{split}
    \end{align*}
\end{lemma}
\begin{proof}
    To derive the estimate for the lower bound of $\tilde{\rho}^n_i$, we multiply $\Delta v$ to $\tilde{f}^n_{i,j}$, and sum over $-N_v\leq j\leq N_v$. Then, Lemma \ref{A3} gives
    \begin{align*}
        \begin{split}
            \tilde{\rho}^n_i&=\sum_{j}\tilde{f}^n_{i,j}\Delta v\cr 
            &\geq C_0\prod^{n+1}_{k=1}(1-\alpha C_{\mathbb{E},1}\Delta t(|v_j|+k\Delta v+1))\left(\frac{\varepsilon}{\varepsilon+\Delta t} \right)^ne^{-|v_j|^\alpha}\cr
            &\geq \frac{C_0}{2}e^{-\left(2\alpha C_{\mathbb{E},1}+\frac{\alpha C^2_{\mathbb{E},1}(T^f+1)}{\eta_1}+\frac{1}{\varepsilon }\right)T^f}\sum_j e^{-|v_j|^{\alpha}-2\alpha C_{\mathbb{E},1}T^f|v_j|}\Delta v\cr
            &\geq\frac{C_0 D_{\alpha,2\alpha C_{\mathbb{E},1}T^f}}{4}e^{-\left(2\alpha C_{\mathbb{E},1}+\frac{\alpha C^2_{\mathbb{E},1}(T^f+1)}{\eta_1}+\frac{1}{\varepsilon }\right)T^f}.
        \end{split}
    \end{align*}
    On the other hand, Lemma \ref{Mac_bound_lemma} and Lemma \ref{A4} lead to
    \begin{align*}
        \begin{split}
            \tilde{T}^n_i \geq \left( \frac{\tilde{\rho}^n_i}{C_M \|\tilde{f}^n\|_{L^{\infty}_q}}\right)^{\frac{2}{3}}>\left(\frac{C_0 D_{\alpha,2\alpha C_{\mathbb{E},1}T^f}}{4(1+C_{\mathbb{E},2})C_M\|f_0\|_{L^{\infty}_q} }e^{-\left(2\alpha C_{\mathbb{E},1}+\frac{\alpha C^2_{\mathbb{E},1}(T^f+1)}{\eta_1}+C_{\mathbb{E},2}+\frac{C_{\mathcal{M}}}{\varepsilon } \right)T^f} \right)^{\frac{2}{3}}.
        \end{split}
    \end{align*}
\end{proof}

\begin{lemma}\label{A6}\textnormal{(Proof of the stability estimate $A^n_6$)}
        Assume $\Delta x$, $\Delta v$, $\Delta t$, and $V_M$ satisfy the hypothesis of Theorem \ref{stability theorem}. Then, $f^{n}_{i,j}$ satisfies $A^n_6$ for all $n\geq 0$, that is,
    \begin{align*}
        \begin{split}
            &\|\tilde{\rho}^n\|_{L^{\infty}_x},\,\|\tilde{U}^n\|_{L^{\infty}_x},\,\|\tilde{T}^n\|_{L^{\infty}_x}\cr
            &\qquad\leq \bar{D}_{q-2}\left(1+\frac{8}{C_0  D_{\alpha,2\alpha C_{\mathbb{E},1}T^f}} \right)(1+C_{\mathbb{E},2})e^{\left(2\alpha C_{\mathbb{E},1}+\frac{\alpha C^2_{\mathbb{E},1}(T^f+1)}{\eta_1}+C_{\mathbb{E},2} + \frac{C_{\mathcal{M}}}{\varepsilon}\right)T^{f}}\|f_{0}\|_{L^{\infty}_{q}}.
        \end{split}
    \end{align*}
\end{lemma}
\begin{proof}
    By Lemma \ref{A4}, we have
    \begin{align*}
        \begin{split}
            \tilde{\rho}^n_i\leq \sum_j \tilde{f}^n_{i,j}\Delta v \leq \|\tilde{f}^n\|_{L^{\infty}_q}\sum_j\frac{\Delta v}{(1+|v_j|)^q}\leq 2\bar{D}_{q} (1+C_{\mathbb{E},2})e^{\left(C_{\mathbb{E},2} + \frac{C_{\mathcal{M}}-1}{\varepsilon}\right)T^{f}}\|f_{0}\|_{L^{\infty}_{q}}.       
        \end{split}
    \end{align*}
    To derive the estimate for $\tilde{U}^n_i$, we use Lemma \ref{A4} and Lemma \ref{A5}:
    \begin{align*}
        \begin{split}
            |\tilde{U}^n_i| &\leq \frac{1}{\tilde{\rho}^n_i}\sum_{j}\tilde{f}^n_{i,j}|v_j|\Delta v\cr
            &\leq \frac{1}{\tilde{\rho}^n_i}\|\tilde{f}^n\|_{L^{\infty}_q}\sum_j\frac{\Delta v}{(1+|v_j|)^{q-1}}\cr
            &\leq \frac{8\bar{D}_{q-1}}{C_0  D_{\alpha,2\alpha C_{\mathbb{E},1}T^f}}(1+C_{\mathbb{E},2})e^{\left(2\alpha C_{\mathbb{E},1}+\frac{\alpha C^2_{\mathbb{E},1}(T^f+1)}{\eta_1}+C_{\mathbb{E},2} + \frac{C_{\mathcal{M}}}{\varepsilon}\right)T^{f}}\|f_{0}\|_{L^{\infty}_{q}}.
        \end{split}
    \end{align*}
    Lastly, estimate for $\tilde{T}^n_i$ follows from
    \begin{align*}
        \begin{split}
            \tilde{T}^n_i&=  \frac{1}{\tilde{\rho}^n_i}\sum_{j}\tilde{f}^n_{i,j}|v_j|^2\Delta v-|\tilde{U}^n_i|\cr
            &\leq \frac{1}{\tilde{\rho}^n_i}\sum_{j}\tilde{f}^n_{i,j}|v_j|^2\Delta v\cr
            &\leq \frac{1}{\tilde{\rho}^n_i}\|\tilde{f}^n\|_{L^{\infty}_q}\sum_j\frac{\Delta v}{(1+|v_j|)^{q-2}}\cr
            &\leq \frac{8\bar{D}_{q-2}}{C_0  D_{\alpha,2\alpha C_{\mathbb{E},1}T^f}}(1+C_{\mathbb{E},2})e^{\left(2\alpha C_{\mathbb{E},1}+\frac{\alpha C^2_{\mathbb{E},1}(T^f+1)}{\eta_1}+C_{\mathbb{E},2} + \frac{C_{\mathcal{M}}}{\varepsilon}\right)T^{f}}\|f_{0}\|_{L^{\infty}_{q}}. 
        \end{split}
    \end{align*}
\end{proof}
   
\section{Convergence to the Vlasov-Poisson-BGK model}\label{sec5}
\subsection{Consistent form}
In this subsection, we transform the VPBGK model (\ref{VPBGK model}) into a form consistent with our proposed scheme (\ref{VPBGK_Scheme}).
\begin{proposition}\label{consistent theorem}\textnormal{(Consistent form)}
Under the assumption of Theorem \ref{analysis}, the VPBGK model (\ref{VPBGK model}) can be represented in the following form:
\begin{align*}
    \begin{split}
        f(x,v,t+\Delta t)&= \frac{\varepsilon}{\varepsilon +\Delta t}\tilde{f}(x,v,t) + \frac{\Delta t}{\varepsilon+\Delta t}\mathcal{M}(\tilde{f})(x,v,t)+ \frac{1}{\varepsilon+ \Delta t}(R_1 + R_2 + R_3 ),
    \end{split}
\end{align*}
where $\tilde{f}(x,v,t)$ is defined by
\begin{align*}
    \begin{split}
        \tilde{f}(x,v,t) = f(x-v\Delta t, v-\mathbb{E}(x,t)\Delta t,t).
    \end{split}
\end{align*}
The remainder terms $R_1$, $R_2$, and $R_3$ are defined as follows:
\begin{align}\label{consistent_remainder}
    \begin{split}
        &R_1 = \varepsilon \Delta tv(\partial_x f(\bar{x},v,t) - \partial_x f(x,v,t))+\varepsilon \Delta t\mathbb{E}(x,t)(\partial_v f(x-v\Delta t,\bar{v},t) - \partial_v f(x,v,t)), \cr
        &R_2 = \frac{\varepsilon (\Delta t)^2}{2}\partial^2_t f(x,v,\bar{t}),\cr
        &R_3 = \Delta t(f(x,v,t+\Delta t) - f(x,v,t))+\Delta t\left(\mathcal{M}(f)(x,v,t) - \mathcal{M}(\tilde{f})(x,v,t)\right),
    \end{split}
\end{align}
for some $(\bar{x},\bar{v},\bar{t})$ which satisfy
\begin{align*}
    \begin{split}
        |\bar{x} - x|<|v|\Delta t,\quad|\bar{v} - v|<|\mathbb{E}(x,t)|\Delta t,\quad|\bar{t} - t|<\Delta t.
    \end{split}
\end{align*}
\end{proposition}

\begin{proof}
   By Taylor's theorem, we have
   \begin{align}\label{Cons 1}
       \begin{split}
           &f(x,v,t+\Delta t)\cr
           &\quad= f(x,v,t) +\Delta t\partial_t f(x,v,t) +\frac{(\Delta t)^2}{2}\partial^2_t f(x,v,\bar{t})\cr
           &\quad= f(x-v\Delta t,v,t) +v\Delta t\partial_x f(\bar{x},v,t)+\Delta t\partial_t f(x,v,t) +\frac{(\Delta t)^2}{2}\partial^2_t f(x,v,\bar{t})\cr
           &\quad= f(x-v\Delta t,v-\mathbb{E}(x,t)\Delta t,t) + v\Delta t\partial_xf(\bar{x},v,t) +\mathbb{E}(x,t)\Delta t\partial_v f(x-v\Delta t,\bar{v},t)\cr
           &\quad+\Delta t\partial_t f(x,v,t) +\frac{(\Delta t)^2}{2}\partial^2_t f(x,v,\bar{t}).
       \end{split}
   \end{align}
    The VPBGK model \eqref{VPBGK model} gives
    \begin{align}\label{Cons 2}
       \begin{split}
           \partial_t f(x,v,t) = -v\partial_x f(x,v,t) - \mathbb{E}(x,t)\partial_v f(x,v,t) + \frac{1}{\varepsilon}(\mathcal{M}(f)(x,v,t) - f(x,v,t)).
       \end{split}
   \end{align}
   Inserting \eqref{Cons 2} into \eqref{Cons 1}, we get
   \begin{align*}
       \begin{split}
           &f(x,v,t+\Delta t)\cr
           &\quad= f(x-v\Delta t,v-\mathbb{E}(x,t)\Delta t,t) + \frac{\Delta t}{\varepsilon}(\mathcal{M}(f)(x,v,t) - f(x,v,t))\cr
           &\quad+v\Delta t(\partial_x f(\bar{x},v,t) - \partial_x f(x,v,t))+ \mathbb{E}(x,t)\Delta t(\partial_v f(x-v\Delta t,\bar{v},t) - \partial_v f(x,v,t))\cr 
           &\quad+ \frac{(\Delta t)^2}{2}\partial^2_t f(x,v,\bar{t}).
       \end{split}
   \end{align*}
   Multiplying $\displaystyle \frac{\varepsilon}{\varepsilon+\Delta t}$ and adding $\displaystyle \frac{\Delta t}{\varepsilon+\Delta t}f(x,v,t+\Delta t)$ to both sides, we rearrange the terms to obtain
   \begin{align*}
       \begin{split}
           &f(x,v,t+\Delta t)\cr
           &\quad= \frac{\varepsilon}{\varepsilon +\Delta t}f(x-v\Delta t,v-\mathbb{E}(x,t)\Delta t,t) + \frac{\Delta t}{\varepsilon+\Delta t}\mathcal{M}(f)(x-v\Delta t,v-\mathbb{E}\Delta t,t)\cr
           &\quad+\frac{\varepsilon\Delta t}{\varepsilon+\Delta t}v(\partial_x f(\bar{x},v,t) - \partial_x f(x,v,t))+\frac{\varepsilon\Delta t}{\varepsilon+\Delta t}\mathbb{E}(x,t)(\partial_v f(x-v\Delta t,\bar{v},t) - \partial_v f(x,v,t))\cr
           &\quad+\frac{\varepsilon}{\varepsilon+\Delta t}\frac{(\Delta t)^2}{2}\partial^2_t f(x,v,\bar{t})\cr
           &\quad+\frac{\Delta t}{\varepsilon + \Delta t}\left\{(f(x,v,t+\Delta t) - f(x,v,t))+(\mathcal{M}(f)(x,v,t) - \mathcal{M}(f)(x-v\Delta t,v-\mathbb{E}(x,t)\Delta t,t))\right\}.
       \end{split}
   \end{align*}
   This completes the proof.
\end{proof}

To estimate the remainder terms, we present a lemma to control the continuous local Maxwellian.

\begin{lemma}\label{LipMaxw}\cite{YunPark2026}
    Let $f$ be a solution in Theorem \ref{analysis} corresponding to initial data $f_0$. Then, we have, for $q>3$,
    \begin{align*}
        \begin{split}
            \|\mathcal{M}(f)\|_{W^{1,\infty}_{q}} \leq C_{q,f_0,T^f}(\|f\|_{W^{1,\infty}_{q}} +1),\quad\|\mathcal{M}(f) - \mathcal{M}(g)\|_{L^{\infty}_{q}} \leq C_{q,f_0,T^f} \|f-g\|_{L^{\infty}_{q}}.
        \end{split}
    \end{align*}
\end{lemma}

We are now ready to derive the estimates for $R_1$, $R_2$, and $R_3$.
\begin{lemma}\label{R_1 estimate}
    Let $R_1$ is defined in \eqref{consistent_remainder}. Then, we have
    \begin{align*}
        \begin{split}
            \|R_1\|_{L^{\infty}_{q}} \leq \varepsilon C_{q,f_0,T^f}(\Delta t)^2.
        \end{split}
    \end{align*}
\end{lemma}
\begin{proof}
    By the mean value theorem, for some $x_1 \in [\bar{x},x]$, we have
    \begin{align*}
        \begin{split}
            |\varepsilon \Delta t v(\partial_x f(\bar{x},v,t)-\partial_x f(x,v,t))|\leq \varepsilon\Delta t|v||\bar{x}-x||\partial^2_x f(x_1,v,t)|.
        \end{split}
    \end{align*}            
Since $|\bar{x}-x|<|v|\Delta t$, the first term of $R_1$ is derived by         
         \begin{align}\label{R_1 1}
        \begin{split}
             |\varepsilon \Delta t v\{\partial_x f(\bar{x},v,t)-\partial_x f(x,v,t) \}|&\leq \varepsilon (\Delta t)^2|v|^2|\partial^2_x f(x_1,v,t)|\cr
            &\leq \varepsilon(\Delta t)^2 \|f(t)\|_{W^{2,\infty}_{q+2}}\frac{1}{(1+|v|)^q}\cr
            &\leq\varepsilon C_{q,f_0,T^f}(\Delta t)^2 \frac{1}{(1+|v|)^q}.
        \end{split}
    \end{align}
    On the other hand, we split the second term of $R_1$ into two parts such that
    \begin{align*}
        \begin{split}
            &|\varepsilon\Delta t\mathbb{E}(x,t)(\partial_v f(x-v\Delta t,\bar{v},t)-\partial_v f(x,v,t))|\cr
            &\quad \leq \varepsilon\Delta t\mathbb{E}(x,t)|\partial_v f(x-v\Delta t,\bar{v},t)-\partial_v f(x,\bar{v},t)|\cr
            &\quad+ \varepsilon\Delta t\mathbb{E}(x,t)|\partial_v f(x,\bar{v},t)-\partial_v f(x,v,t)|\cr
            &\quad \equiv I_1 + I_2.    
        \end{split}
    \end{align*}
    By the mean value theorem, for some $x_2\in [x-v\Delta t,x]$, $I_1$ is controlled by
        \begin{align*}
        \begin{split}
            I_1 &\leq \varepsilon(\Delta t)^2|\mathbb{E}(x,t)||v||\partial_x\partial_v f(x_2,\bar{v},t)|\cr
            &\leq \varepsilon(\Delta t)^2\|\mathbb{E}(t)\|_{L^{\infty}_x} \|f(t)\|_{W^{2,\infty}_{q+1}} \frac{|v|}{(1+|\bar{v}|)^{q+1}}\cr
            &\leq C_{q,f_0,T^f}\varepsilon(\Delta t)^2 \left(\frac{1+|v|}{1+|\bar{v}|}\right)^{q+1} \frac{1}{(1+|v|)^q}\cr
            &\leq C_{q,f_0,T^f}\varepsilon(\Delta t)^2 \frac{1}{(1+|v|)^q},
        \end{split}
    \end{align*}
    where we used
    \begin{align*}
        \begin{split}
            \left(\frac{1+|v|}{1+|\bar{v}|}\right)^{q+1}\leq \left(1+ \frac{|v|-|\bar{v}|}{1+|\bar{v}|} \right)^{q+1}\leq \left(1+|v-\bar{v}| \right)^{q+1}\leq C_{q}(1+|\mathbb{E}(x,t)|\Delta t)^{q+1}\leq C_{q,f_0,T^f}.
        \end{split}
    \end{align*}
    We now consider the estimate for $I_2$ as follows:
    \begin{align*}
        \begin{split}
            I_2 &\leq \varepsilon\Delta t |\mathbb{E}(x,t)||\partial_v f(x,\bar{v},t) -\partial_v f(x,v,t)|\cr
            &\leq \varepsilon\Delta t |\mathbb{E}(x,t)||\bar{v}-v||\partial^2_v f(x,v_{1},t)|\cr
            &\leq \varepsilon|\mathbb{E}(x,t)|^2\|f(t)\|_{W^{2,\infty}_{q}}(\Delta t)^2  \frac{1}{(1+|v_{1}|)^q}\cr
            &\leq C_{q,f_0,T^f}\varepsilon (\Delta t)^2 \left(\frac{1+|v|}{1+|v_{1}|}\right)^q \frac{1}{(1+|v|)^q}\cr
            &\leq \varepsilon C_{q,f_0,T^f}(\Delta t)^2\frac{1}{(1+|v|)^q},
        \end{split}
    \end{align*}
    for some $v_1\in[\bar{v},v]$. Combining the estimates for $I_1$ and $I_2$, we get
    \begin{align}\label{R_1 2}
        \begin{split}
            |\varepsilon\Delta t\mathbb{E}(x,t)(\partial_v f(x-v\Delta t,\bar{v},t)-\partial_v f(x,v,t) )|\leq \varepsilon C_{q,f_0,T^f} (\Delta t)^2\frac{1}{(1+|v|)^q}.
        \end{split}
    \end{align}
    By \eqref{R_1 1} and \eqref{R_1 2}, we obtain the desired result.
\end{proof}

\begin{lemma}\label{R_2 estimate}
     Let $R_{2}$ is defined in \eqref{consistent_remainder}. Then, we have
     \begin{align*}
         \begin{split}
             \|R_2\|_{L^{\infty}_{q}} \leq \varepsilon C_{q,f_0,T^f,\varepsilon}(\Delta t)^{2}.
         \end{split}
     \end{align*}
\end{lemma}
\begin{proof}
    From the VPBGK model \eqref{VPBGK model}, we get
    \begin{align*}
        \begin{split}
            |\partial^2_t f(x,v,t)| &\leq |v||\partial_x\partial_t f(t,x,v)| + |\partial_t \mathbb{E}(x,t)||\partial_v f(t,x,v)| + |\mathbb{E}(x,t)||\partial_v\partial_t f(x,v,t)|\cr
            &+ \frac{1}{\varepsilon}(|\partial_t \mathcal{M}(f)(x,v,t)| + |\partial_t f(x,v,t)|).
        \end{split}
    \end{align*}
    By Lemma \ref{TimeDeri_f}-\ref{TimeDeri_field}, we have
    \begin{align*}
        \|\partial^2_t f(t)\|_{L^{\infty}_q} \leq C_{q,f_0,T^f,\varepsilon}. 
    \end{align*}
    Thus, we yield
    \begin{align*}
        \begin{split}
            \|R_2\|_{L^{\infty}_{q}}\leq \frac{\varepsilon}{2}(\Delta t)^2\|\partial^2_t f(t)\|_{L^{\infty}_q}\leq \varepsilon C_{q,f_0,T^f,\varepsilon}(\Delta t)^2.
        \end{split}
    \end{align*}
\end{proof}

\begin{lemma}\label{R_3 estimate}
    Let $R_{3}$ is defined in \eqref{consistent_remainder}. Then, we have
    \begin{align*}
        \begin{split}
           \|R_3\|_{L^{\infty}_{q}} \leq C_{q,f_0,T^f,\varepsilon}(\Delta t)^2 .
        \end{split}
    \end{align*}
\end{lemma}

\begin{proof}
    We split the estimate for $R_3$ into the two parts:
    \begin{align*}
        \begin{split}
            (1+|v|^q)R_3 &\leq  (1+|v|^q)\Delta t|f(x,v,t+\Delta t)-f(x,v,t)|+(1+|v|^q)\Delta t|(\mathcal{M}(f)-\mathcal{M}(\tilde{f}))(x,v,t)|\cr
            &\equiv I_1+I_2.
        \end{split}
    \end{align*}
    \textbf{(1) Estimate for $I_1$}\\
    From the VPBGK model \eqref{VPBGK model}, Lemma \ref{LipMaxw} gives
    \begin{align*}
        \begin{split}
            |\partial_t f(x,v,t)|&\leq |v||\partial_x f(x,v,t)| + |\mathbb{E}(x,t)||\partial_v f(x,v,t)|+\frac{1}{\varepsilon}\left(|\mathcal{M}(f)(x,v,t)| + |f(x,v,t)| \right)\cr
            &\leq\left\{\|f(t)\|_{W^{1,\infty}_{q+1}}+C_{q,f_0,T^f}\|f(t)\|_{W^{1,\infty}_{q}}+\frac{1}{\varepsilon}(\|\mathcal{M}(f)(t)\|_{L^{\infty}_q}+\|f(t)\|_{L^{\infty}_q})\right\}\frac{1}{(1+|v|)^q}\cr
            &\leq \left\{\|f(t)\|_{W^{1,\infty}_{q+1}}+C_{q,f_0,T^f}\|f(t)\|_{W^{1,\infty}_{q}}+\frac{1}{\varepsilon}(C_{q,f_0,T^f}+1)\|f(t)\|_{L^{\infty}_q}\right\}\frac{1}{(1+|v|)^q}\cr
            &\leq C_{q,f_0,T^f,\varepsilon}\frac{1}{(1+|v|)^q}.            
        \end{split}
    \end{align*}
    Then, for some $t_1\in(t,t+\Delta t)$, Taylor's theorem leads to
    \begin{align*}
        \begin{split}
            I_1&\leq (\Delta t)^2(1+|v|^q)|\partial_tf(x,v,t_1)|\leq (\Delta t)^2\|\partial_t f(t_1)\|_{L^{\infty}_q}\leq C_{q,f_0,T^f,\varepsilon}(\Delta t)^2.
        \end{split}
    \end{align*}
    \textbf{(2) Estimate for $I_2$}\\
    From the definition of $\tilde{f}$, for some $\bar{x}$ and $\bar{v}$ which satisfy $|\bar{x}-x|<|v|\Delta t$ and $|\bar{v}-v|<|\mathbb{E}(x,t)|\Delta t$, respectively, we apply Taylor's theorem to obtain
        \begin{align*}
        \begin{split}
            &|f(x,v,t) - \tilde{f}(x,v,t)|\cr
            &\quad\leq |f(x,v,t) -f(x-v\Delta t,v,t)| + |f(x-v\Delta t,v,t) - f(x-v\Delta t,v-\mathbb{E}(x,t)\Delta t,t)|\cr
            &\quad\leq \Delta t |v||\partial_x f(\bar{x},v,t)| + \Delta t |\mathbb{E}(x,t)||\partial_v f(x-v\Delta t,\bar{v},t)|\cr
            &\quad\leq \Delta t \|f(t)\|_{W^{1,\infty}_{q+1}}\frac{1}{(1+|v|)^q} +\Delta t \|\mathbb{E}(t)\|_{L^{\infty}_x}\|f(t)\|_{W^{1,\infty}_q} \left(\frac{1+|v|}{1+|\bar{v}|} \right)^q\frac{1}{(1+|v|)^q}\cr
            &\quad \leq C_{q,f_0,T^f}\Delta t\frac{1}{(1+|v|)^q},
        \end{split}
    \end{align*}
    where we used 
    \begin{align*}
        \begin{split}
             \left(\frac{1+|v|}{1+|\bar{v}|} \right)^q=\left(1+\frac{|v|-|\bar{v}|}{1+|\bar{v}|}\right)^q\leq (1+|v-\bar{v}|)^q\leq (1+|\mathbb{E}(x,t)|\Delta t)^q\leq C_{q,f_0,T^f}.
        \end{split}
    \end{align*}
    Therefore, Lemma \ref{LipMaxw} gives
    \begin{align*}
        \begin{split}
            I_2&\leq \Delta t\|(\mathcal{M}(f) - \mathcal{M}(\tilde{f}))(t)\|_{L^{\infty}_q }\leq C_{q,f_0}\Delta t\|(f-\tilde{f})(t)\|_{L^{\infty}_q}\leq C_{q,f_0,T^f}(\Delta t)^2.
        \end{split}
    \end{align*}
    Combining the estimates for $I_1$ and $I_2$, this completes the proof.
\end{proof}


\subsection{Error estimates}
In this section, we derive error estimates for the difference between the continuous and the discrete electric field, solution, and local Maxwellian in a weighted $L^{\infty}$ norm.

\begin{lemma}\label{difference field}\textnormal{(Error estimate for the electric field)} 
    Let $\mathbb{E}(x,t)$ and $\mathbb{E}^n_i$ be defined in (\ref{continuous electric field}) and (\ref{Discrete field 2}). Then, under the assumption of Theorem \ref{Main Thm}, we have 
    \begin{align*}
        \begin{split}
            |\mathbb{E}(x_i,t^n) - \mathbb{E}^n_i| \leq C\|f(t^n)-f^n\|_{L^{\infty}_{q}} + C_{q,f_0,T^f}\left\{\Delta x + \Delta v+\frac{1}{V_M^{q-1}}\right\}. 
        \end{split}
    \end{align*}
\end{lemma}
\begin{proof}
    Let $\Delta_k=[y_{k},y_{k+1})$ for $k=0,\cdots,N_x -1$. From the definition of the continuous and the discrete electric field in \eqref{continuous electric field} and \eqref{Discrete field 2}, we have 
    \begin{align}\label{E-E 1}
        \begin{split}
            &|\mathbb{E}(x_i,t^n) - \mathbb{E}^n_i|\cr 
            &\quad= \left|\int^1_0 K(x_i,y)(\rho(y,t^n)-1)dy - \sum_k K(x_i,y_k)(\rho^n_k -1)\Delta y\right|\cr
            &\quad\leq \left|\sum_k\int_{\Delta_k} K(x_i,y_k)(\rho(y,t^n)-1)dy - \sum_k K(x_i,y_k) (\rho^n_k -1)\Delta y\right|\cr 
            &\quad+ \sum_k\left|\int_{\Delta_k} (K(x_i,y) - K(x_i,y_k)) (\rho(y,t^n)-1)dy\right|\cr
            &\quad \equiv I + II.
        \end{split}
    \end{align}
    \textbf{(1) Estimate for $\boldsymbol{I}$}\\
    We expand $\rho(y,t^n)$ using the Taylor formula such that
    \begin{align}\label{rho taylor 6.9}
        \begin{split}
            \rho(y,t^n)=\rho(y_k,t^n) + (y-y_k)\partial_x \rho(y_{k,1},t^n),\quad y_{k,1}\in[y_k,y].
        \end{split}
    \end{align}
    Applying the above expansion, we split $I$ into two parts as follows:
    \begin{align}\label{field diff est I}
        \begin{split}
            I&=\left|\sum_k\int_{\Delta_k} K(x_i,y_k)(\rho(y,t^n)-1)dy - \sum_k K(x_i,y_k) (\rho^n_k -1)\Delta y\right|\cr
            &\leq \sum_k |K(x_i,y_k)||\rho(y_k,t^n)-\rho^n_k|\Delta y + \sum_k\int_{\Delta_k} |y-y_k||K(x_i,y_k)||\partial_x \rho(y_{k,1},t^n)|dy\cr
            &\equiv I_1 + I_2.
        \end{split}
    \end{align}
    To derive the estimate for $I_1$, we let $\Delta_l=[v_{l},v_{l+1})$ for $l=-N_v,\cdots,N_v-1$, and 
    \begin{align*}
        \begin{split}
            \delta^{n}_{k} =\int_{|v|\geq V_M} f(y_k,v,t^n)dv.
        \end{split}
    \end{align*}
    Using the above notation, $\rho(y_k,t^n)$ can be represented to
    \begin{align*}
        \begin{split}
            \rho(y_k,t^n) &= \int_{\mathbb{R}}f(y_k,v,t^n)dv\cr
            &=\sum_{-N_v\leq l\leq N_v-1} \int_{\Delta_l}f(y_k,v,t^n)dv+\delta^{n}_{k}\cr
            &=\sum_{-N_v\leq l\leq N_v} \int_{\Delta_l}f(y_k,v,t^n)dv-\int_{\Delta_{N_v}}f(y_k,v,t^n)dv+\delta^{n}_{k}.
        \end{split}
    \end{align*}
Then, we split $\rho(y_k,t^n) -\rho^n_k$ into three parts $J_1$, $J_2$, and $\delta^n_k$, which are estimates in the interior, at the boundary, and in the tail, respectively: 
\begin{align*}
    \begin{split}
        |\rho(y_k,t^n) -\rho^n_k|&\leq \left|\sum_{-N_v\leq l\leq N_v} \int_{\Delta_l}f(y_k,v,t^n)dv-\sum_{-N_v\leq l\leq N_v}f^n_{k,l}\Delta v \right|+\int_{\Delta_{N_v}}f(y_k,v,t^n)dv+\delta^{n}_{k}\cr
        &\equiv J_1 + J_2  +\delta^{n}_{k}.
    \end{split}
\end{align*}
For $J_1$, we use Taylor formula for $f(y_k,v,t^n)$ to get
 \begin{align}\label{E diff taylor}
    \begin{split}
        f(y_k,v,t^n) = f(y_k,v_l,t^n) + (v-v_l)\partial_v f(y_k,v_{l,1},t^n),\; \text{for some }v_{l,1}\in(v,v_l).
    \end{split}
\end{align}
Applying this, we have
\begin{align*}
    \begin{split}
        J_1 &\leq \left|\sum_{-N_v\leq l\leq N_v}\int_{\Delta_l}f(y_k,v_l,t^n)dv-\sum_{-N_v\leq l\leq N_v}f^n_{k,l}\Delta v \right|+\left|\sum_{-N_v\leq l\leq N_v}\int_{\Delta_l}(v-v_l)\partial_v f(y_k,v_{l,1},t^n)dv \right|\cr
        &\leq \left|\sum_{-N_v\leq l\leq N_v}(f(y_k,v_l,t^n)-f^n_{k,l})\Delta v \right|+\left|\Delta v\sum_{-N_v\leq l\leq N_v}\partial_vf(y_k,v_{l,1},t^n)\Delta v \right|,
    \end{split}
\end{align*}
where we used
\begin{align*}
    \begin{split}
        \int_{\Delta_l}dv=\Delta v,\quad{and}\quad|v-v_l|<|v_{l+1}-v_{l}|<\Delta v.
    \end{split}
\end{align*}
The first term on the right-hand side is bounded by
\begin{align*}
    \begin{split}
        \left|\sum_{-N_v\leq l\leq N_v}(f(y_k,v_l,t^n)-f^n_{k,l})\Delta v \right|\leq\|f(t^n)-f^n\|_{L^{\infty}_q}\sum_{-N_v\leq l\leq N_v}\frac{\Delta v}{(1+|v_l|)^q}\leq C\|f(t^n)-f^n\|_{L^{\infty}_q}.
    \end{split}
\end{align*}
The second term on the right-hand side is controlled by
\begin{align*}
    \begin{split}
        \left|\Delta v\sum_{-N_v\leq l\leq N_v}\partial_vf(y_k,v_{l,1},t^n)\Delta v \right|&\leq \Delta v\|f(t^n)\|_{W^{1,\infty}_q}\sum_{-N_v\leq l\leq N_v}\left(\frac{1+|v_l|}{1+|v_{l,1}|} \right)^q\frac{\Delta v}{(1+|v_l|)^q}\cr
        &\leq C_{q,f_0,T^f}\Delta v\sum_{-N_v\leq l\leq N_v}\frac{\Delta v}{(1+|v_l|)^q}\cr
        &\leq C_{q,f_0,T^f}\Delta v.
    \end{split}
\end{align*}
In these estimates, we used the smallness condition on $\Delta v$ such that
\begin{align*}
    \begin{split}
        \sum_{-N_v\leq l\leq N_v}\frac{\Delta v}{(1+|v_l|)^q}\leq 2\bar{D}_q,
    \end{split}
\end{align*}
and
\begin{align*}
    \begin{split}
        \left(\frac{1+|v_l|}{1+|v_{l,1}|} \right)^q\leq \left(1+\frac{|v_l|-|v_{l,1}|}{1+|v_{l,1}|} \right)^q\leq (1+|v_{l}-v_{l,1}|)^q\leq (1+\Delta v)^q<\left(1+\frac{1}{2}\right)^q.
    \end{split}
\end{align*}
Combining the estimates for the first and second terms, we yield
\begin{align}\label{field diff est J_1}
    \begin{split}
        J_1\leq C\|f(t^n)-f^n\|_{L^{\infty}_q}+C_{q,f_0,T^f}\Delta v.
    \end{split}
\end{align}
For $J_2$, we have
\begin{align}\label{field diff est J_2}
    \begin{split}
        J_2 \leq \|f(t^n)\|_{L^{\infty}_q}\int_{\Delta_{N_v}}dv\leq C_{q,f_0,T^f}\Delta v.
    \end{split}
\end{align}
On the other hand, $\delta^n_k$ is controlled by
\begin{align}\label{field diff est delta}
        \begin{split}
           \delta^{n}_{k} \leq \|f(t^n)\|_{L^{\infty}_{q}}\int_{|v|\geq V_M}\frac{1}{(1+|v|)^q}dv\leq 2C_{q,f_0,T^f}\int^{\infty}_{V_M}(1+|v|)^{-q}dv\leq \frac{C_{q,f_0,T^f}}{V_M^{q-1}}.
        \end{split}
    \end{align}
Combining \eqref{field diff est J_1}, \eqref{field diff est J_2}, and \eqref{field diff est delta}, we deduce
\begin{align*}
    \begin{split}
        |\rho(y_k,t^n)-\rho^n_k|\leq C\|f(t^n)- f^n\|_{L^{\infty}_q}+C_{q,f_0,T^f}\left\{ \Delta v + \frac{1}{V_M^{q-1}}\right\}.
    \end{split}
\end{align*}
From the definition of the Green kernel $K$, we have $|K(x_i,y_k)|\leq 1$. Therefore, $I_1$ is estimated by 
\begin{align}\label{field diff est I_1}
    \begin{split}
        I_1 &\leq \sum_{k}|\rho(y_k,t^n)-\rho^n_k|\Delta y\cr 
        &\leq \left[C\|f(t^n) -f^n\|_{L^{\infty}_q} + C_{q,f_0,T^f}\left\{ \Delta v + \frac{1}{V_M^{q-1}}\right\}\right]\sum_{k}\Delta y\cr
        &=C\|f(t^n) -f^n\|_{L^{\infty}_q} + C_{q,f_0,T^f}\left\{ \Delta v + \frac{1}{V_M^{q-1}}\right\}.
    \end{split}
\end{align}
The estimate for $I_2$ is derived as follows:
\begin{align}\label{field diff est I_2}
    \begin{split}
        I_2&\leq \Delta y\sum_k\int_{\Delta _k}|K(x_i,y_k)|\left(\int_{\mathbb{R}}|\partial_x f(y_{k,1},v,t^n)|dv\right)dy\cr
        &\leq \Delta y\left( \sum_k\int_{\Delta_k}dy\right)\left(\int_{\mathbb{R}}|\partial_x f(y_{k,1},v,t^n)|dv\right)\cr
        &\leq \|f(t^n)\|_{W^{1,\infty}_q}\Delta y\int_{\mathbb{R}}\frac{dv}{(1+|v|)^q}\cr
        &\leq C_{q,f_0,T^f}\Delta y, 
    \end{split}
\end{align}
where we used $|y-y_k|\leq \Delta y$ on the first line and $|K(x_i,y_k)|\leq 1$ on the second line. Since $k$ shares the same mesh size as $i$, we insert \eqref{field diff est I_1}, \eqref{field diff est I_2} into \eqref{field diff est I} to get
\begin{align}\label{E-E 2}
    \begin{split}
        I \leq  C \|f(t^n) -f^n\|_{L^{\infty}_q} + C_{q,f_0,T^f}\left\{\Delta x+ \Delta v + \frac{1}{V_M^{q-1}}\right\}.
    \end{split}
\end{align}
\textbf{(2) Estimate for $\boldsymbol{II}$}\\
We now consider the estimate for $II$. Let $x_i\in \Delta_{k_1}$ for some $k_1\in[0,N_x-1]$. Since $K(x,y)$ has a discontinuity at $x=y$, we split $II$ into two parts whether $K$ is differentiable or not:   
\begin{align*}
    \begin{split}
        II &= \sum_{x_i\notin \Delta_k}\int_{\Delta_k} (K(x_i,y)-K(x_i,y_k))(\rho(y,t^n)-1)dy\cr
        &+ \int_{\Delta_{k_1}}(K(x_i,y)-K(x_i,y_k))(\rho(y,t^n)-1)dy\cr
        &\equiv II_1 + II_2.
    \end{split}
\end{align*}
For some $y_{k_2}\in (y_{k},y)$, the mean value theorem gives the estimate for $II_1$ as follows:
\begin{align}\label{field diff est II_1}
    \begin{split}
       II_1&\leq \sum_k \int_{\Delta_k}|\partial_y K(x_i,y_{k,2})||y-y_k||\rho(y,t^n)-1|dy\cr
       &\leq \sum_{k}\int_{\Delta_k}|\partial_y K(x_i,y_{k,2})||y-y_k|\left\{\int_{\mathbb{R}}|f(y,v,t^n)|dv + 1\right\}dy\cr
       &\leq  C_{q}\{\|f(t^n)\|_{L^{\infty}_q} + 1\}\sum_{k}\int_{\Delta_k}|\partial_y K(x_i,y_{k,2})||y-y_k|dy\cr
        &\leq C_{q,f_0,T^f}\Delta y, 
    \end{split}
\end{align}
where we used $\sup_{x,y}|\partial_y K(x,y)|\leq 1$ and $|y-y_k|\leq \Delta y$. For $II_2$, $\sup_{x,y}|K(x,y)|\leq 1$ implies that
\begin{align}\label{field diff est II_2}
    \begin{split}
        II_2 &\leq  \int_{\Delta_{k_1}} (|K(x_i,y)| + |K(x_i,y_k)| )|\rho(y,t^n)-1|dy\cr
        &\leq 2 \int_{\Delta_{k_1}}\left\{\int_{\mathbb{R}}|f(y,v,t^n)|dv + 1\right\}dy\cr
        &\leq C_{q}\{\|f(t^n)\|_{L^{\infty}_q} + 1\}\int_{\Delta_{k_1}}dy\cr
        & \leq C_{q,f_0,T^f}\Delta y.
    \end{split}   
\end{align}
Since $k$ shares the same mesh size as $i$, \eqref{field diff est II_1} and \eqref{field diff est II_2} give the estimate for $II$:
\begin{align}\label{E-E 3}
    \begin{split}
        II \leq C_{q,f_0,T^f}\Delta x.
    \end{split}
\end{align}
Plugging \eqref{E-E 2} and \eqref{E-E 3} into \eqref{E-E 1}, we obtain the desired result.
\end{proof}  

We now introduce lemmas regarding the error estimate for $\tilde{f}$ and its moments.
\begin{lemma}\label{f estimate}\textnormal{(Error estimate for the distribution function)} 
    Let $f(x,v,t)$ and $f^{n}_{i,j}$ be the solution of (\ref{VPBGK model}) and (\ref{VPBGK_Scheme}). Then, under the assumption of Theorem \ref{Main Thm}, we have 
    \begin{align*}
        \begin{split}
            \|\tilde{f}(t^n) -\tilde{f}^n\|_{L^{\infty}_{q}} &\leq (1+C_{q,f_0,T^f,\varepsilon}\Delta t)\|f(t^n) - f^n\|_{L^{\infty}_{q}}\cr
            &+ C_{q,f_0,T^f}\left\{\Delta x\Delta t+ \Delta v\Delta t + (\Delta t)^2 + \frac{\Delta t}{V_M^{q-1}}\right\}.
        \end{split}
    \end{align*}
\end{lemma}
\begin{proof}
    We only consider $v_j\geq 0$ and $\mathbb{E}^{n}_{i} < 0$, which is the case where the particle velocity is decelerated by the electric field. The proof of an accelerated case is omitted, as the analysis proceeds in the same manner. Recall our update rule \eqref{convex_combi} for $\tilde{f}^n_{i,j}$ when $v_j\geq0$ and $\mathbb{E}^{n}_{i} < 0$:
    \begin{align}\label{lem 5.7 update}
        \begin{split}
            \tilde{f}^n_{i,j}=\left(1-\frac{\Delta t}{\Delta x}|v_j|-\frac{\Delta t}{\Delta v}|\mathbb{E}^n_i| \right)f^n_{i,j}+\frac{\Delta t}{\Delta x}|v_j|f^n_{i-1,j}+\frac{\Delta t}{\Delta v}|\mathbb{E}^n_i|f^n_{i,j+1}.
        \end{split}
    \end{align}
    Since the Neumann boundary condition \eqref{Neumann boundary condition} modifies the last term on the right-hand side at $j=N_v$, we treat the interior case $(0\leq j<N_v)$, and the boundary case $(j=N_v)$ separately.\\
    \textbf{(1) Interior case ($\boldsymbol{0\leq j<N_v}$):} We will reformulate $\tilde{f}(x_i,v_j,t^n)$ to be consistent with the structure of the above formula \eqref{lem 5.7 update}. To do this, we apply the Taylor formula, for some $x_{i,1}\in(x_i - v_j\Delta t,x_i)$ and $v_{j,1},\;v_{j,2} \in (v_j,v_j - \mathbb{E}(x_i,t^n)\Delta t)$, to get
    \begin{align}\label{lem 5.7 tilde taylor}
        \begin{split}
            \tilde{f}(x_i,v_j,t^n) &= f(x_i-v_j\Delta t,v_j-\mathbb{E}(x_i,t^n)\Delta t,t^n)\cr
            &= f(x_i,v_j,t^n) -v_j\Delta t\partial_x f(x_i,v_j,t^n) -\mathbb{E}(x_i,t^n)\Delta t\partial_v f(x_i,v_j,t^n)\cr
            &+\frac{(v_j \Delta t)^2}{2}\partial^2_x f(x_{i,1},v_j-\mathbb{E}(x_i,t^n)\Delta t,t^n)+ v_j\mathbb{E}(x_i,t^n)(\Delta t)^2\partial_v\partial_x f(x_i,v_{j,1},t^n)\cr
            &+\frac{(\mathbb{E}(x_i,t^n)\Delta t)^2}{2}\partial^2_v f(x_i,v_{j,2},t^n).
        \end{split}
    \end{align}
    Using the Taylor formula again, for some $x_{i,2} \in (x_{i-1},x_i)$ and $v_{j,3} \in (v_{j-1},v_j)$, we have 
    \begin{align}\label{lem 5.7 replace}
        \begin{split}
            \partial_x f(x_i,v_j,t^n) &= \frac{f(x_i,v_j,t^n)-f(x_{i-1},v_j,t^n)}{\Delta x}+\frac{\Delta x}{2}\partial^2_x f(x_{i,2},v_j,t^n),\cr
            \partial_v f(x_i,v_j,t^n) &= \frac{f(x_i,v_{j+1},t^n)-f(x_{i},v_{j},t^n)}{\Delta v} -\frac{\Delta v}{2}\partial^2_v f(x_i,v_{j,3},t^n).
        \end{split}
    \end{align}
    Inserting \eqref{lem 5.7 replace} into \eqref{lem 5.7 tilde taylor}, we rewrite $\tilde{f}(x_i,v_j,t^n)$ as
    \begin{align}\label{tilde conti convex}
        \begin{split}
            &\tilde{f}(x_i,v_j,t^n)\cr 
            &\quad= \left(1-\frac{\Delta t}{\Delta x}|v_j|-\frac{\Delta t}{\Delta v}|\mathbb{E}^{n}_{i}|\right)f(x_i,v_j,t^n) + \frac{\Delta t}{\Delta x}|v_j| f(x_{i-1},v_j,t^n) + \frac{\Delta t}{\Delta v}|\mathbb{E}^{n}_{i}|f(x_i,v_{j+1},t^n)\cr
            &\quad+ R_1 + R_2 + R_3,
        \end{split}
    \end{align}
    where $R_1$, $R_2$, and $R_3$ are given by
    \begin{align}\label{remainder f est}
        \begin{split}
            R_1 &= \frac{\Delta t}{\Delta v}(\mathbb{E}^{n}_{i} - \mathbb{E}(x_i,t^n))(f(x_i,v_{j+1},t^n) - f(x_i,v_{j},t^n)),\cr 
            R_2 &= -\frac{\Delta x\Delta t}{2}v_j\partial^2_x f(x_{i,2},v_j,t^n)+\frac{\Delta v\Delta t}{2}\mathbb{E}(x_i,t^n)\partial^2_v f(x_i,v_{j,3},t^n),\cr
            R_3 &= \frac{(v_j \Delta t)^2}{2}\partial^2_x f(x_{i,1},v_j-\mathbb{E}(x_i,t^n)\Delta t,t^n) + v_j\mathbb{E}(x_i,t^n)(\Delta t)^2\partial_v\partial_x f(x_i,v_{j,1},t^n)\cr 
            &+\frac{(\mathbb{E}(x_i,t^n)\Delta t)^2}{2}\partial^2_v f(x_i,v_{j,2},t^n).
        \end{split}
    \end{align}
    Subtracting \eqref{lem 5.7 update} from \eqref{tilde conti convex} and multiplying $(1+|v_j|)^q$, we obtain
    \begin{align}\label{lem 5.7 diff}
        \begin{split}
            (1+|v_j|)^q |\tilde{f}(x_i,v_j,t^n)-\tilde{f}^{n}_{i,j} | &\leq \left(1- \frac{\Delta t}{\Delta x}|v_j| - \frac{\Delta t}{\Delta v}|\mathbb{E}^{n}_{i}|\right)(1+|v_j|)^q|f(x_i,v_j,t^n)-f^{n}_{i,j}|\cr
            &+\frac{\Delta t}{\Delta x}|v_j| (1+|v_j|)^q|f(x_{i-1},v_j,t^n)-f^{n}_{i-1,j}|\cr
            &+\frac{\Delta t}{\Delta v}|\mathbb{E}^{n}_{i}|(1+|v_j|)^q|f(x_i,v_{j+1},t^n)- f^{n}_{i,j+1}|\cr
            &+(1+|v_j|)^q(R_1+R_2+R_3).
            \end{split}
    \end{align}
    Since $\Delta v$ is less than one, we have
    \begin{align*}
        \begin{split}
            \left(\frac{1+|v_j|}{1+|v_{j+1}|} \right)^{q}=\left(1+\frac{|v_j|-|v_{j+1}|}{1+|v_{j+1}|}\right)^q\leq (1+|v_{j}-v_{j+1}|)^q\leq (1+\Delta v)^q\leq 1+C_q\Delta v.
        \end{split}
    \end{align*}
    Using this, we control the third term on the right-hand side in \eqref{lem 5.7 diff} as follows:
    \begin{align*}
        \begin{split}
            &\frac{\Delta t}{\Delta v}|\mathbb{E}^{n}_{i}|(1+|v_j|)^{q}|f(x_i,v_{j+1},t^n) - f^{n}_{i,j+1} |\cr 
            &= \frac{\Delta t}{\Delta v}|\mathbb{E}^{n}_{i}|(1+|v_{j+1}|)^{q}| f(x_i,v_{j+1},t^n) - f^{n}_{i,j+1}|\left(\frac{1+|v_j|}{1+|v_{j+1}|} \right)^{q}\cr
            &\leq (1+C_q\Delta v)\frac{\Delta t}{\Delta v}|\mathbb{E}^{n}_{i}|(1+|v_{j+1}|)^{q}| f(x_i,v_{j+1},t^n)- f^{n}_{i,j+1}|\cr
            &\leq \left(\frac{\Delta t}{\Delta v}|\mathbb{E}^{n}_{i}|+C_{q}C_{\mathbb{E},1} \Delta t\right)\|f(t^n)-f^{n}\|_{L^{\infty}_{q}}.
        \end{split}
    \end{align*}
    This implies 
    \begin{align}\label{f diff 1}
        \begin{split}
            &(1+|v_j|)^{q}|\tilde{f}(x_i,v_j,t^n)-\tilde{f}^{n}_{i,j}|\cr
            &\qquad\leq (1+C_qC_{\mathbb{E},1}\Delta t)\|f(t^n)-f^{n}\|_{L^{\infty}_{q}}+(1+|v_j|)^q(R_1+R_2+R_3).
        \end{split}
    \end{align}
    We now consider the estimates for $R_1$, $R_2$, and $R_3$. Applying Lemma \ref{difference field}, for some $v_{j,4}\in(v_{j},v_{j+1})$, we estimate $R_1$ as follows:
    \begin{align}\label{f diff R_1}
        \begin{split}
            R_1 &\leq \frac{\Delta t}{\Delta v}|f(x_i,v_{j+1},t^n) - f(x_i,v_{j},t^n)||\mathbb{E}(x_i,t^n) - \mathbb{E}^n_i|\cr
            &\leq \Delta t|\partial_v f(x_i,v_{j,4},t^n)||\mathbb{E}(x_i,t^n) - \mathbb{E}^n_i|\cr
            &\leq C_{q,f_0,T^f}\|f(t^n)\|_{W^{1,\infty}_{q}}\Delta t\left\{\|f(t^n)-f^n\|_{L^{\infty}_{q}} + \Delta x + \Delta v + \frac{1}{V_M^{q-1}}\right\}\cr
            &\quad\times\left(\frac{1+|v_j|}{1+|v_{j,4}|}\right)^q\frac{1}{(1+|v_j|)^q}\cr
            &\leq C_{q,f_0,T^f}\left\{\Delta t\|f(t^n)-f^n\|_{L^{\infty}_{q}}+\Delta x\Delta t  + \Delta v\Delta t + \frac{\Delta t}{V_M^{q-1}}\right\}\frac{1}{(1+|v_j|)^q},
        \end{split}
    \end{align}
    where we used the smallness condition $\Delta v<r_{\Delta v}$ as follows:
    \begin{align}\label{v to v}
        \begin{split}
            \left(\frac{1+|v_j|}{1+|v_{j,4}|}\right)^q=\left(1+\frac{|v_j|-|v_{j,4}|}{1+|v_{j,4}|}\right)^q\leq (1+|v_{j}-v_{j,4}|)^q\leq (1+\Delta v)^q\leq \left(1+\frac{1}{2}\right)^q\leq C_q.
        \end{split}
    \end{align}
    For $R_2$, we have
    \begin{align}\label{f diff R_2}
        \begin{split}
            R_2&\leq \frac{\Delta x\Delta t}{2}|v_j||\partial^2_x f(x_{i,2},v_j,t^n)|+\frac{\Delta v\Delta t}{2}|\mathbb{E}(x_i,t^n)||\partial^2_v f(x_i,v_{j,3},t^n)|\cr           
            &\leq \|f(t^n)\|_{W^{2,\infty}_{q+1}}\frac{\Delta x \Delta t}{2(1+|v_j|)^{q}} +  \|f(t^n)\|_{W^{2,\infty}_{q}}\|\mathbb{E}(t^n)\|_{L^{\infty}_{x}}\left(\frac{1+|v_{j}|}{1+|v_{j,3}|}\right)^q\frac{\Delta v \Delta t}{2(1+|v_{j}|)^{q}}\cr
            &\leq C_{q,f_0,T^f}\left\{\Delta x\Delta t + \Delta v\Delta t \right\}\frac{1}{(1+|v_j|)^{q}}.
        \end{split}
    \end{align}
    In the third line, we used the following property similar to \eqref{v to v}:
    \begin{align*}
        \begin{split}
            \left(\frac{1+|v_{j}|}{1+|v_{j,3}|}\right)^q\leq C_q.
        \end{split}
    \end{align*}
    For $R_3$, we get
    \begin{align}\label{f diff R_3}
        \begin{split}
            R_3&\leq \frac{|v_j|^2 (\Delta t)^2}{2}|\partial^2_x f(x_{i,1},v_j-\mathbb{E}(x_i,t^n)\Delta t,t^n)| + |v_j||\mathbb{E}(x_i,t^n)|(\Delta t)^2|\partial_v\partial_x f(x_i,v_{j,1},t^n)|\cr 
            &+\frac{|\mathbb{E}(x_i,t^n)|^2(\Delta t)^2}{2}|\partial^2_v f(x_i,v_{j,2},t^n)|\cr            
            &\leq \frac{(\Delta t)^2\|f(t^n)\|_{W^{2,\infty}_{q+2}}}{2}\left(\frac{1+|v_j|}{1+|v_j -\mathbb{E}(x_i,t^n)\Delta t|} \right)^q\frac{1}{(1+|v_j|)^q}\cr
            &+\frac{(\Delta t)^2\|f(t^n)\|_{W^{2,\infty}_{q+2}}\|\mathbb{E}(t^n)\|_{L^{\infty}_x}}{2}\left(\frac{1+|v_j|}{1+|v_{j,1}|} \right)^q\frac{1}{(1+|v_j|)^q}\cr
            &+\frac{(\Delta t)^2\|\mathbb{E}(t^n)\|^2_{L^{\infty}_x}}{2}\left(\frac{1+|v_j|}{1+|v_{j,2}|} \right)^q\frac{1}{(1+|v_j|)^q}\cr
            &\leq C_{q,f_0,T^f}(\Delta t)^2 \frac{1}{(1+|v_j|)^q}.
        \end{split}
    \end{align}
    In the final step, we used the following properties:
    \begin{align*}
        \begin{split}
            \left(\frac{1+|v_j|}{1+|v_j -\mathbb{E}(x_i,t^n)\Delta t|} \right)^q&=\left(1+\frac{|v_j|-|v_j -\mathbb{E}(x_i,t^n)\Delta t|}{1+|v_j -\mathbb{E}(x_i,t^n)\Delta t|}\right)^q\cr
            &\leq (1+|\mathbb{E}(x_i,t^n)\Delta t|)^q\cr
            &\leq (1+C_{q,f_0,T^f}\Delta t)^q\cr
            &\leq C_{q,f_0,T^f},
        \end{split}
    \end{align*}
    and
    \begin{align*}
        \begin{split}
            \left(\frac{1+|v_{j}|}{1+|v_{j,1}|}\right)^q,\left(\frac{1+|v_{j}|}{1+|v_{j,2}|}\right)^q \leq C_q.
        \end{split}
    \end{align*}
    Plugging \eqref{f diff R_1}, \eqref{f diff R_2}, and \eqref{f diff R_3} into \eqref{f diff 1}, we obtain
    \begin{align}\label{error est f interior}
        \begin{split}
            (1+|v_j|)^q |\tilde{f}(x_i,v_j,t^n)-\tilde{f}^{n}_{i,j} | &\leq (1+C_{q,f_0,T^f,\varepsilon}\Delta t)\|f(t^n) - f^n\|_{L^{\infty}_{q}} \cr
            &+ C_{q,f_0,T^f}\left\{\Delta x\Delta t+ \Delta v\Delta t + (\Delta t)^2 + \frac{\Delta t}{V_M^{q-1}}\right\}.
        \end{split}
    \end{align}
     \textbf{(2) Boundary case ($\boldsymbol{j=N_v}$):} The update rule \eqref{convex_combi} at the boundary is modified due to the Neumann condition \eqref{Neumann boundary condition} as follows:
\begin{align}\label{App boundary}
    \begin{split}
        \tilde{f}^n_{i,N_v}=\left(1-\frac{\Delta t}{\Delta x}V_M-\frac{\Delta t}{\Delta v}|\mathbb{E}^n_i| \right)f^n_{i,N_v} + \frac{\Delta t}{\Delta x}V_Mf^n_{i-1,N_v}+\frac{\Delta t}{\Delta v}|\mathbb{E}^n_i|f^n_{i,N_v}.
    \end{split}
\end{align}
Similar to the proof of the interior case (1), we rewrite $f(x_i,V_M,t^n)$ in a consistent form with \eqref{App boundary} such that
\begin{align}\label{app consistent f}
    \begin{split}
        &\tilde{f}(x_i,V_M,t^n)\cr
        &=\left(1-\frac{\Delta t}{\Delta x}V_M-\frac{\Delta t}{\Delta v}|\mathbb{E}^n_i| \right)f(x_{i},V_M,t^n) + \frac{\Delta t}{\Delta x}V_Mf(x_{i-1},V_M,t^n)+\frac{\Delta t}{\Delta v}|\mathbb{E}^n_i|f(x_i,V_M,t^n)\cr
        &+R_1 + R_2 +R_3+R_4,
    \end{split}
\end{align}
where $R_1$, $R_2$, and $R_3$ are given by \eqref{remainder f est} with $v_{j,1},v_{j,2}\in (V_M,V_M-\mathbb{E}(x_i,t^n)\Delta t)$ and $v_{j,3}\in (V_M,V_M+\Delta v)$. The additional remainder term $R_4$ is defined by
\begin{align*}
    \begin{split}
        R_4 =-\frac{\Delta t}{\Delta v}\mathbb{E}^n_i(f(x_i,V_M+\Delta v,t^n)-f(x_i,V_M,t^n)).
    \end{split}
\end{align*}
Subtracting \eqref{app consistent f} from \eqref{App boundary} and multiplying by $(1+V_M)^q$, we have
\begin{align}\label{f diff 2}
    \begin{split}
        &(1+|V_M|)^q|\tilde{f}(x_i,V_M,t^n)-\tilde{f}^n_{i,N_v}|\cr
        &\quad\leq \left(1-\frac{\Delta t}{\Delta x}V_M-\frac{\Delta t}{\Delta v}|\mathbb{E}^n_i| \right)(1+V_M)^q|f(x_i,V_M,t^n)-f^n_{i,N_v}|\cr 
        &\quad+ \frac{\Delta t}{\Delta x}V_M(1+V_M)^q|f(x_{i-1},V_M,t^n)-f^n_{i-1,N_v}|\cr
        &\quad+\frac{\Delta t}{\Delta v}|\mathbb{E}^n_i|(1+V_M)^q|f(x_i,V_M,t^n)-f^n_{i,N_v}|\cr
        &\quad+(1+V_M)^q(R_1 + R_2 +R_3+R_4)\cr
        &\quad\leq \|f(t^n)-f^n\|_{L^{\infty}_q} + (1+V_M)^q(R_1 + R_2 +R_3+R_4).
    \end{split}
\end{align}
The estimates for $R_1$, $R_2$, and $R_3$ are already derived in the proof of interior case (1) as follows:
\begin{align}\label{f diff bdd R123}
    \begin{split}
        &(1+V_M)^q(R_1 + R_2 +R_3)\cr
        &\qquad\leq C_{q,f_0,T^f}\left\{\Delta t\|f(t^n)-f^n\|_{L^{\infty}_q}+\Delta x\Delta v+\Delta v\Delta t + (\Delta t)^2 +\frac{\Delta t}{V_M^{q-1}} \right\}.
    \end{split}
\end{align}
Therefore, it is sufficient to obtain an estimate for $R_4$ to prove this lemma. Using the mean value theorem, for $\theta\in(0,1)$, we have
\begin{align*}
    \begin{split}
        (1+V_M)^qR_4&\leq \frac{\Delta t}{\Delta v}\|\mathbb{E}^n\|_{L^{\infty}_x}\Delta v(1+V_M)^q|\partial_v f(x_i,V_M+\theta\Delta v,t^n)|\cr
        &\leq C_{\mathbb{E},1}\Delta t\frac{(1+V_M)^{q+2}}{V^2_M}|\partial_v f(x_i,V_M+\theta\Delta v,t^n)|\cr
        &\leq C_{\mathbb{E},1}\Delta t\frac{(1+(V_M+\theta\Delta v))^{q+2}}{V^2_M}|\partial_v f(x_i,V_M+\theta\Delta v,t^n)|\cr
        &\leq C_{\mathbb{E},1}\|f(t^n)\|_{W^{1,\infty}_{q+2}}\frac{\Delta t}{V^2_M}.
    \end{split}
\end{align*}
Recalling the truncation assumption \eqref{V_M large}, that is, $V_M=C_{\mathbb{E},1}(\Delta v)^{-\gamma}$, we get
\begin{align*}
    \begin{split}
        (1+V_M)^q R_4\leq C_{q,f_0,T^f,\varepsilon}\Delta t(\Delta v)^{2\gamma}.
    \end{split}
\end{align*}
Since $\frac{1}{2}\leq \gamma< 1$, this implies
\begin{align}\label{f diff bdd R4}
    \begin{split}
        (1+V_M)^q R_4\leq C_{q,f_0,T^f,\varepsilon}\Delta v\Delta t.
    \end{split}
\end{align}
Plugging \eqref{f diff bdd R123} and \eqref{f diff bdd R4} into \eqref{f diff 2}, we observe 
\begin{align}\label{error est f boundary}
    \begin{split}
         (1+|V_M|)^q|\tilde{f}(x_i,V_M,t^n)-\tilde{f}^n_{i,N_v}| &\leq (1+\Delta t)\|f(t^n) - f^n\|_{L^{\infty}_{q}} \cr
            &+ C_{q,f_0,T^f}\left\{\Delta x\Delta t+ \Delta v\Delta t + (\Delta t)^2 + \frac{\Delta t}{V_M^{q-1}}\right\}.
    \end{split}
\end{align}
Consequently, we yield the desired result by taking the $\sup_{0\leq j\leq N_v}$ to \eqref{error est f interior} and \eqref{error est f boundary}.
\end{proof}

\begin{lemma}\label{Moment diff}\textnormal{(Error estimate for the moments)}
    Let $f(x,v,t)$ and $f^{n}_{i,j}$ be the solution of (\ref{VPBGK model}) and (\ref{VPBGK_Scheme}). Then, under the assumption of Theorem \ref{Main Thm}, we have
    \begin{align*}
        \begin{split}
            \left| \int_{\mathbb{R}}\tilde{f}(x_{i},v,t^n)\varphi(v)dv-\sum_{j}\tilde{f}^{n}_{i,j}\varphi(v_{j})\Delta v \right|
            &\leq C_{q,f_0,T^f}\|f^n-f(t^n)\|_{L^{\infty}_{q}}\cr         
            &+C_{q,f_0,T^f}\left\{\Delta x\Delta t +\Delta v + \Delta t + \frac{1}{V_M^{q-3}} \right\},
        \end{split}
    \end{align*}
    where $\varphi(v) = (1,v,|v|^{2})$.
\end{lemma}
\begin{proof}
    We only consider the following deceleration case: $v_j\geq 0$ and $\mathbb{E}^n_i< 0$. The proof of the other case proceeds in the same argument. We first define $\Delta_j$ and $\delta(x,t)$ as follows:
    \begin{align*}
        \Delta_j=[v_{j},v_{j+1}),\quad\delta(x,t) = \int_{|v|\geq V_M} \tilde{f}(x,v,t)\varphi(v)dv.
    \end{align*}
    Then, we have 
    \begin{align*}
        \begin{split}
            \int^{\infty}_{0} \tilde{f}(x_i,v,t^n)\varphi(v)dv &= \sum_{0\leq j\leq N_v-1} \int_{\Delta_j}\tilde{f}(x_i,v,t^n)\varphi(v)dv + \delta(x,t).
        \end{split}
    \end{align*}
    Using this, we divide the difference between the continuous and the discrete moments for $\tilde{f}$ into four parts:
    \begin{align*}
        \begin{split}
            &\int^{\infty}_{0} \tilde{f}(x_i,v,t^n)\varphi(v)dv - \sum_{0\leq j\le N_v} \tilde{f}^n_{i,j}\varphi(v_j)\Delta v\cr
            &\quad= \sum_{0\leq j\leq N_v-1}\int_{\Delta_j}\tilde{f}(x_i,v,t^n)\varphi(v)dv - \sum_{0\leq j\leq N_v-1}\tilde{f}^{n}_{i,j}\varphi(v_j)\Delta v +\tilde{f}^n_{i,N_v}\varphi(V_M)\Delta v+ \delta(x,t)\cr
            &\quad= \left(\sum_{0\leq j\leq N_v-1}\int_{\Delta_j}\tilde{f}(x_i,v,t^n)\varphi(v_j)dv - \sum_{0\leq j\leq N_v-1} \tilde{f}^n_{i,j}\varphi(v_j)\Delta v \right)\cr 
            &\quad+ \sum_{0\leq j\leq N_v-1} \int_{\Delta_j}\tilde{f}(x_i,v,t^n)(\varphi(v)-\varphi(v_j))dv+\tilde{f}^n_{i,N_v}\varphi(V_M)\Delta v+ \delta(x,t)\cr
            &\quad\equiv I + II + III +\delta(x,t).
        \end{split}
    \end{align*}
    \textbf{(1) Estimate for $\boldsymbol{I}$}\\
    To derive the estimate for $I$, we expand $\tilde{f}(x_i,v,t^n)$ using Taylor's formula for some $x_{i,1}\in(x_i-v\Delta t,x_i)$ and $v_1,\,v_2\in (v,v-\mathbb{E}(x_i,t^n)\Delta t)$:
    \begin{align}\label{moment lemma tilde f expand 1}
        \begin{split}
            \tilde{f}(x_i,v,t^n) &=f(x_i-v\Delta t,v-\mathbb{E}(x_i,t^n)\Delta t,t^n)\cr
            &=f(x_i,v,t^n) - v\Delta t\partial_x f(x_i,v,t^n) -\mathbb{E}(x_i,t^n)\Delta t \partial_v f(x_i,v,t^n)\cr
            &+\frac{(v\Delta t)^2}{2}\partial^2_x f(x_{i,1},v-\mathbb{E}(x_i,t^n)\Delta t,t^n)+ \frac{v\mathbb{E}(x_i,t^n)(\Delta t)^2}{2}\partial_v\partial_x f(x_i,v_{1},t^n)\cr 
            &+\frac{(\mathbb{E}(x_i,t^n)\Delta t)^2}{2}\partial^2_v f(x_i,v_{2},t^n).
        \end{split}
    \end{align}
    Applying Taylor's formula again, for some $x_{i,2}\in(x_{i-1},x_i)$ and $v_{j,1}\in(v_j,v_{j+1})$, we get
    \begin{align*}
        \begin{split}
            \partial_x f(x_i,v_j,t^n)&=\frac{f(x_i,v_j,t^n)-f(x_{i-1},v_j,t^n)}{\Delta x}+\frac{\Delta x}{2}\partial^2_x f(x_{i,2},v_j,t^n),\cr
            \partial_v f(x_i,v_j,t^n)&=\frac{f(x_i,v_{j+1},t^n)-f(x_i,v_{j},t^n)}{\Delta x}-\frac{\Delta v}{2}\partial^2_v f(x_i,v_{j,1},t^n).
        \end{split}
    \end{align*}    
    By the above formula, we rewrite \eqref{moment lemma tilde f expand 1} as follows:
    \begin{align}\label{moment lemma tilde f expand 2}
        \begin{split}
            &\tilde{f}(x_i,v,t^n)\cr
            &\quad=\left(1-\frac{\Delta t}{\Delta x}|v_j|-\frac{\Delta t}{\Delta v}|\mathbb{E}^n_i| \right)f(x_i,v_j,t^n)+\frac{\Delta t}{\Delta x}|v_j|f(x_{i-1},v_j,t^n) + \frac{\Delta t}{\Delta v}|\mathbb{E}^n_i|f(x_i,v_{j+1},t^n)\cr
            &\quad+I_1+I_2+I_3,
        \end{split}
    \end{align}
    where $I_1$, $I_2$, and $I_3$ are given by
    \begin{align}\label{moment lemma remainder}
        \begin{split}
            I_1&=(f(x_i,v,t^n)-f(x_i,v_j,t^n))+v\Delta t(\partial_x f(x_i,v,t^n)-\partial_x f(x_i,v_j,t^n)))\cr
            &+\mathbb{E}(x_i,t^n)\Delta t(\partial_v f(x_i,v,t^n)-\partial_v f(x_i,v_j,t^n))\cr
            &+\partial_x f(x_i,v_j,t^n)\Delta t(v_j-v)+\partial_v f(x_i,v_j,t^n)\Delta t(\mathbb{E}^n_i-\mathbb{E}(x_i,t^n)),\cr
            I_2&=-\frac{\Delta x\Delta t}{2}v \partial^2_x f(x_{i,2},v_j,t^n)+\frac{\Delta v\Delta t}{2}\mathbb{E}^n_i\partial^2_v f(x_i,v_{j,1},t^n),\cr
            I_3&=\frac{(v\Delta t)^2}{2}\partial^2_x f(x_{i,1},v-\mathbb{E}(x_i,t^n)\Delta t,t^n)+ \frac{v\mathbb{E}(x_i,t^n)(\Delta t)^2}{2}\partial_v\partial_x f(x_i,v_{1},t^n)\cr 
            &+\frac{(\mathbb{E}(x_i,t^n)\Delta t)^2}{2}\partial^2_v f(x_i,v_{2},t^n).
        \end{split}
    \end{align}    
    Multiplying $\varphi(v_j)$ to \eqref{moment lemma tilde f expand 2} and integrating over $\Delta_j$ with respect to $v$, we obtain
    \begin{align*}
    \begin{split}
        &\int_{\Delta_j}\tilde{f}(x_i,v,t^n)\varphi(v_j)dv\cr
        &=\left\{\left(1-\frac{\Delta t}{\Delta x}|v_j|-\frac{\Delta t}{\Delta v}|\mathbb{E}^{n}_{i}|\right)f(x_i,v_j,t^n)+\frac{\Delta t}{\Delta x}|v_j| f(x_{i-1},v_j,t^n)\right\}\varphi(v_j)\Delta v\cr
        &+ \frac{\Delta t}{\Delta v}|\mathbb{E}^{n}_{i}|f(x_i,v_{j-1},t^n)\varphi(v_j)\Delta v\cr
        &+\int_{\Delta_j}(I_1 + I_2+I_3)\varphi(v_j)dv,
    \end{split}
    \end{align*}
   where we used $\int_{\Delta_j}dv=\Delta v$. Recalling the update rule for $\tilde{f}^n_{i,j}$ in \eqref{convex_combi}, we have
    \begin{align}\label{moment lemma sum}
        \begin{split}
            &\tilde{f}^n_{i,j}\varphi(v_j)\Delta v\cr
            &=\left(1-\frac{\Delta t}{\Delta x}|v_j|-\frac{\Delta t}{\Delta v}|\mathbb{E}^{n}_{i}|\right)f^n_{i,j}\varphi(v_j)\Delta v+\frac{\Delta t}{\Delta x}|v_j| f^n_{i-1,j}\varphi(v_j)\Delta v+ \frac{\Delta t}{\Delta v}|\mathbb{E}^{n}_{i}|f^n_{i,j-1}\varphi(v_j)\Delta v.
        \end{split}
    \end{align}
    Subtracting \eqref{moment lemma sum} from the moments of the continuous $\tilde{f}$ and summing over $j\in[0,N_v-1]$, we get 
    \begin{align}\label{moment lemma I build}
        \begin{split}
            &\left|\sum_{0\leq j\leq N_v-1}\int_{\Delta _j}\tilde{f}(x_i,v,t^n)\varphi(v_j)dv-\sum_{0\leq j\leq N_v-1} \tilde{f}^{n}_{i,j}\varphi(v_j)\Delta v\right|\cr 
            &\qquad\leq \sum_{0\leq j\leq N_v-1}\left(1-\frac{\Delta t}{\Delta x}|v_j| -\frac{\Delta t}{\Delta v}|\mathbb{E}^{n}_{i}|\right)|f(x_i,v_j,t^n)-f^n_{i,j}||\varphi(v_j)|\Delta v\cr 
            &\qquad+\sum_{0\leq j\leq N_v-1}\frac{\Delta t}{\Delta x}|v_j| |f(x_{i-1},v_j,t^n)-f^n_{i-1,j}||\varphi(v_j)|\Delta v\cr
            &\qquad+\sum_{0\leq j\leq N_v-1}\frac{\Delta t}{\Delta v}|\mathbb{E}^{n}_{i}| |f(x_i,v_{j+1},t^n)-f^n_{i,j+1}||\varphi(v_j)|\Delta v\cr
            &\qquad+\sum_{0\leq j\leq N_v-1}\int_{\Delta_j}(I_1 + I_2+I_3)\varphi(v_j)dv.
        \end{split}
    \end{align}
    Then, the first three terms on the right-hand side are controlled by
    \begin{align}\label{moment lemma first three term}
        \begin{split}
            &(\text{first three terms})\cr 
            &\quad\leq \|f(t^n)-f^n\|_{L^{\infty}_q}\sum_{0\leq j\leq N_v-1}\left\{\left(1-\frac{\Delta t}{\Delta x}|v_j| -\frac{\Delta t}{\Delta v}|\mathbb{E}^{n}_{i}|\right)+\frac{\Delta t}{\Delta x}|v_j|\right\}\frac{|\varphi(v_j)|\Delta v}{(1+|v_j|)^q}\cr
            &\quad +  \|f(t^n)-f^n\|_{L^{\infty}_q}\sum_{0\leq j\leq N_v}\frac{\Delta t}{\Delta v}|\mathbb{E}^{n}_{i}|\left(\frac{1+|v_j|}{1+|v_{j+1}|} \right)^q\frac{|\varphi(v_j)|\Delta v}{(1+|v_j|)^q}\cr
            &\quad\leq \|f(t^n)-f^n\|_{L^{\infty}_q}\sum_{0\leq j\leq N_v-1}\left(1+C_{\mathbb{E},2}\Delta t\right)\frac{|\varphi(v_j)|\Delta v}{(1+|v_j|)^q}\cr
            &\quad\leq 2\bar{D}_{q-2}(1+C_{\mathbb{E},2}\Delta t)\|f(t^n)-f^n\|_{L^{\infty}_q}.
        \end{split}
    \end{align}
    In the third line and the last line, we used
    \begin{align*}
        \begin{split}
            \left(\frac{1+|v_j|}{1+|v_{j-1}|} \right)^q\frac{\Delta t}{\Delta v}\mathbb{E}^{n}_{i}\leq (1+C_q\Delta v)\frac{\Delta t}{\Delta v}\mathbb{E}^{n}_{i}\leq \left(\frac{\Delta t}{\Delta v}\mathbb{E}^{n}_{i}+C_{\mathbb{E},2}\Delta t\right),
        \end{split}
    \end{align*}
    and
    \begin{align*}
        \begin{split}
            \sum_{0\leq j\leq N_v-1}\frac{|\varphi(v_j)|\Delta v}{(1+|v_j|)^q}\leq \sum_j\frac{|\varphi(v_j)|\Delta v}{(1+|v_j|)^q}\leq 2\bar{D}_{q-2},
        \end{split}
    \end{align*}
    where $\bar{D}_{q-2}$ is defined in Definition \ref{Constant series}. Now we consider the estimates for $I_1$, $I_2$, and $I_3$.\\
    \textbf{(1-1) Estimate for $I_1$}\\
    Applying the mean value theorem, for some $v_{j,2}$, $v_{j,3}$, and $v_{j,4}$ belong to $(v,v_j)$, we obtain 
    \begin{align*}
        \begin{split}
            I_1 &\leq \Delta v|\partial_v f(x_i,v_{j,2},t^n)|+\Delta tv_j|\partial_v\partial_x f(x_i,v_{j,3},t^n)|+\|\mathbb{E}(t^n)\|_{L^{\infty}_x}\Delta t|\partial^2_v f(x_i,v_{j,4},t^n)|\cr
            &+\Delta v\Delta t|\partial_x f(x_i,v_j,t^n)|+\|\mathbb{E}(t^n)-\mathbb{E}^n_i\|_{L^{\infty}_x}\Delta t |\partial_v f(x_i,v_j,t^n)|.
        \end{split}
    \end{align*}
   Then, the first four terms are bounded by
    \begin{align*}
        \begin{split}
           \left(\|f(t^n)\|_{W^{1,\infty}_q}\Delta v+\|f(t^n)\|_{W^{2,\infty}_{q+1}}\Delta t+\|\mathbb{E}(t^n)\|_{L^{\infty}_x}\|f(t^n)\|_{W^{2,\infty}_q}\Delta t+\|f(t^n)\|_{W^{1,\infty}_q}\Delta v\Delta t\right)\frac{1}{(1+|v_j|)^q}.
        \end{split}
    \end{align*}
    In this process, we used
    \begin{align*}
        \begin{split}
            \left( \frac{1+|v_j|}{1+|v_{j,m}|}\right)^q\leq \left(1+\frac{|v_j|-|v_{j,m}|}{1+|v_{j,m}|}\right)^q\leq (1+\Delta v)^q\leq C_q,\text{ for }m=2,3,4.
        \end{split}
    \end{align*}
    The estimate for the last term is derived from Lemma \ref{difference field}:
    \begin{align*}
        \begin{split}
            &\|\mathbb{E}(t^n)-\mathbb{E}^n_i\|_{L^{\infty}_x}\Delta t |\partial_v f(x_i,v_j,t^n)|\cr
            &\quad\leq\left\{ C_q\|f(t^n)\|_{W^{1,\infty}_q}\Delta t\|f(t^n)-f^n\|_{L^{\infty}_q}+C_{q,f_0,T^f}\|f(t^n)\|_{W^{1,\infty}_q}\left(\Delta x\Delta t+\Delta v\Delta t+\frac{\Delta t}{V^{q-1}_M} \right)\right\}\cr
            &\qquad\times\frac{1}{(1+|v_j|)^q}.
        \end{split}
    \end{align*}
    As a result, $I_1$ satisfies
    \begin{align}\label{moment lemma I_1 est}
        \begin{split}
            I_1&\leq C_{q,f_0,T^f}\left\{\Delta t\|f(t^n)-f^n\|_{L^{\infty}_q}+\Delta x\Delta t+\Delta v+\Delta t +\frac{\Delta t}{V^{q-1}_M} \right\}\frac{1}{(1+|v_j|)^q}.
        \end{split}
    \end{align}
    \textbf{(1-2) Estimate for $I_2$}\\
    From \eqref{moment lemma remainder}, we directly compute the estimate for $I_2$:
    \begin{align}\label{moment lemma I_2 est}
        \begin{split}
            I_2 &\leq \frac{\Delta x\Delta v}{2}v_j|\partial^2_xf(x_{i,2},v_j,t^n)|+\frac{\Delta v\Delta t}{2}\|\mathbb{E}^n_i\|_{L^{\infty}_x}|\partial^2_v f(x_i,v_{j,1},t^n)|\cr
            &\leq \left\{\frac{\Delta x\Delta v}{2}\|f(t^n)\|_{W^{2,\infty}_{q+1}}+\frac{\Delta v\Delta t}{2}C_{\mathbb{E},1}\|f(t^n)\|_{W^{2,\infty}_{q}}\left(\frac{1+|v_j|}{1+|v_{j,1}|} \right)^q\right\}\frac{1}{(1+|v_j|)^q}\cr
            &\leq C_{q,f_0,T^f}(\Delta x\Delta t+\Delta v\Delta t)\frac{1}{(1+|v_j|)^q}.
        \end{split}
    \end{align}
    \textbf{(1-3) Estimate for $I_3$}\\
    By \eqref{moment lemma remainder}, we present the estimate for $I_3$ as follows:
    \begin{align}\label{moment lemma I_3 est}
        \begin{split}
            I_3 &\leq \frac{(\Delta t)^2}{2}\|f(t^n)\|_{W^{2,\infty}_{q+2}}\left(\frac{1+|v_j|}{1+|v-\mathbb{E}(x_i,t^n)\Delta t|} \right)^q\frac{1}{(1+|v_j|)^q}\cr  
            &+ \frac{(\Delta t)^2}{2}\|\mathbb{E}(t^n)\|_{L^{\infty}_x}\|f(t^n)\|_{W^{2,\infty}_{q+1}}\left(\frac{1+|v_j|}{1+|v_1|} \right)^q \frac{1}{(1+|v_j|)^q}\cr
            &+\frac{(\Delta t)^2}{2}\|\mathbb{E}(t^n)\|^2_{L^{\infty}_x}\|f(t^n)\|_{W^{2,\infty}_{q}}\left( \frac{1+|v_j|}{1+|v_2|}\right)^q\frac{1}{(1+|v_j|)^q}\cr
            &\leq C_{q,f_0,T^f} (\Delta t)^2\frac{1}{(1+|v_j|)^q},
        \end{split}
    \end{align}
    where we used the smallness condition on $\Delta v$ and $\Delta t$, which is $\Delta v<1$ and $\Delta t<1$,
    \begin{align*}
        \begin{split}
            \left(\frac{1+|v_j|}{1+|v-\mathbb{E}(x_i,t^n)\Delta t|} \right)^q&\leq (1+|v_j-v-\mathbb{E}(x_i,t^n)\Delta t|)^q\cr
            &\leq (1+|\Delta v + C_{q,f_0,T^f}\Delta t|)^q\cr
            &\leq C_{q,f_0,T^f}.
        \end{split}
    \end{align*}
    Consequently, we combine \eqref{moment lemma I_1 est}, \eqref{moment lemma I_2 est}, and $\eqref{moment lemma I_3 est}$ to deduce
    \begin{align}\label{moment lemma I remainder}
        \begin{split}
            &\sum_{0\leq j\leq N_v-1}\int_{\Delta_j}(I_1 + I_2 +I_3)|\varphi(v_j)|dv\cr
            &\leq C_{q,f_0,T^f}\left\{\Delta t\|f(t^n)-f^n\|_{L^{\infty}_q}+  \Delta x\Delta t + \Delta v + \Delta t+\frac{\Delta t}{V^{q-1}_M} \right\}\sum_{0\leq j\leq N_v-1}\frac{|\varphi(v_j)|\Delta v}{(1+|v_j|)^q}\cr
            &\leq 2C_{q,f_0,T^f}\bar{D}_{q-2}\left\{\Delta t\|f(t^n)-f^n\|_{L^{\infty}_q}+  \Delta x\Delta t + \Delta v + \Delta t+\frac{\Delta t}{V^{q-1}_M} \right\}.
        \end{split}
    \end{align}
    Plugging \eqref{moment lemma first three term} and \eqref{moment lemma I remainder} into \eqref{moment lemma I build}, we observe the estimate for $I$ as follows:
    \begin{align}\label{moment lemma I est}
        \begin{split}
        I &\leq C_{q,f_0}\|f^{n}-f(t^n)\|_{L^{\infty}_{q}} + C_{q,f_0}\left\{\Delta x \Delta t + \Delta v + \Delta t+\frac{\Delta t}{V^{q-1}_M}  \right\}.
        \end{split}
    \end{align}
    \textbf{(2) Estimates for $\boldsymbol{II}$, $\boldsymbol{III}$, and $\boldsymbol{\delta(x,t)}$ }\\
    We now consider the estimate of $II$. Since, for $\beta\in(0,1)$, $|v_j|=|v+\beta\Delta v|\leq 1+|v|$ in $\Delta_j$, for $v\in\Delta_j$, we have
    \begin{align*}
        \begin{split}
            |\varphi(v) - \varphi(v_j)| \leq C|v-v_j|(|v|^{p}+|v_j|^{p}) \leq C\Delta v (1+|v|)^{p},\;\text{ for }p=0,1,2.
        \end{split}
    \end{align*}
    Then, we obtain
    \begin{align}\label{moment lemma II est}
    \begin{split}
        II &= \sum_{0\leq j\leq N_v-1} \int_{\Delta_j}\tilde{f}(x_i,v,t^n)(\varphi(v_j) - \varphi(v))dv \cr
        &\leq C\Delta v\sum_{0\leq j\leq N_v-1} \int_{\Delta_j}\tilde{f}(x_i,v,t^n)(1+|v|)^{p}dv \cr
        &= C\Delta v\sum_{0\leq j\leq N_v-1} \int_{\Delta_j}f(x_i-v\Delta t,v-\mathbb{E}(x_i,t^n)\Delta t,t^n)(1+|v|)^{p}dv\cr
        &\leq C\Delta v\|f(t^n)\|_{L^{\infty}_q}\sum_{0\leq j\leq N_v-1}\int_{\Delta_j}\left(\frac{1+|v|}{1+|v-\mathbb{E}(x_i,t^n)\Delta t|}\right)^q\frac{dv}{(1+|v|)^{q-p}}\cr
        &\leq C_{q,f_0,T^f}\Delta v,
        \end{split}
    \end{align}
    where we used
    \begin{align*}
        \begin{split}
            \left(\frac{1+|v|}{1+|v-\mathbb{E}(x_i,t^n)\Delta t|}\right)^q=\left(1+\frac{|v|-|v-\mathbb{E}(x_i,t^n)\Delta t|}{1+|v-\mathbb{E}(x_i,t^n)|\Delta t} \right)^q\leq (1+|\mathbb{E}(x_i,t^n)\Delta t|)^q\leq C_{q,f_0,T^f},
        \end{split}
    \end{align*}
    and
    \begin{align*}
        \begin{split}
            \sum_{0\leq j\leq N_v-1}\int_{\Delta_j}\frac{dv}{(1+|v|)^{q-p}}\leq \int_{\mathbb{R}}\frac{dv}{(1+|v|)^{q-p}}\leq C_{q}.
        \end{split}
    \end{align*}
    On the other hand, the estimate for $III$ is obtained by
    \begin{align}\label{moment lemma III est}
        \begin{split}
            III\leq \|\tilde{f}^n\|_{L^{\infty}_q}\Delta v\leq C_{q,f_0,T^f}\Delta v.
        \end{split}
    \end{align}
    Lastly, for $p=0,1,2$, $\delta$ is controlled as follows:
    \begin{align}\label{moment lemma delta est}
        \begin{split}
            \delta(x,t) &\leq \|f(t^n)\|_{L^{\infty}_{q}}\int_{|v|\geq V_M} \left(\frac{1+|v|}{1+|v-\mathbb{E}(x_i,t^n)\Delta t|} \right)^q\frac{|\varphi(v)|}{(1+|v|)^q}dv\cr
            &\leq C_{q,f_0,T^f}\int^{\infty}_{V_M}v^{p-q}dv\cr
            &\leq \frac{C_{q,f_0,T^f}}{V_M^{q-3}}.
        \end{split}
    \end{align}
    In conclusion, we obtain the desired result from \eqref{moment lemma I est}, \eqref{moment lemma II est}, \eqref{moment lemma III est}, and \eqref{moment lemma delta est}.
\end{proof}

Based on Lemma \ref{Moment diff}, we establish the error estimate for the macroscopic fields and the local Maxwellian. 
\begin{lemma}\label{tilde moment difference}\textnormal{(Error estimate for the macroscopic fields)}
      Let $f(x,v,t)$ and $f^{n}_{i,j}$ be the solution of (\ref{VPBGK model}) and (\ref{VPBGK_Scheme}).Then, under the assumption of Theorem \ref{Main Thm}, we have
    \begin{align*}
        \begin{split}
            |\tilde{\rho}(x_{i},t^n)-\tilde{\rho}^n_i |,\;|\tilde{U}(x_{i},t^n)-\tilde{U}^n_i |,\;|\tilde{T}(x_{i},t^n)-\tilde{T}^n_i|&\leq C_{q,f_0,T^f,\varepsilon}\|f^n-f(t^n)\|_{L^{\infty}_{q}}\cr         
            &+C_{q,f_0,T^f,\varepsilon}\left\{\Delta x\Delta t +\Delta v + \Delta t + \frac{1}{V_M^{q-3}} \right\}.
        \end{split}
    \end{align*}
\end{lemma}
\begin{proof}
    By the definition of continuous and discrete density, we get
    \begin{align*}
         \tilde{\rho}(x_i,t^n)  - \tilde{\rho}^n_i =\int_{\mathbb{R}}\tilde{f}(x_{i},v,t^n)dv- \sum_{j}\tilde{f}^{n}_{i,j}\Delta v.
    \end{align*}
    Then, Lemma \ref{Moment diff} gives the first assertion directly. To derive $\tilde{U}(x_{i},t^n)-\tilde{U}^n_i$, we have
    \begin{align*}
        \begin{split}
            |\tilde{U}(x_i,t^n)-\tilde{U}^n_i| &= \frac{1}{\tilde{\rho}^n_i}\{\tilde{\rho}(x_i,t^n)\tilde{U}(x_i,t^n)-\tilde{\rho}^n_i \tilde{U}^n_i  \} - \frac{\tilde{U}(x_i,t^n)}{\tilde{\rho}^n_i}\{\tilde{\rho}(x_i,t^n)-\tilde{\rho}^n_i \}\cr
            &\leq \frac{1}{\tilde{\rho}^n_i}|\tilde{\rho}(x_i,t^n)\tilde{U}(x_i,t^n)-\tilde{\rho}^n_i \tilde{U}^n_i| + \frac{|\tilde{U}(x_i,t^n)|}{\tilde{\rho}^n_i}|\tilde{\rho}(x_i,t^n)-\tilde{\rho}^n_i|.
        \end{split}
    \end{align*}
   Combining Lemma \ref{A5} and Lemma \ref{continuous tilde moment bound}, we get
    \begin{align*}
        \begin{split}
            &|\tilde{U}(x_i,t^n)-\tilde{U}^n_i|\cr 
            &\quad\leq C_{q,f_0,T^f,\varepsilon}\{|\tilde{\rho}(x_i,t^n)\tilde{U}(x_i,t^n)-\tilde{\rho}^n_i \tilde{U}^n_i| + |\tilde{\rho}(x_i,t^n)-\tilde{\rho}^n_i|\}\cr
            &\quad\leq  C_{q,f_0,T^f,\varepsilon}\left|\int_{\mathbb{R}}\tilde{f}(x_{i},v,t^n)(1+|v|)dv-\sum_{j}\tilde{f}^{n}_{i,j}(1+|v|)\Delta v\right|.
        \end{split}
    \end{align*}
    Then, this leads to the desired result. Lastly, for $\tilde{T}(x_{i},t^n)-\tilde{T}^n_i$, we have
    \begin{align*}
        \begin{split}
            \tilde{T}(x_i,t^n)-\tilde{T}^n_i &=\frac{1}{\tilde{\rho}^n_i}\{(\tilde{\rho}(x_i,t^n)\tilde{T}(x_i,t^n)+\tilde{\rho}(x_i,t^n)|\tilde{U}(x_i,t^n)|^2)-(\tilde{\rho}^n_i\tilde{T}^n_i + \tilde{\rho}^n_i|\tilde{U}^n_i|^2) \}\cr
            &+ \frac{\tilde{\rho}(x_i,t^n)\tilde{T}(x_i,t^n)+\tilde{\rho}(x_i,t^n)|\tilde{U}(x_i,t^n)|^2}{\tilde{\rho}^n_i\tilde{\rho}(x_i,t^n)}(\tilde{\rho}(x_i,t^n)-\tilde{\rho}^n_i ) \cr
            &+\frac{\tilde{\rho}^n_i\tilde{U}^n_i + \tilde{\rho}(x_i,t^n)\tilde{U}(x_i,t^n)}{\tilde{\rho}^n_i}(\tilde{\rho}(x_i,t^n)\tilde{U}(x_i,t^n)-\tilde{\rho}^n_i\tilde{U}^n_i)\cr 
            &- \frac{\tilde{\rho}^n_i + \tilde{\rho}(x_i,t^n)}{\{\tilde{\rho}^n_i\tilde{\rho}(x_i,t^n)\}^2}\{\tilde{\rho}(x_i,t^n)\tilde{U}(x_i,t^n)\}^2(\tilde{\rho}(x_i,t^n)-\tilde{\rho}^n_i).
        \end{split}
    \end{align*}
    Therefore, Lemma \ref{A5}, Lemma \ref{A6}, Lemma \ref{continuous tilde moment bound}, and Lemma \ref{Moment diff} give
    \begin{align*}
        \begin{split}
            &|\tilde{T}(x_i,t^n)-\tilde{T}^n_i|\cr
            &\qquad\leq C_{q,f_0,T^f,\varepsilon}\left\{|\tilde{\rho}(x_i,t^n)-\tilde{\rho}^n_i|+|\tilde{\rho}(x_i,t^n)\tilde{U}(x_i,t^n)-\tilde{\rho}^n_i\tilde{U}^n_i|+|\tilde{\rho}(x_i,t^n)\tilde{T}(x_i,t^n)-\tilde{\rho^n_i}\tilde{T}^n_i| \right\}\cr
            &\qquad\leq C_{q,f_0,T^f,\varepsilon}\left|\int_{\mathbb{R}}\tilde{f}(x_{i},v,t^n)(1+|v|^2)dv-\sum_{j}\tilde{f}^{n}_{i,j}(1+|v|^2)\Delta v\right|\cr
            &\qquad\leq C_{q,f_0,T^f,\varepsilon}\|f^n-f(t^n)\|_{L^{\infty}_{q}}+C_{q,f_0,\varepsilon}\left\{\Delta x\Delta t +\Delta v + \Delta t + \frac{1}{V_M^{q-3}} \right\}.
        \end{split}
    \end{align*}
\end{proof}

\begin{lemma}\label{Maxewllian estimate}\textnormal{(Error estimate for the local Maxwellian)}
    Let $f(x,v,t)$ and $f^{n}_{i,j}$ be the solution of (\ref{VPBGK model}) and (\ref{VPBGK_Scheme}). Then, under the assumption of Theorem \ref{Main Thm}, we have
    \begin{align*}
        \begin{split}
            \|\mathcal{M}(\tilde{f}(t^n))-\mathcal{M}(\tilde{f}^{n})\|_{L^{\infty}_{q}}
            &\leq C_{q,f_0,T^f,\varepsilon}\|f^n-f(t^n)\|_{L^{\infty}_{q}}\cr         
            &+C_{q,f_0,T^f,\varepsilon}\left\{\Delta x\Delta t +\Delta v + \Delta t + \frac{1}{V_M^{q-3}} \right\}.
        \end{split}
    \end{align*}
\end{lemma}
\begin{proof}
    For $\kappa\in[0,1]$, we define $(\tilde{\rho}^n_{i,\kappa},\tilde{U}^n_{i,\kappa},\tilde{T}^n_{i,\kappa})$ as 
    \begin{align*}
        \begin{split}
            (\tilde{\rho}^n_{i,\kappa},\tilde{U}^n_{i,\kappa},\tilde{T}^n_{i,\kappa}) = (1-\kappa)(\tilde{\rho}(x_i,t^n),\tilde{U}(x_i,t^n),\tilde{T}(x_i,t^n)) + \kappa (\tilde{\rho}^n_{i},\tilde{U}^n_{i},\tilde{T}^n_{i}).
        \end{split}
    \end{align*}
    Applying Taylor's theorem, we have
    \begin{align*}
        \begin{split}
            \mathcal{M}(\tilde{f}(x_i,v_j,t^n)) - \mathcal{M}(\tilde{f}^n_{i,j}) &= \mathcal{M}(\tilde{\rho}(x_i,t^n),\tilde{U}(x_i,t^n),\tilde{T}(x_i,t^n))(v_j) - \mathcal{M}(\tilde{\rho}^n_i,\tilde{U}^n_i,\tilde{T}^n_i)(v_j)\cr
            &= \{\tilde{\rho}(x_i,t^n) - \tilde{\rho}^{n}_{i} \}\int^1_0 \frac{\partial \mathcal{M}(\kappa)}{\partial \rho}d\kappa \cr
            &+ \{\tilde{U}(x_i,t^n) - \tilde{U}^{n}_{i} \}\int^1_0 \frac{\partial \mathcal{M}(\kappa)}{\partial U}d\kappa\cr &+ \{\tilde{T}(x_i,t^n) - \tilde{T}^{n}_{i} \}\int^1_0 \frac{\partial \mathcal{M}(\kappa)}{\partial T}d\kappa.
        \end{split}
    \end{align*}
    By Lemma \ref{A5}, Lemma \ref{A6}, and Lemma \ref{continuous tilde moment bound}, $(\tilde{\rho}^n_{i,\kappa},\tilde{U}^n_{i,\kappa},\tilde{T}^n_{i,\kappa})$ has uniform lower and upper bounds. Thus, we obtain
    \begin{align*}
        \begin{split}
            \mathcal{M}(\kappa)(v_j) = \frac{\tilde{\rho}^n_{i,\kappa}}{(2\pi \tilde{T}^n_{i,\kappa})^{\frac{1}{2}}} e^{-\frac{|v_j-\tilde{U}^n_{i,\kappa}|^2}{2\tilde{T}^n_{i,\kappa}}}\leq C e^{-C|v_j - \tilde{U}^n_{i,\kappa}|^2}\leq Ce^{-C|v_j|^2}.
        \end{split}
    \end{align*}
    Therefore, by direct calculation, we get  
    \begin{align*}
        \begin{split}
        &\left|\int^1_0 \frac{\partial \mathcal{M}(\theta)}{\partial \rho^n_{i,\kappa}} \right| = \left|\int^1_0 \frac{1}{\tilde{\rho}^n_{i,\kappa}}\mathcal{M}(\theta) \right| \leq Ce^{-C|v_j|^2},\cr
        &\left|\int^1_0 \frac{\partial \mathcal{M}(\theta)}{\partial U^n_{i,\kappa}} \right| = \left|\int^1_0 \frac{|v_j - \tilde{U}^n_{i,\kappa}|}{2\tilde{T}^n_{i,\kappa}}\mathcal{M}(\theta) \right|\leq C(1+|v_j|)^2 e^{-C|v_j|^2},\cr
        &\left|\int^1_0 \frac{\partial \mathcal{M}(\theta)}{\partial T^n_{i,\kappa}} \right| = \left|\int^1_0 \frac{|\tilde{T}^n_{i,\kappa}|+|v_j - \tilde{U}^n_{i,\kappa}|^2}{2\tilde{T}^n_{i,\kappa}}\mathcal{M}(\theta) \right|\leq C(1+|v_j|)^2 e^{-C|v_j|^2}.
        \end{split}
    \end{align*}
    Using $x^2e^{-nx}\leq C$ for some $C>0$, we have
    \begin{align*}
        \begin{split}
            &|\mathcal{M}(\tilde{f}(x_i,v_j,t^n)) - \mathcal{M}(\tilde{f}^n_{i,j})|\cr 
            &\quad \leq C\{|\tilde{\rho}(x_i,t^n) - \tilde{\rho}^n_i| + |\tilde{U}(x_i,t^n) - \tilde{U}^n_i| + |\tilde{T}(x_i,t^n) - \tilde{T}^n_i| \}(1+|v_j|)^2 e^{-C|v_j|^2}\cr
            &\quad\leq C\{|\tilde{\rho}(x_i,t^n) - \tilde{\rho}^n_i| + |\tilde{U}(x_i,t^n) - \tilde{U}^n_i| + |\tilde{T}(x_i,t^n) - \tilde{T}^n_i| \},
        \end{split}
    \end{align*}
    where $C$ is the constant that depends on $q$, $T^f$, $f_0$, and $\varepsilon$. Then, Lemma \ref{tilde moment difference} completes the proof.
\end{proof}

\section{Proof of the Main Theorem}\label{sec6}
In this section, we present the proof of Theorem \ref{Main Thm}. We note that the constant $C$ appearing in the proof depends on $q$, $T^f$, $f_0$, $\varepsilon_0$, and is not necessarily the same in each line. By Theorem \ref{consistent theorem}, we have
\begin{align*}
    \begin{split}
        &\|f(t^{n+1})-f^{n+1}\|_{L^{\infty}_{q}}\cr
        &\quad\leq \frac{\varepsilon}{\varepsilon + \Delta t}\|\tilde{f}(t^n)-\tilde{f}^{n} \|_{L^{\infty}_{q}}+ \frac{\Delta t}{\varepsilon + \Delta t}\|\mathcal{M}(\tilde{f}(t^n)) - \mathcal{M}(\tilde{f}^{n})\|_{L^{\infty}_{q}}\cr 
        &\quad+ \frac{1}{\varepsilon + \Delta t}\{\|R_1\|_{L^{\infty}_{q}} + \|R_2\|_{L^{\infty}_{q}}+\|R_3\|_{L^{\infty}_{q}}\}. 
    \end{split}
\end{align*}
According to the error estimates in Section \ref{sec5}, we obtain 
\begin{align}\label{iteration}
    \begin{split}
        \|f(t^{n+1}) - f^{n+1}\|_{L^{\infty}_{q}} \leq \left(1+\frac{C\Delta t}{\varepsilon + \Delta t} \right)\|f(t^n) -f^{n}\|_{L^{\infty}_{q}} + CP(\Delta x,\Delta v,\Delta t),
    \end{split}
\end{align}
where $P(\Delta x,\Delta v,\Delta t)$ is given by
\begin{align*}
    \begin{split}
        P(\Delta x,\Delta v,\Delta t) = \frac{\Delta t}{\varepsilon+\Delta t}\left\{\Delta x+\Delta v+\Delta t + \frac{1}{{V}^{q-3}_M} \right\}.
    \end{split}
\end{align*}
Put $\Gamma = \frac{C\Delta t}{\varepsilon + \Delta t}$, and we iterate \eqref{iteration} to get 
\begin{align*}
    \begin{split}
        \|f(N_t \Delta t) - f^{N_t}\|_{L^{\infty}_{q}} \leq (1+\Gamma)^{N_{t}}\|f_{0} - f^{0}\|_{L^{\infty}_{q}} + C\sum^{N_t}_{i=1} (1+\Gamma)^{i-1}P.
    \end{split}
\end{align*}
By the definition of initial iteration; $f^{0}_{i,j} = f_{0}(x_i,v_j)$, we have
\begin{align*}
    \|f^{0} - f_{0} \|_{L^{\infty}_{q}} = 0.
\end{align*}
On the other hand, applying $(1+x)^{n} \leq e^{nx}$ and $N_t\Delta t=T^f$, we derive
\begin{align*}
    \begin{split}
        (1+\Gamma)^{N_t} \leq e^{N_{t} \Gamma} \leq e^{\frac{CN_{t}\Delta t}{\varepsilon+\Delta t}} = e^{\frac{CT^{f}}{\varepsilon + \Delta t}}.
    \end{split}
\end{align*}
Thus, we have
\begin{align*}
    \begin{split}
        \sum^{N_t}_{i=1}(1+\Gamma)^{i-1} = \frac{(1+\Gamma)^{N_{t}}-1}{(1+\Gamma)-1} \leq C(e^{\frac{CT^{f}}{\varepsilon + \Delta t}} -1 )\left(\frac{\varepsilon + \Delta t}{\Delta t}\right).
    \end{split}
\end{align*}
This gives
\begin{align*}
    \begin{split}
        \sum^{N_t}_{i=1}(1+\Gamma)^{i-1}P \leq C(e^{\frac{CT^{f}}{\varepsilon + \Delta t}} -1 )\left\{\Delta x + \Delta v + \Delta t +\frac{1}{V_M^{q-3}}\right\}.
    \end{split}
\end{align*}
Therefore, we conclude
\begin{align*}
    \begin{split}
        \|f(T^f)- f^{N_t}\|_{L^{\infty}_{q}} \leq C\left\{\Delta x + \Delta v + \Delta t + \frac{1}{V_M^{q-3}}\right\}.
    \end{split}
\end{align*}
Moreover, by Lemma \ref{difference field}, we deduce that
\begin{align*}
    \begin{split}
        \|\mathbb{E}(T^f)- \mathbb{E}^{N_t}\|_{L^{\infty}_{x}} \leq C\left\{\Delta x + \Delta v + \Delta t + \frac{1}{V_M^{q-3}}\right\}.
    \end{split}
\end{align*}
This completes the proof.


 \section{Numerical test}
 In this section, we verify the convergence estimates in Theorem \ref{Main Thm} by checking the order of convergence of the proposed method. For the test, we consider uniform grids in space and velocity, and take smooth non-equilibrium initial data that satisfies the lower bound condition in Theorem \ref{analysis}:
 \begin{align}
 f_0(x,v)=\frac{1}{\sqrt{2\pi}} (1 +0.01\cos(2\pi x))\exp\left(-v^2/2\right),
 \end{align}
 where the periodic boundary condition is assumed on the spatial domain $[0,1]$, and the velocity domain is truncated by $[-15,15]$. We compute numerical solutions up to $t=0.4$ for which the solutions remain smooth even for small Knudsen numbers. We use time steps that correspond to CFL$=0.9$.
 In Table \ref{tab1}, we report the numerical errors and the corresponding order of convergence for $\|f\|_{L_{q}^\infty}$ with $q=4,5$ and $\|\mathbb{E}\|_{L_x^\infty}$ in a wide range of Knudsen numbers. 
\begin{table}[htbp]
	\centering	
		\begin{tabular}{ |p{3cm}||p{3cm}|p{3cm}|p{3cm}|}
			\hline
			\multicolumn{4}{ |c| }{$\varepsilon=1$} \\ 
			\hline
			& 
			\hspace{10mm}$ \|f\|_{L_{4}^\infty}$ &
			\hspace{10mm}$ \|f\|_{L_{5}^\infty}$ & 
			\hspace{10mm}$ \|\mathbb{E}\|_{L_x^\infty}$\\
			\hline
			($N_x$,$2N_x$) & Error \, \, \, \, \, Order& Error \, \, \, \, \, Order &  Error \, \, \, \, \, Order\\
			\hline
			$(40,80)$    & 8.4656e-04
			 \, \, \, 0.862
			  & 1.3385e-03
			   \, \, \, 0.862  
			& 3.9372e-05
			 \, \, \, 1.404
			\\
			$(80,160)$    & 4.6567e-04
			 \, \, \, 0.933
			 &  7.3629e-04
			   \, \, \, 0.921		 
			& 1.4876e-05
			 \, \, \, 1.096
			 \\
			(160,320) & 2.4390e-04
			  \, \, \, 0.967
			&  3.8880e-04
			   \, \, \, 0.959	  
			& 6.9595e-06
			 \, \, \, 1.024
			 \\
			(320,640) &  1.2477e-04
			 \, \, \, 0.983 &  
			 2.0006e-04  \, \, \, 0.981 	& 3.4220e-06
			 \, \, \, 1.006 \\
			(640,1280) &  6.3107e-05 \, \, \,  &  1.0136e-04  \, \, \,
			& 1.7035e-06 \, \, \, \\
			\hline
			\hline
		\end{tabular}
	\begin{tabular}{ |p{3cm}||p{3cm}|p{3cm}|p{3cm}|}
		\hline
		\multicolumn{4}{ |c| }{$\varepsilon=0.01$} \\ 
		\hline
		& 
		\hspace{10mm}$ \|f\|_{L_{4}^\infty}$ &
		\hspace{10mm}$ \|f\|_{L_{5}^\infty}$ & 
		\hspace{10mm}$ \|\mathbb{E}\|_{L_x^\infty}$\\
		\hline
		($N_x$,$2N_x$) & Error \, \, \, \, \, Order & Error \, \, \, \, \, Order &  Error \, \, \, \, \, Order\\
		\hline
		$(40,80)$    & 8.4759e-04
		\, \, \, 0.849
		& 1.3402e-03
		\, \, \, 0.765	  
		& 3.5723e-05
		\, \, \, 0.752
		\\
		$(80,160)$    & 4.7062e-04
		\, \, \, 0.889   
		& 7.8887e-04
		\, \, \, 0.917		 
		& 2.1214e-05
		\, \, \, 0.928
		\\
		(160,320) & 2.5422e-04
		\, \, \, 0.959
		& 4.1768e-04
		\, \, \, 0.950		   
		& 1.1148e-05
		\, \, \, 0.976
		\\
		(320,640) &  1.3082e-04
		\, \, \, 0.979 &  2.1626e-04
		\, \, \, 0.979
		& 5.6660e-06
		\, \, \, 0.991 \\
		(640,1280) &  6.6371e-05 \, \, \, &  1.0972e-04  \, \, \,
		& 2.8507e-06 \, \, \, \\
		\hline
		\hline
	\end{tabular}
	\begin{tabular}{ |p{3cm}||p{3cm}|p{3cm}|p{3cm}|}
		\hline
		\multicolumn{4}{ |c| }{$\varepsilon=0.0001$} \\ 
		\hline
		& 
		\hspace{10mm}$ \|f\|_{L_{4}^\infty}$ &
		\hspace{10mm}$ \|f\|_{L_{5}^\infty}$ & 
		\hspace{10mm}$ \|\mathbb{E}\|_{L_x^\infty}$\\
		\hline
		($N_x$,$2N_x$) & Error \, \, \, \, \, Order & Error \, \, \, \, \, Order &  Error \, \, \, \, \, Order\\
		\hline
		$(40,80)$    & 1.0338e-03
		\, \, \,  0.844
		& 1.6346e-03
		\, \, \, 0.808 
		& 4.9554e-05
		\, \, \, 0.785
		\\
		$(80,160)$    & 5.7586e-04
		\, \, \, 0.913
		& 9.3386e-04
		\, \, \, 0.897
		& 2.8768e-05
		\, \, \, 0.940
		\\
		(160,320) & 3.0584e-04
		\, \, \, 0.953
		& 5.0138e-04
		\, \, \, 0.958  
		& 1.4998e-05
		\, \, \, 0.980
		\\
		(320,640) &  1.5797e-04
		\, \, \, 0.979 &  2.5803e-04
		\, \, \, 0.977
		& 7.6016e-06
		\, \, \, 0.993 \\
		(640,1280) &  8.0128e-05 \, \, \,  &  1.3113e-04  \, \, \,
		& 3.8198e-06 \, \, \, \\
		\hline
		\hline
	\end{tabular}
	\caption{Numerical errors and order of convergence for various Knudsen numbers $\varepsilon=1,\, 0.01,\, 0.0001$ with $q=4,\,5$.
}\label{tab1}
\end{table}
The numerical results show that the proposed method attains expected first-order accuracy for all ranges of Knudsen number regardless of the values of $q$.\\


%
%






\noindent{\bf Acknowledgment}\newline
The work of Cho was supported by the National Research Foundation of Korea(NRF) grant funded by the Korea government(MSIT) (RS-2026-25493252).\newline
\noindent The work of Yun was supported by the National Research Foundation of Korea(NRF) grant funded by
the Korean government(MSIT) (RS-2023-NR076676).\newline

\appendix
\section{A priori estimates for the continuous solution}\label{sec Ap A}
In this appendix, we present several estimates for the continuous solution required in the main text. First, we prove that the moments of $\tilde{f}$, which is defined $\tilde{f}(x,v,t)=f(x-v\Delta t,v-\mathbb{E}(x,t)\Delta t,t)$, are uniformly bounded.

\begin{lemma}\label{continuous tilde moment bound}
    Assume that $\tilde{f}$ is defined in Theorem \ref{consistent theorem}. Let $\tilde{\rho}(x,t)$, $\tilde{U}(x,t)$, and $\tilde{T}(x,t)$ be defined by
    \begin{align*}
        \begin{split}
            & \tilde{\rho}(x,t) = \int_{\mathbb{R}} \tilde{f}(x,v,t)dv, \cr
            &\tilde{\rho}(x,t)\tilde{U}(x,t) = \int_{\mathbb{R}} v\tilde{f}(x,v,t)dv, \cr
            &\tilde{\rho}(x,t)\tilde{T}(x,t) + \tilde{\rho}(x,t)|\tilde{U}(x,t)|^2 = \int_{\mathbb{R}} |v|^2\tilde{f}(x,v,t)dv.
        \end{split}
    \end{align*}
    Then, under the assumption of Theorem \ref{analysis}, there exist constants $C_{q,f_0,T^f,\varepsilon}$ which satisfies
     \begin{align*}
         \begin{split}
             \tilde{\rho}(x,t) \geq C_{q,f_0,T^f,\varepsilon} > 0,
         \end{split}
     \end{align*}
    and
    \begin{align*}
        \begin{split}
            \|\tilde{\rho}(t)\|_{L^{\infty}_{x}}, \;\|\tilde{U}(t)\|_{L^{\infty}_{x}},\;\|\tilde{T}(t)\|_{L^{\infty}_{x}} \leq C_{q,f_0,T^f,\varepsilon}.
        \end{split}
    \end{align*}
\end{lemma}
\begin{proof}
    First, we prove the lower bound result. To do this, we consider the ODE system, which is given by
    \begin{align*}
        \begin{split}
                 &\frac{d}{ds}(X(s),V(s)) = (V(s),\mathbb{E}(X(s),s)),\quad(X(t),V(t))=(x-v\Delta t,v-\mathbb{E}(x,t)\Delta t).
        \end{split}
    \end{align*}
    By Theorem \ref{analysis}, $V(0)$ is derived by
    \begin{align}\label{V_0}
        \begin{split}
            |V(0)| \leq |v-\mathbb{E}\Delta t| + \int^t_0 \|\mathbb{E}(s)\|_{L^{\infty}_{x}} ds \leq |v| + C_{q,f_0,T^f}\Delta t + C_{q,f_0,T^f}\leq |v|+C_{q,f_0,T^f}.
        \end{split}
    \end{align}
    On the other hand,  \eqref{VPBGK model} gives
    \begin{align*}
        \begin{split}
            \frac{d}{ds} f(X(s),V(s),s) = \frac{1}{\varepsilon}(\mathcal{M}(f) - f)(X(s),V(s),s)\geq-\frac{1}{\varepsilon}f(X(s),V(s),s).
        \end{split}
    \end{align*}
    This implies
    \begin{align}\label{lower bdd}
        \begin{split}
            \tilde{f}(x,v,t) \geq e^{-\frac{t}{\varepsilon}}f_0(X(0),V(0)) \geq C_{0}e^{-\frac{T^f}{\varepsilon}}e^{-|V(0)|^\alpha}.
        \end{split}
    \end{align}
    Inserting (\ref{V_0}) into (\ref{lower bdd}), we get
    \begin{align*}
        \begin{split}
            \tilde{f}(x,v,t)  \geq C_{q,f_0,T^f,\varepsilon} e^{-|v|^\alpha}.
        \end{split}
    \end{align*}
    Therefore, we derive
    \begin{align}\label{continuous tilde density lower bound}
        \begin{split}
            \tilde{\rho}(x,t) \geq \int_{\mathbb{R}}C_{q,f_0,T^f,\varepsilon} e^{-|v|^\alpha} dv = C_{q,f_0,T^f,\varepsilon}>0.
        \end{split}
    \end{align}
    To obtain the upper bound, we have
    \begin{align*}
        \begin{split}
            |\tilde{\rho}(x,t)| &\leq \int_{\mathbb{R}} |\tilde{f}(x,v,t)|dv \leq \|\tilde{f}(t)\|_{L^{\infty}_{q}} \int_{\mathbb{R}} \frac{1}{(1+|v|)^q}dv\leq C_q \|\tilde{f}(t)\|_{L^{\infty}_{q}}.
        \end{split}
    \end{align*}
    By Theorem \ref{analysis}, we get
    \begin{align}\label{continuous tilde f upper bound}
        \begin{split}
            (1+|v|)^q |\tilde{f}(x,v,t)| &= (1+|v|)^q |f(x-v\Delta t,v-\mathbb{E}(x,t)\Delta t,t)|\cr
            &= \|f(t)\|_{L^{\infty}_{q}}\left(\frac{1+|v|}{1+|v-\mathbb{E}(x,t)\Delta t|}\right)^q\cr
            &\leq \|f(t)\|_{L^{\infty}_{q}} (1+ C_{q}\|\mathbb{E}(t)\|_{L^{\infty}_x}\Delta t)^q\cr
            &\leq C_{q,f_0,T^f}.
        \end{split}
    \end{align}
    Thus, we yield
    \begin{align*}
        \begin{split}
        \|\tilde{\rho}(t)\|_{L^{\infty}_{q}} \leq C_{q,f_0,T^f}.
        \end{split}
    \end{align*}
    By \eqref{continuous tilde density lower bound} and \eqref{continuous tilde f upper bound}, $\tilde{U}$ is estimated by
    \begin{align*}
        \begin{split}
            |\tilde{U}(x,t)| &\leq \frac{1}{\tilde{\rho}(x,t)}\int_{\mathbb{R}} |\tilde{f}(x,v,t)||v|dv\leq C_{q,f_0,\varepsilon}\|\tilde{f}(t)\|_{L^{\infty}_{q}} \int_{\mathbb{R}} \frac{|v|}{(1+|v|)^q}dv\leq C_{q,f_0,T^f,\varepsilon}.
        \end{split}
    \end{align*}
    Lastly, we use \eqref{continuous tilde density lower bound} and \eqref{continuous tilde f upper bound} again to get
    \begin{align*}
        \begin{split}
            \tilde{T}(x,t) &= \frac{1}{\tilde{\rho}(x,t)}\left\{\int_{\mathbb{R}}|v|^2 \tilde{f}(x,v,t)dv - |\tilde{U}|^2 \right\}\cr
            &\leq \frac{1}{\tilde{\rho}(x,t)}\int_{\mathbb{R}}|v|^2 \tilde{f}(x,v,t)dv\cr
            &\leq C_{q,f_0,T^f,\varepsilon}\|\tilde{f}(t)\|_{L^{\infty}_{q}} \int_{\mathbb{R}}\frac{|v|^2}{(1+|v|)^q}dv\cr
            &\leq C_{q,f_0,T^f,\varepsilon}.
        \end{split}
    \end{align*}
\end{proof}

Next, we will control the time derivatives for the solution such as $\|\partial_t f\|_{L^{\infty}_{q}}$, $\|\partial_t \partial_x f\|_{L^{\infty}_{q}}$, $\|\partial_t \partial_v f\|_{L^{\infty}_{q}}$, and $\|\partial_t \mathcal{M}(f)\|_{L^{\infty}_{q}}$. 
 \begin{lemma}\label{TimeDeri_f}
     Let $f$ and $\mathbb{E}$ is the solution of (\ref{VPBGK model}). Then, under the assumption of Theorem \ref{Main Thm}, we have
     \begin{align*}
         \begin{split}
             \|\partial_t f \|_{L^{\infty}_{q}} \leq C_{q,f_0,T^f,\varepsilon}.
         \end{split}
     \end{align*}
 \end{lemma}
 \begin{proof}
 Multiplying $(1+|v|)^q$ to \eqref{VPBGK model}, Lemma \ref{LipMaxw} gives
 \begin{align*}
     \begin{split}
         (1+|v|^{q})|\partial_t f| &\leq (1+|v|^{q})\left\{|v||\partial_x f| + |\mathbb{E}||\partial_v f|+ \frac{1}{\varepsilon}|\mathcal{M}(f)| + \frac{1}{\varepsilon}|f|\right\}\cr
         &\leq C_{q,f_0}\left\{\|f\|_{W^{1,\infty}_{q+1}} + \|f\|_{W^{1,\infty}_{q}} + \frac{1}{\varepsilon}\|f\|_{L^{\infty}_{q}} + \frac{1}{\varepsilon}\|\mathcal{M}(f)\|_{L^{\infty}_{q}}\right\}\cr
         &\leq C_{q,f_0,T^f,\varepsilon}.
     \end{split}
 \end{align*}
 This completes the proof.
 \end{proof}
 
\begin{lemma}\label{TimeDeri_f2}
    Let $f$ and $\mathbb{E}$ is the solution of (\ref{VPBGK model}). Then, under the assumption of Theorem \ref{Main Thm}, we have
    \begin{align*}
        \|\partial_t \partial_x f\|_{L^{\infty}_{q}} + \|\partial_t \partial_v f\|_{L^{\infty}_{q}}\leq C_{q,f_0,T^f,\varepsilon}.
    \end{align*}
\end{lemma} 
 \begin{proof}
     We differentiate $\eqref{VPBGK model}$ in $x$, and multiply $(1+|v|)^q$ to get
     \begin{align*}
         \begin{split}
             &(1+|v|^{q})|\partial_x\partial_t f|\cr
             &\qquad\leq (1+|v|^{q})\left\{|v||\partial^2_x f(x,v,t)| + |\partial_x \mathbb{E}||\partial_v f| +|\mathbb{E}||\partial_x\partial_v f|+\frac{1}{\varepsilon}|\partial_x \mathcal{M}(f)| + \frac{1}{\varepsilon}|\partial_x f(x,v,t)|\right\}\cr
             &\qquad\leq (1+|v|^{q})\left\{|v||\partial^2_x f(x,v,t)| + (|\rho|+1)|\partial_v f| +|\mathbb{E}||\partial_x\partial_v f|+\frac{1}{\varepsilon}|\partial_x \mathcal{M}(f)| + \frac{1}{\varepsilon}|\partial_x f(x,v,t)|\right\}.
         \end{split}
     \end{align*}
     Since $\partial_t\partial_x f=\partial_x\partial_t f$ by condition $(A3)$ in Theorem \ref{analysis}, we utilize Lemma \ref{LipMaxw} to obtain
     \begin{align*}
         \begin{split}
         \|\partial_t \partial_x f(t)\|_{L^{\infty}_{q}} \leq C_{q,f_0,T^f,\varepsilon}.
         \end{split}
     \end{align*}
     On the other hand, differentiating $\eqref{VPBGK model}$ in $v$, and multiplying $(1+|v|)^q$, we have
    \begin{align*}
         \begin{split}
             &(1+|v|^{q})|\partial_v\partial_t f(x,v,t)|\cr 
             &\qquad\leq (1+|v|^{q})\left\{|\partial_x f| + |v||\partial_v \partial_x f| +|\mathbb{E}||\partial^2_v f|+\frac{1}{\varepsilon}|\partial_v \mathcal{M}(f)| + \frac{1}{\varepsilon}|\partial_v f|\right\}.
         \end{split}
     \end{align*}
     Similar to the above case, this leads to 
     \begin{align*}
         \begin{split}
             \|\partial_t \partial_v f(t)\|_{L^{\infty}_{q}} \leq C_{q,f_0,T^f,\varepsilon}.
         \end{split}
     \end{align*}
 \end{proof}

 \begin{lemma}\label{TimeDeri_Maxwellian}
     Let $f$ and $\mathbb{E}$ is the solution of (\ref{VPBGK model}) and $\mathcal{M}(f)$ is defined by (\ref{local Maxwellian}). \ Then, under the assumption of Theorem \ref{analysis}, we have
    \begin{align*}
        \|\partial_t \mathcal{M}(f)(t)\|_{L^{\infty}_{q}}\leq C_{q,f_0,T^f}.
    \end{align*}
 \end{lemma}
 \begin{proof}
 By the chain rule,
    \begin{align}\label{t deri Mx}
        \begin{split}
            \partial_t \mathcal{M}(f) = \partial_t \rho \frac{\partial\mathcal{M}(f)}{\partial U}+ \partial_t U \frac{\partial\mathcal{M}(f)}{\partial \rho}+\partial_t T \frac{\partial\mathcal{M}(f)}{\partial T}.
        \end{split}
    \end{align}
    By direct calculation, we get
\begin{align*}
    \begin{split}
        &\left|\frac{\partial\mathcal{M}(f)}{\partial \rho}\right| \leq \frac{1}{(2\pi T)^{\frac{1}{2}}}\exp\left(-\frac{|v-U|^{2}}{2T} \right)\leq C_{q,T^{f},f_{0}}e^{-C_{q,f_0,T^f}|v|^{2}},\cr
        &\left|\frac{\partial\mathcal{M}(f)}{\partial U}\right| \leq \frac{|v-U|}{T}\mathcal{M}(f)\leq C_{q,T^{f},f_{0}}(1+|v|)e^{-C_{q,f_0,T^f}|v|^{2}},\cr
        &\left|\frac{\partial\mathcal{M}(f)}{\partial T}\right| \leq C\left( \frac{|v-U|^2}{T^2}+\frac{1}{T}\right)\mathcal{M}(f) \leq C_{q,T^{f},f_{0}}(1+|v|^{2})e^{-C_{q,T^{f},f_{0}}|v|^{2}}.
    \end{split}
\end{align*}
Since $x^{n}e^{-Cx^{2}} \leq C$ for some constant $C>0$, we obtain
\begin{align}\label{A.3 result 1}
    \begin{split}
        \|\partial_\rho \mathcal{M}(f)\|_{L^{\infty}_{q}}+\|\partial_U \mathcal{M}(f)\|_{L^{\infty}_{q}}+\|\partial_T \mathcal{M}(f)\|_{L^{\infty}_{q}} \leq C_{q,f_0,T^f}.
    \end{split}
\end{align}
On the other hand, we multiply $\varphi(v)=(1,v,|v|^2)$ to \eqref{VPBGK model}, and integrate with respect to $v$ to get
\begin{align*}
    \begin{split}
        \left|\partial_t \int_{\mathbb{R}}\varphi(v) f
        dv\right| &\leq \left|\int_{\mathbb{R}} v\varphi(v) \partial_x f dv\right| + \left|\int_{\mathbb{R}} \mathbb{E}\varphi(v) \partial_v f  dv\right|\cr
        &\leq C(1+\|\mathbb{E}\|_{L^{\infty}_x})\int_{\mathbb{R}} (1+|v|)^3|\partial_{x,v}f|dv \cr
        &\leq C(1+\|\mathbb{E}\|_{L^{\infty}_x})\|f(t)\|_{W^{1,\infty}_{q+1}} \int_{\mathbb{R}} \frac{dv}{(1+|v|)^{q-2}}\cr
        &\leq C_{q,f_0,T^f}.
    \end{split}
\end{align*}
Therefore, we have
\begin{align}\label{A.3 result 2}
    \begin{split}
        |\partial_t \rho|+ |\partial_t U|+|\partial_t T| \leq C_{q,f_0,T^f}. 
    \end{split}
\end{align}
According to \eqref{t deri Mx}, \eqref{A.3 result 1}, and \eqref{A.3 result 2}, we get
\begin{align*}
    \begin{split}
        (1+|v|)^q |\partial_t \mathcal{M}(f)| &\leq |\partial_t \rho|\|\partial_\rho \mathcal{M}(f)\|_{L^{\infty}_q} + |\partial_t U|\|\partial_U \mathcal{M}(f)\|_{L^{\infty}_q} + |\partial_t T|\|\partial_T \mathcal{M}(f)\|_{L^{\infty}_q}\leq C_{q,f_0,T^f}.
    \end{split}
\end{align*}
This completes the proof.
\end{proof}

\begin{lemma}\label{TimeDeri_field}
         Let $f$ and $\mathbb{E}$ is the solution of (\ref{VPBGK model}) and $\mathcal{M}(f)$ is defined by (\ref{local Maxwellian}). \ Then, under the assumption of Theorem \ref{analysis}, we have
    \begin{align*}
        \|\partial_t \mathbb{E}(t)\|_{L^{\infty}_{x}}\leq C_{q,f_0,T^f}.
    \end{align*}  
\end{lemma}
\begin{proof}
    From charge conservation, we have
    \begin{align*}
        \begin{split}
            \partial_t \rho +\int_{\mathbb{R}}v\partial_x fdv=0.
        \end{split}
    \end{align*}
    Differentiating the Poisson equation in $t$, we get
    \begin{align*}
        \begin{split}
            \partial_t(\partial_x \mathbb{E})=\partial_t(\rho -1)=-\int_{\mathbb{R}}v\partial_x fdv.
        \end{split}
    \end{align*}
    Then, \eqref{continuous electric field} gives
    \begin{align*}
        \begin{split}
            \partial_t \mathbb{E}(x,t)=\int^1_0K(x,y)\left( -\int_{\mathbb{R}}v\partial_x f(y,v,t)dv\right)dy,
        \end{split}
    \end{align*}
    for $K$ is defined in \eqref{Green kernel}. Therefore, it is enough to show that $\|\partial_t \mathbb{E}(t)\|_{L^{\infty}_x}\leq C_{q,f_0,T^f}$ since we see that
    \begin{align*}
        \begin{split}
            \left|\int_{\mathbb{R}}v\partial_x fdv\right|\leq \|f\|_{W^{1,\infty}_{q+1}}\int_{\mathbb{R}}\frac{dv}{(1+|v|)^{q-1}}\leq C_{q,f_0,T^f}.
        \end{split}
    \end{align*}
\end{proof}

\bibliographystyle{abbrv}
\bibliography{ref}

@article{PhysRev.94.511,
  title = {A Model for Collision Processes in Gases. {I}. Small Amplitude Processes in Charged and Neutral One-Component Systems},
  author = {Bhatnagar, P. L. and Gross, E. P. and Krook, M.},
  journal = {Phys. Rev.},
  volume = {94},
  issue = {3},
  pages = {511--525},
  numpages = {0},
  year = {1954},
  month = {May},
  publisher = {American Physical Society},
  doi = {10.1103/PhysRev.94.511},
  url = {https://link.aps.org/doi/10.1103/PhysRev.94.511}
}

@article{pieraccini2007implicit,
  title={Implicit--explicit schemes for {BGK} kinetic equations},
  author={Pieraccini, Sandra and Puppo, Gabriella},
  journal={Journal of Scientific Computing},
  volume={32},
  pages={1--28},
  year={2007},
  publisher={Springer}
}

@article{filbet2001convergence,
  title={Convergence of a finite volume scheme for the {V}lasov--{P}oisson system},
  author={Filbet, Francis},
  journal={SIAM Journal on Numerical Analysis},
  volume={39},
  number={4},
  pages={1146--1169},
  year={2001},
  publisher={SIAM}
}

@article{dimarco2014asymptotic,
  title={An asymptotic preserving automatic domain decomposition method for the {V}lasov--{P}oisson--{BGK} system with applications to plasmas},
  author={Dimarco, Giacomo and Mieussens, Luc and Rispoli, Vittorio},
  journal={Journal of Computational Physics},
  volume={274},
  pages={122--139},
  year={2014},
  publisher={Elsevier}
}

@article{crestetto2012kinetic,
  title={Kinetic/fluid micro-macro numerical schemes for {V}lasov-{P}oisson-{BGK} equation using particles},
  author={Crestetto, Ana{\"\i}s and Crouseilles, Nicolas and Lemou, Mohammed},
  journal={Kinetic and Related Models},
  volume={5},
  number={4},
  pages={787--816},
  year={2012}
}

@article{crouseilles2016multiscale,
  title={Multiscale Schemes for the {BGK}--{V}lasov--{P}oisson System in the Quasi-Neutral and Fluid Limits. {S}tability Analysis and First Order Schemes},
  author={Crouseilles, Nicolas and Dimarco, Giacomo and Vignal, Marie-H{\'e}l{\'e}ne},
  journal={Multiscale Modeling \& Simulation},
  volume={14},
  number={1},
  pages={65--95},
  year={2016},
  publisher={SIAM}
}

@article{laidin2022hybrid,
  title={Hybrid Kinetic/Fluid numerical method for the {V}lasov-{BGK} equation in the diffusive scaling},
  author={Laidin, Tino},
  journal={Kinetic and Related Models},
  volume={16},
  number={6},
  pages={913--947},
  year={2023},
}

@article{russo2018convergence,
  title={Convergence of a semi-{L}agrangian scheme for the ellipsoidal {BGK} model of the {B}oltzmann equation},
  author={Russo, Giovanni and Yun, Seok-Bae},
  journal={SIAM Journal on Numerical Analysis},
  volume={56},
  number={6},
  pages={3580--3610},
  year={2018},
  publisher={Kinetic and Related Models}
}

@article{crouseilles2011asymptotic,
  title={An asymptotic preserving scheme based on a micro-macro decomposition for collisional {V}lasov equations: diffusion and high-field scaling limits.},
  author={Crouseilles, Nicolas and Lemou, Mohammed},
  journal={Kinetic and Related Models},
  volume={4},
  number={2},
  pages={441--477},
  year={2011}
}

@article{hu2018asymptotic,
  title={Asymptotic-preserving and positivity-preserving implicit-explicit schemes for the stiff {BGK} equation},
  author={Hu, Jingwei and Shu, Ruiwen and Zhang, Xiangxiong},
  journal={SIAM Journal on Numerical Analysis},
  volume={56},
  number={2},
  pages={942--973},
  year={2018},
  publisher={SIAM}
}

@article{boscarino2022convergence,
  title={Convergence estimates of a semi-{L}agrangian scheme for the ellipsoidal {BGK} model for polyatomic molecules},
  author={Boscarino, Sebastiano and Cho, Seung Yeon and Russo, Giovanni and Yun, Seok-Bae},
  journal={ESAIM: Mathematical Modelling and Numerical Analysis},
  volume={56},
  number={3},
  pages={893--942},
  year={2022},
  publisher={EDP Sciences}
}

@article{hu2017class,
  title={On a class of implicit--explicit runge--kutta schemes for stiff kinetic equations preserving the navier--stokes limit},
  author={Hu, Jingwei and Zhang, Xiangxiong},
  journal={Journal of Scientific Computing},
  volume={73},
  pages={797--818},
  year={2017},
  publisher={Springer}
}

@article{issautier1996convergence,
  title={Convergence of a weighted particle method for solving the {B}oltzmann {BGK} equation},
  author={Issautier, D},
  journal={SIAM Journal on Numerical Analysis},
  volume={33},
  number={6},
  pages={2099--2119},
  year={1996},
  publisher={SIAM}
}

@article{russo2012convergence,
  title={Convergence of a semi-{L}agrangian scheme for the {BGK} model of the Boltzmann equation},
  author={Russo, Giovanni and Santagati, Pietro and Yun, Seok-Bae},
  journal={SIAM Journal on Numerical Analysis},
  volume={50},
  number={3},
  pages={1111--1135},
  year={2012},
  publisher={SIAM}
}

@article{einkemmer2014convergence,
  title={Convergence analysis of a discontinuous {G}alerkin/{S}trang splitting approximation for the {V}lasov--{P}oisson equations},
  author={Einkemmer, Lukas and Ostermann, Alexander},
  journal={SIAM Journal on Numerical Analysis},
  volume={52},
  number={2},
  pages={757--778},
  year={2014},
  publisher={SIAM}
}

@article{besse2008convergence2,
  title={Convergence of classes of high-order semi-{L}agrangian schemes for the {V}lasov--{P}oisson system},
  author={Besse, Nicolas and Mehrenberger, Michel},
  journal={Mathematics of Computation},
  volume={77},
  number={261},
  pages={93--123},
  year={2008}
}

@article{besse2008convergence1,
  title={Convergence of a high-order semi-{L}agrangian scheme with propagation of gradients for the one-dimensional {V}lasov--{P}oisson system},
  author={Besse, Nicolas},
  journal={SIAM Journal on Numerical Analysis},
  volume={46},
  number={2},
  pages={639--670},
  year={2008},
  publisher={SIAM}
}

@article{besse2004convergence,
  title={Convergence of a semi-{L}agrangian scheme for the one-dimensional {V}lasov--{P}oisson system},
  author={Besse, Nicolas},
  journal={SIAM Journal on Numerical Analysis},
  volume={42},
  number={1},
  pages={350--382},
  year={2004},
  publisher={SIAM}
}

@article{cottet1984particle,
  title={Particle methods for the one-dimensional {V}lasov--{P}oisson equations},
  author={Cottet, G-H and Raviart, P-A},
  journal={SIAM Journal on Numerical Analysis},
  volume={21},
  number={1},
  pages={52--76},
  year={1984},
  publisher={SIAM}
}

@article{ganguly1989convergence,
  title={On the convergence of particle methods for multidimensional {V}lasov--{P}oisson systems},
  author={Ganguly, Keshab and Victory, Jr, HD},
  journal={SIAM Journal on Numerical Analysis},
  volume={26},
  number={2},
  pages={249--288},
  year={1989},
  publisher={SIAM}
}

@article{victory1991convergence2,
  title={The convergence analysis of fully discretized particle methods for solving {V}lasov--{P}oisson systems},
  author={Victory, Jr, HD and Tucker, Garry and Ganguly, Keshab},
  journal={SIAM Journal on Numerical Analysis},
  volume={28},
  number={4},
  pages={955--989},
  year={1991},
  publisher={SIAM}
}

@article{victory1991convergence1,
  title={The convergence theory of particle-in-cell methods for multidimensional {V}lasov--{P}oisson systems},
  author={Victory, Jr, HD and Allen, Edward J},
  journal={SIAM Journal on Numerical Analysis},
  volume={28},
  number={5},
  pages={1207--1241},
  year={1991},
  publisher={SIAM}
}

@incollection{Cho2024,
  author    = {Cho, S. Y. and Groppi, M. and Qiu, J. M. and Russo, G. and Yun, S.-B.},
  title     = {Conservative Semi-{L}agrangian Methods for Kinetic Equations},
  booktitle = {Active Particles, Volume 4},
  publisher = {Birkh{\"a}user, Cham},
  year      = {2024},
  pages     = {283-420},
  doi       = {10.1007/978-3-031-73423-6_7},
  isbn      = {978-3-031-73423-6}
}

@article{schaeffer1998convergence,
  title={Convergence of a Difference Scheme for the Vlasov--Poisson--Fokker--Planck System in One Dimension},
  author={Schaeffer, Jack},
  journal={SIAM Journal on Numerical Analysis},
  volume={35},
  number={3},
  pages={1149--1175},
  year={1998},
  publisher={SIAM}
}

@unpublished{YunPark2026,
  author = {Park, Sungsu and Yun, Seok-Bae},
  title  = {Classical solution for the {V}lasov-{P}oisson-{BGK} model in the periodic box},
  note   = {In preparation},
  year   = {2026}
}

\end{document}